\documentclass[10pt]{amsart}
\usepackage{amsxtra, amsfonts, amsmath, amsthm, amstext, amssymb, amscd, mathrsfs, verbatim, color}
\usepackage{threeparttable}
\usepackage{mathtools}
\usepackage[ansinew]{inputenc}\usepackage[T1]{fontenc}
\usepackage[all,cmtip]{xy}
\usepackage{tikz}
\usetikzlibrary {positioning}

\theoremstyle{plain}

\newtheorem{theorem}{Theorem}[section]
\newtheorem{lemma}[theorem]{Lemma}
\newtheorem{definition-theorem}[theorem]{Definition-Theorem}
\newtheorem{proposition}[theorem]{Proposition}
\newtheorem{corollary}[theorem]{Corollary}

\theoremstyle{definition}
\newtheorem{definition}[theorem]{Definition}
\newtheorem{example}[theorem]{Example}
\newtheorem{remark}[theorem]{Remark}
\newtheorem{notation}[theorem]{Notation}
\newcommand \bth[1] { \begin{theorem}\label{t#1} }
\newcommand \ble[1] { \begin{lemma}\label{l#1} }

\newcommand \bpr[1] { \begin{proposition}\label{p#1} }
\newcommand \bco[1] { \begin{corollary}\label{c#1} }
\newcommand \bde[1] { \begin{definition}\label{d#1}\rm }
\newcommand \bex[1] { \begin{example}\label{e#1}\rm }
\newcommand \bre[1] { \begin{remark}\label{r#1}\rm }

\newcommand \bnota[1] {\begin{notation}\label{n#1}\rm }
\newcommand {\ele} { \end{lemma} }

\newcommand {\epr} { \end{proposition} }
\newcommand {\eco} { \end{corollary} }
\newcommand {\ede} { \end{definition} }
\newcommand {\eex} { \end{example} }
\newcommand {\ere} { \end{remark} }
\newcommand {\enota} { \end{notation} }

\begin{document}
\title[Derivatives]{Construction of simple quotients of Bernstein-Zelevinsky derivatives and  highest derivative multisegments I: Reduction to combinatorics} 

\author[Kei Yuen Chan]{Kei Yuen Chan}
\address{
Department of Mathematics, The University of Hong Kong
}

\email{kychan1@hku.hk}
\maketitle

\begin{abstract}
Let $F$ be a local non-Archimedean field. A sequence of derivatives of generalized Steinberg representations can be used to construct simple quotients of Bernstein-Zelevinsky derivatives of irreducible representations of $\mathrm{GL}_n(F)$. In the first of a series of articles, we introduce a notion of a highest derivative multisegment, which in turn gives a combinatorial approach to study problems about those simple quotients. We also prove a double derivative result along the way.

\end{abstract}
\setcounter{tocdepth}{1}

\section{Introduction}

Let $F$ be a local non-Archimedean field. The Bernstein-Zelevinsky derivative is a twisted Jacquet functor, originally introduced for classifying the irreducible representations of $\mathrm{GL}_{n}(F)$ in \cite{BZ76, BZ77, Ze80}. This is the first one of a series of articles in studying the theory of Bernstein-Zelevinsky derivatives. The main result of this article provides a combinatorial approach to the problems of socles and cosocles of Bernstein-Zelevinsky derivatives of irreducible representations. In the sequels \cite{Ch22+d, Ch22+e}, we shall obtain a canonical sequence from some minimality by using results of this one, and then establish properties for such sequence. Applications to branching laws will be considered in \cite{Ch22+b}. Indeed, we shall show in \cite{Ch22+b} that any simple quotients of Bernstein-Zelevinsky derivatives can be constructed from such sequences.

On the other hand, there is a notion of $\rho$-derivatives introduced and studied by C. Jantzen \cite{Ja07} and independently by M\'inguez \cite{Mi09}, which will be important in our study. To be more precise, the notion of derivatives in this article is the one using essentially square-integrable representations to replace cuspidal representations $\rho$ in \cite{Ja07, Mi09}, which we shall simply call St-derivatives. Such derivative is also used in recent work of Atobe-M\'inguez \cite{AM20}, and Lapid-M\'inguez \cite{LM22} for other studies. A certain sequence of St-derivatives can be used to construct some simple quotients of Bernstein-Zelevinsky (BZ) derivatives (see Section \ref{ss two derivatives}). This is based on the observation that any standard module in GL case has unique submodule and such submodule is generic \cite{JS83} (also see \cite{CaSh98, Ch21}). 

We recall some classical known results on BZ derivatives. The highest derivative of an arbitrary irreducible representation is determined by Zelevinsky \cite{Ze80}. A complete description for all the derivatives has been previously established for Steinberg representations and their Zelevinsky duals by Zelevinsky \cite{Ze80}, and for ladder representations (including Speh representations) by Lapid-M\'inguez \cite{LM14} (also see \cite{Ta87, CS19}) using a determinantal formula of Tadi\'c. The Bernstein-Zelevinsky derivatives of an irreducible representation are in general not semisimple, and so it could be more flexible to study their simple quotients rather than the entire structure. For example, an asymmetry property of simple quotients between left and right BZ derivatives is shown in \cite{Ch21} and has applications.

\subsection{Two notions of derivatives} \label{ss notion derivative}

We now introduce the two notions of derivatives and more notations will be given in Sections \ref{ss prelim} and \ref{ss two derivatives}. Let $G_n=\mathrm{GL}_n(F)$, the general linear group over $F$. For $a,b \in \mathbb Z$ with $b-a\geq 0$ and a cuspidal representation $\rho$ of $G_m$, we shall call $[a,b]_{\rho}$ a {\it segment} (see Section \ref{ss notations}) and define $l_{abs}([a,b]_{\rho})=(b-a+1)m$, called the {\it absolute length} of the segment $[a,b]_{\rho}$. Zelevinsky \cite{Ze80} showed that essentially square-integrable representations of $G_n$ can be parametrized by those segments. For each segment $\Delta$, we shall denote by $\mathrm{St}(\Delta)$ the corresponding essentially square-integrable representation (see Section \ref{ss zel langlands class}). 

Let $\mathrm{Irr}(G_n)$ be the set of (isomorphism classes of) irreducible smooth complex representations of $G_n$. Let $\mathrm{Irr}=\sqcup_n \mathrm{Irr}(G_n)$. Let $N_i \subset G_n$ (depending on $n$) be the unipotent radical containing matrices of the form $\begin{pmatrix} I_{n-i} & u \\ & I_i \end
{pmatrix}$, where $u$ is a $(n-i)\times i$ matrix. There exists at most one irreducible module $\tau \in \mathrm{Irr}(G_{n-i})$ such that 
 \[\tau \boxtimes \mathrm{St}(\Delta) \hookrightarrow \pi_{N_i} ,
\]
where $\pi_{N_i}$ is defined as the (normalized) Jacquet module of $\pi$ (see Remark \ref{rmk unique derivative} and \cite{Ch22+}). If such $\tau$ exists, we denote such $\tau$ by $D_{\Delta}(\pi)$. Otherwise, we set $D_{\Delta}(\pi)=0$. We shall refer $D_{\Delta}$ to be a {\it St-derivative}.

Let 
\[ R_i=\left\{ \begin{pmatrix} I_{n-i} & m \\ & u \end{pmatrix} :  m \in Mat_{(n-i)\times i}, u \in U_i \right\}, 
\]
where $U_i$ is the group of unipotent upper triangular matrices in $G_i$, and $Mat_{(n-i)\times i}$ is the set of $(n-i)\times i$ matrices over $F$. The right $i$-th {\it Bernstein-Zelevinsky derivative} $\pi^{(i)}$ of $\pi$ is defined as
\begin{eqnarray} \label{eqn def derivative}
 \delta_{R_{i}}^{-1/2} \cdot \frac{\pi}{\langle x.v-\psi(x)v : x \in R_i, v \in \pi \rangle } ,
\end{eqnarray}
where $\delta_{R_{i}}$ is the modular character of $R_{n-i}$, and $\psi$ is a non-degenerate character on $U_i$ extended trivially to $R_i$. Regarding $G_{n-i}$ as a subgroup $G_n$ via $g \mapsto \begin{pmatrix} g & \\ & I_i \end{pmatrix}$, we obtain a natural $G_{n-i}$-module structure on $\pi^{(i)}$. 

The {\it level} of $\pi$ is the largest integer $i^*$ such that $\pi^{(i^*)}\neq 0$. For the level $i^*$ of $\pi$, let $\pi^-=\pi^{(i^*)}$, which is known to be irreducible \cite[Theorem 8.1]{Ze80}. We shall call $\pi^-$ the {\it highest derivative} of $\pi$.

\subsection{Motivations from branching laws}

Our goal is applications towards branching laws for general linear groups or even other classical groups, in view of the recent derivative approach in studying branching laws e.g. \cite{MW12, Ve13, SV17, Pr18, CS21, Gu22, Ch21, GGP20, Ch20, Ch21+}. Those applications will appear elsewhere, see e.g. \cite{Ch22+b}, \cite{Ch22+c}. 

Let $\nu: G_n \rightarrow \mathbb C^{\times}$ be the character $\nu(g)=|\mathrm{det}(g)|_F$, where $|.|_F$ is the norm for $F$. The close relation of derivatives and branching laws is based on the Bernstein-Zelevinsky theory (see e.g. \cite{Ch21}):

\begin{lemma} \label{lem derivative gives branching law}
Let $\pi \in \mathrm{Irr}(G_{n+1})$. Let $\tau$ be a simple quotient of $\nu^{1/2}\cdot\pi^{(i)}$. Then, for some cuspidal representation $\sigma \in \mathrm{Irr}(G_{n-i})$, 
\[   \mathrm{Hom}_{G_n}( \pi, \tau \times \sigma) \neq 0 .
\]
\end{lemma}

One interesting consequence of the above lemma is that the multiplicity at-most-one phenomenon \cite{AGRS10, Ch21+} implies the multiplicity-freeness on socles and cosocles of the Bernstein-Zelevinsky derivatives of an irreducible representation (see Sections \ref{ss prop derivatives} and \ref{ss derivatives jacquet functor}). 

Instead of asking for simple quotients, one may also ask for simple submodules of the Bernstein-Zelevinsky derivatives of an irreducible representation. Such two problems are indeed equivalent by a dual structure of the Bernstein-Zelevinsky derivative (see Lemma \ref{lem dual property}, \cite[Lemma 2.4]{CS21}).

\subsection{Main results} \label{ss main results}

Fix an irreducible cuspidal representation $\rho$. Let $\mathrm{Irr}_{\rho}(G_n)$ be the subset of $\mathrm{Irr}(G_n)$ precisely consisting all irreducible representations of $G_n$ which are an irreducible quotient of some $\nu^{a_1}\rho \times \ldots \times \nu^{a_k}\rho$, for some integers $a_1, \ldots, a_k$. Let $\mathrm{Irr}_{\rho}=\sqcup_{n\geq 0} \mathrm{Irr}_{\rho}(G_n)$. The representations in $\mathrm{Irr}_{\rho}$ are the most interesting ones for our study, and the general case can be deduced from this (see Section \ref{ss reduction to cuspidal case}). 

Let $\mathrm{Seg}$ be the set of all segments including the empty set. Let $\mathrm{Seg}_{\rho}$ be the subset of $\mathrm{Seg}$ exactly consisting of all segments of the form $[a,b]_{\rho}$ for some $a,b\in \mathbb Z$, and the empty set $\phi$. 

A {\it multisegment} is a multiset of non-empty segments  (see (\ref{eqn seg def})). Let $\mathrm{Mult}$ be the set of all multisegments and let $\mathrm{Mult}_{\rho}$ be the subset of $\mathrm{Mult}$ of all multisegments whose segments are in $\mathrm{Seg}_{\rho}$. Note that the empty set $\emptyset$ is also an element in $\mathrm{Mult}$ and $\mathrm{Mult}_{\rho}$ in our convention.

For $\mathfrak m_1, \mathfrak m_2 \in \mathrm{Mult}$, we write $\mathfrak m_2 \leq_Z \mathfrak m_1$, called the {\it Zelevinsky ordering}, if $\mathfrak m_2$ can be obtained by a sequence of elementary intersection-union operations from $\mathfrak m_1$ (see Section \ref{ss intersect union oper}) or $\mathfrak m_1=\mathfrak m_2$. We shall equip $\mathrm{Mult}$ with the poset structure given by $\leq_Z$.

A sequence of segments $[a_1,b_1]_{\rho}, \ldots, [a_k,b_k]_{\rho}$ (all $a_j, b_j\in \mathbb Z$) is said to be in an {\it ascending order} if for any $i\leq j$, either $[a_i,b_i]_{\rho}$ and $[a_j,b_j]_{\rho}$ are unlinked (see Section \ref{ss notations} for the meaning of unlinked); or $a_i<a_j$. For a multisegment $\mathfrak n \in \mathrm{Mult}_{\rho}$, we write the segments in $\mathfrak n$ in an ascending order $\Delta_1, \ldots, \Delta_k$. Define 
\begin{align} \label{eqn derivatives of multisegment}
  D_{\mathfrak n}(\pi):=D_{\Delta_k}\circ \ldots \circ D_{\Delta_1}(\pi) .
\end{align}
We will show in Lemma \ref{lem comm derivative 1} that the derivative $D_{\mathfrak n}$ is independent of the choice of an ascending sequence for $\mathfrak n$. In particular, one may choose an ordering satisfying $a_1\leq \ldots \leq a_k$.

 In general, we have the following connection between the two notions of derivatives:

\begin{proposition} \label{prop simple quotient from derivative} (Consequence from Proposition \ref{prop unique as segments}(3))
Let $\pi \in \mathrm{Irr}_{\rho}$. Let $\mathfrak n \in \mathrm{Mult}_{\rho}$ such that
\[  D_{\mathfrak n}(\pi) \neq 0 .
\]
Then $D_{\mathfrak n}(\pi)$ is a simple quotient of $\pi^{(i)}$, where $i=l_{abs}(\Delta_1)+\ldots +l_{abs}(\Delta_k)$.
\end{proposition}
We remark that when $\pi$ is generic, those simple quotients have been described in \cite[Corollary 2.6]{CS21} by using a suitable filtration on the derivatives from the geometric lemma.

In general, two different sequences can give isomorphic simple quotients. Hence it is natural to study the combinatorial structure of the following set: for $\pi \in \mathrm{Irr}_{\rho}$ and for a simple quotient $\tau$ of $\pi^{(i)}$ (for some $i$), define
\[ \mathcal S(\pi, \tau) := \left\{ \mathfrak n \in \mathrm{Mult}_{\rho}: D_{\mathfrak n}(\pi) \cong \tau \right\} .
\]
The ordering  $\leq_Z$ induces a partial ordering on $\mathcal S(\pi, \tau)$. In \cite{Ch22+b}, we shall show that $\mathcal S(\pi, \tau)\neq \emptyset$, giving a converse of Proposition \ref{prop simple quotient from derivative}. The proof for  $\mathcal S(\pi, \tau)\neq \emptyset$ in \cite{Ch22+b} shall use some machinery from branching laws.

We need two ingredients in studying the poset $\mathcal S(\pi, \tau)$: highest derivative multisegments and removal processes.

We now explain the first ingredient. A multisegment $\mathfrak m$ is said to be {\it at the point $\nu^c\rho$} if any segment $\Delta$ in $\mathfrak m$ takes the form $[c,b]_{\rho}$ for some $b \geq c$. For $\pi \in \mathrm{Irr}_{\rho}$, define $\mathfrak{mxpt}^a(\pi,c)$ to be the maximal multisegment  at the point $\nu^c\rho$ such that $D_{\mathfrak{mxpt}^a(\pi,c)}(\pi)\neq 0$. (We refer to Section \ref{ss max at a pt} for the meaning of maximality.) Define the {\it highest derivative multisegment} of $\pi \in \mathrm{Irr}_{\rho}$ to be
\[  \mathfrak{hd}(\pi):=\sum_{c\in \mathbb Z} \mathfrak{mxpt}^a(\pi,c).
\]
The highest derivative multisegments of some special cases are given in Section \ref{ss examples}.

One main property of $\mathfrak{hd}$ is to give a new construction of the highest derivative of an irreducible representation:
\begin{theorem} \label{thm highest derivative multi} (=Theorem \ref{thm highest derivative multiseg})
Let $\pi \in \mathrm{Irr}_{\rho}$. Then 
\[  D_{\mathfrak{hd}(\pi)}(\pi)=\pi^- .\]
\end{theorem}





We now explain the second ingredient. In Section \ref{sec effect hd}, we define a combinatorial algorithm, called {\it removal process}, on a pair $(\Delta, \mathfrak h)$ for a segment $\Delta$ and a multisegment $\mathfrak h$. The algorithm outputs a multisegment, denoted by $\mathfrak r(\Delta, \mathfrak h)$. The case that we are interested in is when $\mathfrak h=\mathfrak{hd}(\pi)$. We also develop some rules and properties for computing $\mathfrak r(\Delta, \mathfrak h)$ in Section \ref{sec effect hd}, and the relation to $D_{\Delta}(\pi)$ is given in Theorem \ref{thm effect of Steinberg}.

For a multisegment $\mathfrak n=\left\{ \Delta_1, \ldots, \Delta_k\right\} \in \mathrm{Mult}_{\rho}$ with segments in $\mathfrak n$ labeled in an ascending order, define 
\begin{align} \label{eqn removal multi}
  \mathfrak r(\mathfrak n, \pi):= \mathfrak r(\Delta_k, \mathfrak r(\Delta_{k-1}, \ldots \mathfrak r(\Delta_1, \mathfrak{hd}(\pi))\ldots )) .
\end{align}
One useful property of the multisegment $\mathfrak r(\mathfrak n, \pi)$ is to measure the difference between the derivative $D_{\mathfrak n}(\pi)$ and the highest derivative $\pi^-$:

\begin{theorem} (=Theorem \ref{lem segment derivative}) \label{thm complement mutli}
Let $\pi \in \mathrm{Irr}_{\rho}$. Let $\mathfrak n \in \mathrm{Mult}_{\rho}$ such that $D_{\mathfrak n}(\pi)\neq 0$. Then
\[  D_{\mathfrak r(\mathfrak n, \pi)}\circ D_{\mathfrak n}(\pi) \cong \pi^- .\]
\end{theorem}
Theorem \ref{thm complement mutli} has applications in \cite{Ch22+b}. On the other hand, by shifting the 'right' branching law in Lemma \ref{lem derivative gives branching law} to the 'left' branching law, it gives an interpretation to Theorem \ref{thm complement mutli} (more details given in \cite{Ch22+b}), which is also a starting point of this article and the sequels \cite{Ch22+d, Ch22+e}.

The second main property of $\mathfrak r(\mathfrak n, \pi)$ is to determine when two derivatives give isomorphic simple quotients:

\begin{theorem} \label{thm derivative resultant multi} (=Corollary \ref{cor admissible by resultant}+Theorem \ref{thm isomorphic derivatives})
Let $\pi \in \mathrm{Irr}_{\rho}$. Let $\mathfrak n_1, \mathfrak n_2 \in \mathrm{Mult}_{\rho}$.
\begin{enumerate}
\item Then $D_{\mathfrak n_i}(\pi)\neq 0$ if and only if $\mathfrak n_i$ is admissible to $\mathfrak{hd}(\pi)$. (Refer the definition of admissibility to Definition \ref{def admissible multiseg})
\item Suppose $D_{\mathfrak n_1}(\pi)\neq 0$ and $D_{\mathfrak n_2}(\pi)\neq 0$. Then 
\[  D_{\mathfrak n_1}(\pi) \cong D_{\mathfrak n_2}(\pi) \quad \Longleftrightarrow \quad \mathfrak r(\mathfrak n_1, \pi) =\mathfrak r(\mathfrak n_2, \pi) .
\]
\end{enumerate}
\end{theorem}

Combining Theorem \ref{thm derivative resultant multi} with \cite{Ch22+b}, $\mathfrak r(\mathfrak n, \pi)$ provides a combinatorial invariant for describing the socle and cosocle of the Bernstein-Zelevinsky derivative of an irreducible representation. Applications of Theorem \ref{thm derivative resultant multi} will appear in the sequel \cite{Ch22+d}.

We finally comment on the proof of Theorems \ref{thm complement mutli} and \ref{thm derivative resultant multi}. The multisegment associated to the derivative $D_{\mathfrak n}(\pi)$ can be in general computed via explicit algorithms (see e.g. \cite{LM16}), but our proof does not directly use that. Our proof is more combinatorially soft in the sense that we use some commutation relations between derivatives studied in Section \ref{sec commutation of der}. Our proof also uses certain inductions via taking the $\rho$-derivatives.

\subsection{Remarks}

Apart from branching laws, there are many other applications for both derivatives. For example, it is important for unitarity by Tadi\'c \cite{Ta86}, theta correspondence by M\'inguez \cite{Mi08, Mi09}, $L$-functions (e.g. work of Matringe, Cogdell--Pietaski-Shapiro, Jo-Krishnamurthy \cite{Ma13, CPS17, JK22}), other distinction problems (e.g. Offen \cite{Of18}), and Aubert-Zelevinsky duals by M\oe glin-Waldspurger and Atobe-M\'inguez \cite{MW86, AM20}, and Arthur packets by Xu \cite{Xu17} and many others.  

We finally give some background of our study. In \cite{CS19} joint with Savin, we formulated the analogous Bernstein-Zelevinsky derivative functor for affine Hecke algebras of type A and so one could also formulate the analogous results in such setting, which will be explained in more detail in Section \ref{ss bz derivatives aha}. Some parts in this article are originally inspired by the work of Grojnowski-Vazirani \cite{GV01} in Hecke algebra setting few years ago, in which they used $\rho$-derivatives to study branching problems.

Using $\rho$-derivatives to study Bernstein-Zelevinsky derivatives also explicitly appears before, for example Deng \cite{De16} studying with orbital varieties and a more recent work of Gurevich \cite{Gu21} studying with RSK model. However, we emphasis that St-derivatives are important in the study of simple quotients of classical Bernstein-Zelevinsky derivatives, while $\rho$-derivatives seem to be not enough for such purpose. In particular, using machinery in Sections \ref{sec effect hd} and \ref{ss isomorphic bz derivatives}, one can find some simple quotients of Bernstein-Zelevinsky derivatives which cannot be constructed from a sequence of $\rho$-derivatives (see Example \ref{example rho not BZ}).
 
\subsection{Organization of this article}
Section \ref{ss prelim} defines some basic notations and results such as the Zelevinsky and Langlands classification. Section \ref{ss two derivatives} gives some basic results of using Jacquet functors to study Bernstein-Zelevinsky derivatives. Section \ref{s multi at a point} studies a notion of multisegment at a right point, which is a useful tool while Section \ref{s multisegment at a left pt} studies a notion of multisegment at a left point, which is an important ingredient in defining the highest derivative multisegment. Section \ref{sec commutation of der} studies some preliminary results for commutativity of derivatives (mainly by the geometric lemma). Section \ref{s highest der multisegment} defines the highest derivative multisegment and shows that its corresponding derivative gives the highest derivative (Theorem \ref{thm highest derivative multiseg}). Section \ref{sec effect hd} introduces the removal process and Section \ref{s comparison theorem} compares the removal process (combinatorial side) and the St-derivative (representation-theoretic side). Section \ref{ss isomorphic bz derivatives} proves that the effect of ascending sequences of derivatives is determined by the removal process (Theorem \ref{thm isomorphic derivatives}). Explicit descriptions of the highest derivatives of some representations are given in Section \ref{ss examples}. In the Appendix, we discuss how to transfer results to affine Hecke algebras of type A.

\subsection{Summary}

We summarize relations of several notions and results. Let $\pi \in \mathrm{Irr}_{\rho}$. Note that the notion $\mathfrak{mxpt}^b(\pi,c)$ is defined in Section \ref{s multi at a point} and other notions are discussed above.

\[ \xymatrix{ \varepsilon_{\Delta}(\pi) \ar[d]_{(\ref{eqn mxpt a})}  \ar[r]^{\mbox{Def. \ref{def max R multiseg}}} & \mathfrak{mxpt}^b(\pi,c) \ar[rr]^{\mbox{tools e.g. Prop \ref{prop prod highest derivative 1}}} &  &  \mbox{Examples in Sec. \ref{ss examples}} & \\  \mathfrak{mxpt}^a(\pi,c)  \ar[d]_{\mbox{(\ref{eqn define highest derivative mult})}} &  &  & & \\
                     \mathfrak{hd}(\pi)   \ar[d]_{\mbox{(\ref{eqn removal for hd pi})}} \ar@{<->}[rr]^{\mbox{Theorem \ref{thm highest derivative multiseg}}}  &    &     \pi^-  & &\\
					\mathfrak{r}(\Delta, \pi) \ar[d]_{(\ref{eqn removal multi})}    \ar@{<->}[rr]^{\mbox{Theorem \ref{thm effect of Steinberg}}}   && D_{\Delta}(\pi) \ar[d]^{(\ref{eqn derivatives of multisegment})}   &   &	\\
					\mathfrak{r}(\mathfrak n, \pi) \ar[rr]^{\mbox{Theorem \ref{thm isomorphic derivatives}}} &&  D_{\mathfrak n}(\pi) \ar@{->}[rr]^{\mbox{Prop. \ref{prop unique as segments}}} &  			& \mbox{simple quotient of $\pi^{(i)}$} \ar[uull]_{\mbox{double derivative (Theorem \ref{lem segment derivative})}}								}
\]

\subsection{Acknowledgement} The author would like to thank the referee for very helpful comments. This project is partly supported by the Research Grants Council of the Hong Kong Special Administrative Region, China (Project No: 	17305223) and NSFC grant for Excellent Young Scholar (Project No.: 12322120).

\section{Preliminaries} \label{ss prelim}

\subsection{Notations} \label{ss notations}

All representations that we consider are smooth and over $\mathbb{C}$. We will typically omit these descriptions and do not usually distinguish representations in the same isomorphism class. We keep using the notations in Section \ref{ss notion derivative}. Let $\mathrm{Alg}(G_n)$ be the category of smooth representations of $G_n$. For $\pi \in \mathrm{Alg}(G_n)$, denote by $\pi^{\vee}$ the smooth dual of $\pi$.

 For $\pi \in \mathrm{Irr}$, $\mathrm{deg}(\pi)$ is defined to be the number that $\pi \in \mathrm{Irr}(G_{\mathrm{deg}(\pi)})$. Let $\mathrm{Irr}^c(G_n)$ be the set of (irreducible) cuspidal representations of $G_n$. Let $\mathrm{Irr}^c=\sqcup_n \mathrm{Irr}^c(G_n)$.

For any $\pi_1 \in \mathrm{Alg}(G_{n_1})$ and $\pi_2 \in \mathrm{Alg}(G_{n_2})$, define the normalized parabolically induced module as:
\[  \pi_1 \times \pi_2 = \mathrm{Ind}_{P_{n_1,n_2}}^{G_{n_1+n_2}} (\pi_1\boxtimes \pi_2),\]
where we inflate the $G_{n_1}\times G_{n_2}$-representation $\pi_1 \boxtimes \pi_2$ to a $P_{n_1,n_2}$-representation. Here $\mathrm{Ind}$ is the normalized parabolic induction.  

                                                                                                                                                                                                            For $a,b \in \mathbb Z$ with $b-a \in \mathbb Z_{\geq 0}$ and $\rho \in \mathrm{Irr}^c$, the segment
\begin{align} \label{eqn seg def}
[a,b]_{\rho}:=\left\{ \nu^a\rho, \ldots, \nu^b\rho \right\} 
\end{align}
is considered to be a set so that the intersection and union can be carried out (see Section \ref{ss intersect union oper} below).  We also set $[a,a-1]_{\rho}=\emptyset$ for $a \in \mathbb Z$. For a segment $\Delta=[a,b]_{\rho}$, we write $a(\Delta)=\nu^a\rho$ and $b(\Delta)=\nu^b\rho$. We also write: \[[a]_{\rho} :=[a,a]_{\rho}. \]
We may also write $[\nu^a\rho, \nu^b\rho]$ for $[a,b]_{\rho}$ and write $[\nu^a\rho]$ for $[a]_{\rho}$. Recall that $l_{abs}([a,b]_{\rho})$ is defined in Section \ref{ss notion derivative}. Two segments $[a,b]_{\rho}$ and $[a',b']_{\rho'}$ are said to be equal if $\nu^{a'}\rho' \cong \nu^a \rho$ and $b-a+1=b'-a'+1$. When $\rho=1$ is the trivial representation of $G_1$, we simply write $[a,b]$ for the segment $[a,b]_1$. 

Two segments $\Delta$ and $\Delta'$ are said to be {\it linked} if $\Delta \cup \Delta'$ is still a segment in the sense of (\ref{eqn seg def}), and $\Delta \not\subset \Delta'$ and $\Delta' \not\subset \Delta$. Otherwise, it is called to be not linked or unlinked.
																																																	
For any $\pi \in \mathrm{Irr}$, there exist $\rho_1, \ldots, \rho_r \in \mathrm{Irr}^c$ such that $\pi$ is a simple composition factor for $\rho_1 \times \ldots \times \rho_r$, and we shall denote such multiset to be $\mathrm{csupp}(\pi)$, which is called the {\it cuspidal support} of $\pi$. For a multisegment $\mathfrak m$, define $\mathrm{csupp}(\mathfrak m)=\cup_{\Delta \in \mathrm{m}} \Delta$ as a multiset.

																																																																															\subsection{More notations for multisegments} \label{ss notations multi}
		For two multisegments $\mathfrak m$ and $\mathfrak n$, define $\mathfrak m+\mathfrak n$ to be the sum of two multisegments (counting multiplicities). For a multisegment $\mathfrak m$ and a segment $\Delta$, define
\[ \mathfrak m+\Delta= \left\{ \begin{array}{cc} \mathfrak m+\left\{ \Delta \right\} & \mbox{ if $\Delta\neq \emptyset$ } \\
                             \mathfrak m & \mbox{ if $\Delta=\emptyset$ }
													\end{array}
												\right.
		\]
The notions $\mathfrak m-\mathfrak n$ and  $\mathfrak m-\Delta$ are defined in a similar way, provided that $\mathfrak n$ is a submultisegment of $\mathfrak m$ and $\Delta \in \mathfrak m$ respectively. 	

For a multisegment $\mathfrak m=\left\{ \Delta_1, \ldots, \Delta_k \right\}$, define
\[  l_{abs}(\mathfrak m)=l_{abs}(\Delta_1)+\ldots+l_{abs}(\Delta_k) .
\]

For $\rho_1, \rho_2 \in \mathrm{Irr}^c$, we write $\rho_2 < \rho_1$ (resp. $\rho_2\leq \rho_1$) if $\rho_1 \cong \nu^a \rho_2$ for some integer $a > 0$ (resp. $a \geq 0$). For two segments $\Delta_1, \Delta_2$, we write $\Delta_1 <\Delta_2$ if $\Delta_1$ and $\Delta_2$ are linked and $b(\Delta_1)< b(\Delta_2)$.


 For an integer $c$ and $\rho \in \mathrm{Irr}^c$, let $\mathrm{Mult}_{\rho,c}^a$ be the subset of $\mathrm{Mult}_{\rho}$ containing all multisegments $\mathfrak m$ such that any segment $\Delta$ in $\mathfrak m$ satisfies $a(\Delta) \cong \nu^c\rho$. Define $\mathrm{Mult}_{\rho,c}^b$ to be the subset of $\mathrm{Mult}_{\rho}$ containing all multisegments $\mathfrak m$ such that any segment $\Delta$ in $\mathfrak m$ satisfies $b(\Delta) \cong \nu^c\rho$. The empty set is also considered to be an element in both $\mathrm{Mult}_{\rho,c}^a$ and $\mathrm{Mult}_{\rho,c}^b$. 

For a multisegment $\mathfrak m$ in $\mathrm{Mult}_{\rho}$ and an integer $c$, define $\mathfrak m[c]$ to be the submultisegment of $\mathfrak m$ containing all the segments $\Delta$ satisfying $a(\Delta)\cong \nu^c\rho$; and define $\mathfrak m\langle c \rangle$ to be the submultisegment of $\mathfrak m$ containing all the segments $\Delta$ satisfying $b(\Delta)\cong \nu^c\rho$.


\subsection{Ordering on segments} \label{ss ordering}

For two non-empty segments $[a',b']_{\rho}$ and $[a'',b'']_{\rho}$, we write 
\[  [a',b']_{\rho} \prec^L [a'',b'']_{\rho} 
\]
if either $a'<a''$; or $a'=a''$ and $b'<b''$. We also write $[a',b']_{\rho} \preceq^L [a'',b'']_{\rho}$ if $[a',b']_{\rho} \prec^L [a'',b'']_{\rho}$ or $[a',b']_{\rho}=[a'',b'']_{\rho}$. The ordering $\prec^R$ can be defined in a similar manner by using $b$-values.


\subsection{Intersection-union operation} \label{ss intersect union oper}
We say that a multisegment $\mathfrak m_2$ is obtained from $\mathfrak m_1$ by an {\it elementary intersection-union operation} if  for two segments $\Delta_1, \Delta_2$ in $\mathfrak m_1$, 
\[  \mathfrak m_2=\mathfrak m_1-\left\{ \Delta_1, \Delta_2\right\} + \Delta_1\cup \Delta_2 + \Delta_1\cap \Delta_2  .
\] 
The ordering $\leq_Z$ is defined in Section \ref{ss main results}.

\subsection{Ordering on $\mathrm{Mult}_{\rho, c}^a$ and $\mathrm{Mult}_{\rho,c}^b$} \label{ss ordering mult}

Fix an integer $c$. Let $\Delta_1=[c,b_1]_{\rho}, \Delta_2=[c,b_2]_{\rho}$ be two non-empty segments. We write $\Delta_1 \leq_c^a \Delta_2$ if $b_1 \leq b_2$, and write $\Delta_1 <_c^a \Delta_2$ if $b_1<b_2$.

For non-empty $\mathfrak m_1, \mathfrak m_2$ in $\mathrm{Mult}_{\rho, c}^a$, label the segments in $\mathfrak m_1$ as: $\Delta_{1,k} \leq_c^a   \ldots \leq_c^a \Delta_{1,2} \leq_c^a \Delta_{1,1}$ and label the segments in $\mathfrak m_2$ as: $\Delta_{2,r} \leq_c^a \ldots  \leq_c^a \Delta_{2,2}  \leq_c^a \Delta_{2,1}$. We define the lexicographical ordering: $\mathfrak m_1 \leq_c^a \mathfrak m_2$ if $k \leq r$, and for any $i \leq k$, $\Delta_{1,i} \leq_c^a \Delta_{2,i}$. We write $\mathfrak m_1 <_c^a \mathfrak m_2$ if $\mathfrak m_1 \leq_c^a \mathfrak m_2$ and $\mathfrak m_1\neq \mathfrak m_2$.

We also need a 'right' ordering. One can define $[a_1,c]_{\rho}\leq_c^b[a_2,c]_{\rho}$ if $a_1<a_2$, and similarly define $[a_1,c]_{\rho}<^b_c [a_2,c]_{\rho}$. One similarly define $\leq_c^b$ on $\mathrm{Mult}_{\rho,c}^b$.

\subsection{Zelevinsky and Langlands classification} \label{ss zel langlands class}

For a segment $\Delta=[a,b]_{\rho}$, define $\langle \Delta \rangle$ to be the the unique simple submodule of
\[  \nu^a \rho \times \ldots \times \nu^b \rho
\]
and define $\mathrm{St}(\Delta)$ to be the unique simple quotient of 
\[  \nu^a \rho \times \ldots \times \nu^b\rho . \]

For any multisegment $\mathfrak m=\left\{ \Delta_1, \ldots, \Delta_k \right\}$, we may assume the labellings satisfy that for $i<j$, $\Delta_i \not <\Delta_j$. With such labellings, we define, as in \cite[Theorem 6.1]{Ze80}, $\langle \mathfrak m \rangle$ to be the unique simple submodule of
\begin{align} \label{eqn zelevinsky induced modules}
 \zeta(\mathfrak m):= \langle \Delta_1 \rangle \times \ldots \times \langle \Delta_k \rangle,
\end{align}
and define $\mathrm{St}(\mathfrak m)$ to be the unique simple quotient of 
\[ \lambda(\mathfrak m):=  \mathrm{St}(\Delta_1) \times \ldots \times \mathrm{St}(\Delta_k) .
\]

We frequently use the following standard fact (see \cite[Theorems 4.2, 6.1]{Ze80}): for two unlinked segments $\Delta_1$ and $\Delta_2$,
\begin{align}
 \langle \Delta_1 \rangle \times \langle \Delta_2 \rangle \cong \langle \Delta_2 \rangle \times \langle \Delta_1 \rangle,
\end{align}
\begin{align} \label{eqn product commut st}
 \mathrm{St}(\Delta_1) \times \mathrm{St}(\Delta_2) \cong \mathrm{St}(\Delta_2) \times \mathrm{St}(\Delta_1) .
\end{align}

\subsection{Geometric lemma} \label{ss geo lemma}

The geometric lemma is a key tool in our study. We shall describe a special case that we frequently use. 

For $n_1+ \ldots+ n_k=n$, we write $P_{n_1, \ldots, n_k}$ to be the parabolic subgroup generated by the matrices $\mathrm{diag}(g_1, \ldots, g_k)$, where $g_j \in G_{n_j}$, and upper triangular matrices; and write $N_{n_1, \ldots, n_k}$ to be the unipotent radical of $P_{n_1, \ldots, n_k}$. For $i \leq n$, $N_i$ defined in Section \ref{ss notion derivative} is the unipotent radical of the parabolic subgroup $P_{n-i, i}$. 

Let $\pi_1 \in \mathrm{Alg}(G_{n_1})$ and let $\pi_2 \in \mathrm{Alg}(G_{n_2})$. Let $n=n_1+n_2$. Then the geometric lemma on $(\pi_1 \times \pi_2)_{N_i}$ asserts that $(\pi_1 \times \pi_2)_{N_i}$ admits a filtration whose successive subquotients take the form:
\[  \mathrm{Ind}_{P_{n_1-i_1,n_2-i_2} \times P_{i_1,i_2}}^{G_{n-i}\times G_i} ((\pi_1)_{N_{i_1}}\boxtimes (\pi_2)_{N_{i_2}})^{\phi} ,
\]
where $i_1+i_2=i$. Here $((\pi_1)_{N_{i_1}} \boxtimes (\pi_2)_{N_{i_2}})^{\phi}$ is a $G_{n_1-i_1}\times G_{n_2-i_2}\times G_{i_1}\times G_{i_2}$-representation with underlying space $(\pi_1)_{N_{i_1}} \boxtimes (\pi_2)_{N_{i_2}}$ determined by the action:
\[  (g_1,g_2,g_3,g_4).(v_1\boxtimes v_2) =(g_1,g_3).v_1 \boxtimes (g_2,g_4).v_2 ,
\]
where $v_1 \in (\pi_1)_{N_{i_1}} $ and $v_2 \in (\pi_2)_{N_{i_2}}$. 

\subsection{Quotients and submodules of Jacquet functors} \label{ss quotients and sub jacquet}

Let $\theta=\theta_n: G_n \rightarrow G_n$ be given by $\theta(g)=g^{-t}$, the inverse transpose of $g$. This induces a self-equivalence exact functor on $\mathrm{Alg}(G_n)$, still denoted by $\theta$. We shall call it the {\it Gelfand-Kazhdan involution}.

\begin{proposition} \label{prop sub quo Jacquet}
Let $\pi \in \mathrm{Irr}(G_n)$. Let $n_1+\ldots+n_r=n$. Let $N$ be the unipotent radical of the parabolic subgroup $P_{n_1,\ldots, n_r}$. Let $\theta'$ be the involution on $\mathrm{Alg}(G_{n_1}\times \ldots \times G_{n_r})$ arisen from $\theta'(g_1, \ldots, g_r)=(\theta(g_1), \ldots, \theta(g_r))$. Then $\theta'(\pi_N)^{\vee} \cong \pi_N$. In particular, for an irreducible representation $\omega$ of $G_{n_1}\times \ldots \times G_{n_r}$, $\omega$ is a simple submodule of $\pi_N$ if and only if  $\omega$ is a simple quotient of $\pi_N$.

\end{proposition}

\begin{proof}
Recall that $\theta(\pi)$ and $\pi$ have the same underlying space, which we refer to $V$. Note that
\[  W:= \left\{ \theta(n).v-v: n \in N^-, v \in V \right\} = \left\{ n.v-v : n \in N, v \in V \right\} \subset V.
\]
Then it induces a natural identification as vector spaces:
\[  \theta_n(\pi)_{N^-} = \theta'(\pi_N') =V/W .
\]
Now one checks the isomorphism is also a $G_{n_1}\times \ldots \times G_{n_r}$-morphism. This proves that:
\begin{align} \label{eqn isom under opposite 1}
 \quad \theta(\pi)_{N^-}\cong \theta'(\pi_N) .
\end{align}

On the other hand, 
\begin{align} \label{eqn isom under opposite 2}
  (\theta(\pi)_{N^-})^{\vee} \cong (\theta(\pi)^{\vee})_N \cong \pi_N,
\end{align}
where the first isomorphism follows from a result of Bernstein-Casselman \cite[Page 66]{Be92} and \cite[Corollary 4.2.5]{Ca95}, and the second isomorphism follows from \cite[Theorem 7.3]{BZ76}. The proposition follows by combining (\ref{eqn isom under opposite 1}) and (\ref{eqn isom under opposite 2}). 
\end{proof}


\subsection{Opposite parabolic induction}

\begin{lemma} \label{lem twist the isomorphism}
Let $n_1, \ldots, n_k$ be positive integers and let $n=n_1+\ldots+n_k$. Let $\pi_k \in \mathrm{Alg}(G_{n_k})$ ($k=1, \ldots r$). Let $P^-$ be the parabolic subgroup opposite to $P_{n_1, \ldots, n_r}$. Then the normalized parabolically induced module $\mathrm{Ind}_{P^-}^{G_n} (\pi_1\boxtimes \ldots \boxtimes \pi_r)$ is naturally isomorphic to $\pi_r\times \ldots \times \pi_1$.
\end{lemma}

\begin{proof}
For $f \in \mathrm{Ind}_{P^-}^{G_n} (\pi_1\boxtimes \ldots \boxtimes \pi_r)$, one defines $\Lambda(f)(g)=f(\dot{w}_0g)$, where $\dot{w}_0$ is the anti-diagonal matrix in $G_n$. Note that $\dot{w}_0^{-1}(N_{n_r, \ldots, n_1})\dot{w}_0$ is the unipotent radical of $P^-$; $\dot{w}_0^{-1}P_{n_r,\ldots, n_1}\dot{w}_0=P^-$ for each $g_k \in G_{n_k}$ and $\dot{w}_0$ also twists the modular character of $P_{n_r, \ldots,n_1}$ to the modular character of $P^-$. Then $\Lambda$ gives the isomorphism.
\end{proof}

\subsection{Jacquet functors on Steinberg representations and segment representations} \label{ss jacquet steinberg }

We shall frequently use the following formulas \cite[Proposition 3.4 and Proposition 9.5]{Ze80}: for any integer $i$ and a segment $[a,b]_{\rho}$,
\begin{align} \label{eqn segment jacquet}
  \langle [a,b]_{\rho} \rangle_{N_{i \cdot \mathrm{deg}(\rho)}} = \langle [a, b-i]_{\rho} \rangle \boxtimes \langle [b-i+1, b]_{\rho} \rangle
\end{align}
and 
\begin{align} \label{eqn steinberg jacquet}
 \mathrm{St}([a,b]_{\rho})_{N_{i \cdot \mathrm{deg}(\rho)}}= \mathrm{St}([a+i, b]_{\rho})\boxtimes \mathrm{St}([a,a+i-1]_{\rho}) .
\end{align}

Later on, we sometimes say to use the geometric lemma and compare cuspidal supports, which we mean to use the geometric lemma in Section \ref{ss geo lemma} and then use the Jacquet functor computations of (\ref{eqn segment jacquet}) and/or (\ref{eqn steinberg jacquet}) above.

\section{Bernstein-Zelevinsky derivatives via Jacquet functors} \label{ss two derivatives}

The main results of this section are Propositions \ref{prop embedding deri sub} and \ref{prop unique as segments}, which use Jacquet functors and the duals of standard modules to construct some simple submodules of Bernstein-Zelevinsky derivatives. We fix a cuspidal representation $\rho \in \mathrm{Irr}^c$ for the rest of the article.

\subsection{$\rho$-derivatives} \label{ss rho derivatives}

For a representation $\pi$ of some $G_m$, define $\pi^{\times k}$ to be the representation of $G_{mk}$
\[ \overbrace{\pi \times \ldots \times \pi}^{k \ \mbox{ times}}. \]

\begin{lemma}  \cite{GV01, Ja07, Mi09}\label{lem unique sub quo product}
Let $\pi \in \mathrm{Irr}_{\rho}$ and let $c \in \mathbb Z$. For any non-negative integer $k$,
\[    \pi \times (\nu^c\rho)^{\times k}
\]
has unique irreducible submodule and unique irreducible quotient.
\end{lemma}

Applying Frobenius reciprocity on Lemma \ref{lem unique sub quo product} for $k=1$, one has:
\begin{lemma} \label{lem embedding exists}
Let $\pi \in \mathrm{Irr}_{\rho}(G_n)$ and let $c \in \mathbb Z$.  There is at most one irreducible representation $\tau \in \mathrm{Irr}(G_{n-r})$ such that
\[  \tau \boxtimes (\nu^c\rho) \hookrightarrow \pi_{N_r} .
\]
\end{lemma}

\begin{notation} \label{notn derivatives epsilon}
\begin{enumerate}
\item For $c \in \mathbb Z$ and $\pi \in \mathrm{Irr}_{\rho}$, if there exists $\tau \in \mathrm{Irr}$ such that
\[   \tau \boxtimes \nu^c\rho \hookrightarrow  \pi_{N_{\mathrm{deg}(\rho)}}, \]
we shall denote $D_c(\pi)$ to be such $\tau$. Otherwise, set $D_c(\pi)=0$. The notion $D_c$ is well-defined from Lemma \ref{lem embedding exists}. We shall call $D_c(\pi)$ to be a {\it $\rho$-derivative} of $\pi$ (depending on $c$).
\item For a non-negative integer $k$, we shall write $D_c^k(\pi)$ for
\[  \overbrace{D_c\circ \ldots \circ D_c}^{k \mbox{ times}}(\pi) .
\] 
When $k=0$, it simply means that $D_c^0(\pi)=\pi$. 
\item We shall denote the largest non-negative integer $k$ such that $D^k_c(\pi)\neq 0$ by $\varepsilon_c(\pi)$. 
\end{enumerate}
\end{notation}

\subsection{Properties of $\rho$-derivatives}

\begin{lemma} \label{lem epsilon for rho der} (see \cite[Lemma 3.5]{GV01}, \cite[Corollary 2.3.2]{Ja07}, \cite[Corollaire 6.5.]{Mi09})
Let $\pi \in \mathrm{Irr}_{\rho}$. Let $c \in \mathbb Z$. Let $\widetilde{\pi}$ be the unique submodule of $ \pi \times (\nu^c\rho)^{\times k}$. Then 
\begin{enumerate}
\item $\varepsilon_c(\widetilde{\pi})=\varepsilon_c(\pi)+k$;
\item $\widetilde{\pi}$ appears with multiplicity one in the Jordan-H\"older series of $\pi \times (\nu^c\rho)^{\times k}$;
\item for any irreducible composition factor $\tau$ of $ \pi \times (\nu^c\rho)^{\times k}$ which is not isomorphic to $\widetilde{\pi}$, $\varepsilon_c(\tau)<\varepsilon_c(\pi)+k$;
\item $\pi \cong (D_c)^{k}(\widetilde{\pi})$.
\end{enumerate}
\end{lemma}

The above result follows from an application on the geometric lemma, and may see  \cite[Propositions 11.1 and 11.2]{Ch22+} for more discussions.


 As a consequence, we have the following:

\begin{corollary} \label{cor embeding derivative}
Let $\pi \in \mathrm{Irr}_{\rho}(G_n)$. Let $k=\varepsilon_c(\pi)$. Let $\omega$ be an admissible $G_{n-k\mathrm{deg}(\rho)}$-representation such that 
\[   \pi \hookrightarrow   \omega \times (\nu^c\rho)^{\times k} 
\]
Then $D^k_c(\pi) \hookrightarrow \omega$.
\end{corollary}
\begin{proof}
By Frobenius reciprocity, we have a non-zero map $\pi_{N_{k\cdot \mathrm{deg}(\rho)}}  \rightarrow \omega \boxtimes (\nu^c\rho)^{\times k}$. By Lemma \ref{lem epsilon for rho der}, the only composition factor of the form $\tau \boxtimes (\nu^c\rho)^{\times k}$ is $D_c^k(\pi)\boxtimes (\nu^c\rho)^{\times k}$ and hence that factor must be mapped to a submodule of $\omega \boxtimes (\nu^c \rho)^{\times k}$. It then follows from K\"unneth formula (see e.g. \cite{Ra07}) that $D_c^k(\pi)$ is a submodule of $\omega$. 
\end{proof}

\subsection{Notations for derivatives} \label{ss notation derivatives}

For a segment $\Delta=[a,b]_{\rho}$, write
\[  {}^-\Delta =[a+1,b]_{\rho}, \quad \Delta^-=[a,b-1]_{\rho}.
  \]
 We also define ${}^{0}\Delta= \Delta^{0}=\Delta$. 

For a multisegment $\mathfrak m=\left\{ \Delta_1, \ldots, \Delta_r\right\}$ in $\mathrm{Mult}$, define
\[  \mathfrak m^-=\Delta_1^-+\ldots +\Delta_r^- ,\quad {}^-\mathfrak m= {}^-\Delta_1+ \ldots+ {}^-\Delta_r .
\]

\subsection{Highest derivatives via $\rho$-derivatives}

For a multisegment $\mathfrak m \in \mathrm{Mult}_{\rho}$, define
\[ \mathrm{mult}^b(\mathfrak m, c)=|\left\{  \Delta \in \mathfrak m:  b(\Delta)\cong \nu^c\rho  \right\}|. \] 
For $\pi \in \mathrm{Irr}_{\rho}$, define
\begin{align} \label{eqn mult b values}
 \mathrm{mult}^b(\pi, c):=\mathrm{mult}^b(\mathfrak m,c) 
\end{align}
where $\mathfrak m \in \mathrm{Mult}_{\rho}$ such that $\pi \cong \langle \mathfrak m \rangle$. 

Recall that the highest derivative is defined in Section \ref{ss notion derivative}. We have the following construction of the highest derivatives from $\rho$-derivatives.

\begin{proposition} \label{prop construction highest from rho}
Let $\mathfrak m \in \mathrm{Mult}_{\rho}$ and let $\pi=\langle \mathfrak m \rangle$. Let $c$ (resp. $d$) be the smallest (resp. largest) integer such that $\nu^c\rho \cong b(\Delta)$ (resp. $\nu^d\rho \cong b(\Delta)$  for some $\Delta \in \mathfrak m$. For each $e=c,\ldots, d$, let 
\[ k_e =|\left\{ \Delta \in \mathfrak m: b(\Delta) \cong \nu^e\rho \right\}|. \] Then $D_d^{k_d}\circ \ldots \circ D_{c+1}^{k_{c+1}} \circ D_{c}^{k_c}(\pi) \cong \pi^{-}$.
\end{proposition}


We need the following Lemma \ref{lem simple compute epsilon factor} to prove Proposition \ref{prop construction highest from rho}. One possible proof is to use explicit rules of $\rho$-derivatives in \cite[Th\'eor\`eme 7.5]{Mi09}, relying on results of M\oe glin-Waldspurger \cite[Lemme II.11]{MW86} on computing Zelevinsky duals. We shall not reproduce those explicit rules here. Instead, for the convenience of the reader, we sketch a quick proof using some more elementary representation-theoretic facts.

\begin{lemma} \label{lem simple compute epsilon factor}
Let $\mathfrak m \in \mathrm{Mult}_{\rho}$ and let $\pi=\langle \mathfrak m \rangle$. Suppose, for some $e \in \mathbb Z$ such that $\mathrm{mult}^b(\mathfrak m, e-1)=0$. Then $\varepsilon_e(\pi)=\mathrm{mult}^b(\pi, e)$, and $(D_e)^{\mathrm{mult}^b(\pi,e)}(\pi) \cong \langle \mathfrak m-\mathfrak m\langle e \rangle+(\mathfrak m\langle e \rangle)^- \rangle$, where $\mathfrak m\langle e\rangle$ is defined in Section \ref{ss notations multi}.
\end{lemma}

\begin{proof}
 Using the Zelevinsky classification, we have an embedding:
\[   \pi \hookrightarrow \langle \mathfrak m_1 \rangle \times \langle \mathfrak m_2 \rangle \times \langle \mathfrak m_3 \rangle ,
\] 
where 
\begin{itemize}
 \item $\mathfrak m_1$ is the submultisegment of $\mathfrak m$ precisely containing all segments in $\mathfrak m$ with $b(\Delta)>\nu^e\rho$;
 \item $\mathfrak m_2=\mathfrak m\langle e \rangle$;
 \item $\mathfrak m_3$ is the submultisegment of $\mathfrak m$ precisely containing all segments in $\mathfrak m$ with $b(\Delta)<\nu^e\rho$.
\end{itemize}
On the other hand, we also have the following embedding:
\[  \langle \mathfrak m_2 \rangle \hookrightarrow  \langle \mathfrak m_2^- \rangle \times (\nu^e \rho)^{\times \mathrm{mult}^b(\pi, e)} .
\]
Combining above embeddings, we also have:
\[   \pi \hookrightarrow \langle \mathfrak m_1 \rangle \times \langle \mathfrak m_2^- \rangle \times (\nu^e \rho)^{\times \mathrm{mult}^b(\pi, e)}  \times \langle \mathfrak m_3 \rangle ,
\] 
Now, the condition that $\mathrm{mult}^b(\pi, e-1)=0$ implies all segments in $\mathfrak m_3$ is not linked to $[e]_{\rho}$, and this implies that 
\[  (\nu^e \rho)^{\times \mathrm{mult}^b(\pi, e)} \times \langle \mathfrak m_3 \rangle  \cong \langle \mathfrak m_3 \rangle \times  (\nu^e \rho)^{\times \mathrm{mult}^b(\pi, e)} .
\]
Hence, $\pi \hookrightarrow \langle \mathfrak m_1 \rangle \times \langle \mathfrak m_2^- \rangle \times \langle \mathfrak m_3 \rangle \times (\nu^e \rho)^{\times \mathrm{mult}_b(\pi, e)}$. Thus, by Frobenius reciprocity, we obtain a non-zero map:
\[   \pi_{N_{k}} \rightarrow \langle \mathfrak m_1 \rangle \times \langle \mathfrak m_2^- \rangle \times \langle \mathfrak m_3 \rangle \boxtimes (\nu^e\rho)^{\times \mathrm{mult}^b(\pi, e)} ,
\]
where $k=\mathrm{deg}(\rho)\cdot \mathrm{mult}^b(\pi, e)$. This, in particular, implies that $\pi_{N_k}$ has an irreducible submodule of the form $\omega \boxtimes (\nu^e\rho)^{\times \mathrm{mult}^b(\pi, e)}$ and so $(D_e)^{\mathrm{mult}^b(\pi, e)}(\pi)\neq 0$. Thus, $\varepsilon_e(\pi)\geq \mathrm{mult}^b(\pi, e)$. 

It remains to show that $\varepsilon_e(\pi)> \mathrm{mult}^b(\pi,e)$ is not possible. To this end, we consider:
\[   \pi \hookrightarrow \zeta(\mathfrak m) 
\]
and so, apply the Jacquet functor to obtain:
\[  \pi_{N_l} \hookrightarrow \zeta(\mathfrak m)_{N_l} 
\]
for some $l$. Recall that $\zeta(\mathfrak m)$ can be written as a parabolically induced module and so one can again apply the geometric lemma on $\zeta(\mathfrak m)_{N_l}$. Now, from (\ref{eqn segment jacquet}), one sees that a segment in $\mathfrak m$ that can precisely contribute a factor $\nu^e\rho$ must come from those in $\mathfrak m_2=\widetilde{\mathfrak m}$. Hence, if $l>k$, it is impossible to have a simple composition factor in $\zeta(\mathfrak m)_{N_l}$ of the form $\omega \boxtimes (\nu^e\rho)^{\times p}$ (for some $p$). This implies that $\varepsilon_e(\pi)> \mathrm{mult}^b(\pi,e)$ is impossible, as desired.

We now prove the second assertion. We again consider the embedding:
\[  \pi \hookrightarrow \zeta(\mathfrak m) .
\]
Thus we have:
\begin{align} \label{eqn embedding in derivative}
  (D_e)^{\mathrm{mult}^b(\pi,e)}(\pi) \boxtimes (\nu^e\rho)^{\times \mathrm{mult}^b(\pi,e)} \hookrightarrow \pi_{N_k} \hookrightarrow \zeta(\mathfrak m)_{N_k} . 
	\end{align}
As discussed above, one considers the filtration from the geometric lemma. And, by (\ref{eqn segment jacquet}) again, the only segments that contribute $\nu^e\rho$ come from those in $\mathfrak m\langle e\rangle$. Thus, we only have one layer that can contribute to the embedding in (\ref{eqn embedding in derivative}), namely the layer of the form $\zeta(\mathfrak m-\mathfrak m\langle e \rangle+\mathfrak m\langle e \rangle^-)\boxtimes (\nu^e\rho)^{\times \mathrm{mult}^b(\pi,e)}$. Thus, we have:
$(D_e)^{\mathrm{mult}^b(\pi,e)}(\pi) \boxtimes (\nu^e\rho)^{\times \mathrm{mult}^b(\pi,e)} \hookrightarrow \zeta(\mathfrak m-\mathfrak m\langle e \rangle+\mathfrak m\langle e \rangle^-)\boxtimes (\nu^e\rho)^{\times \mathrm{mult}^b(\pi,e)}$ and so
\[ (D_e)^{\mathrm{mult}^b(\pi,e)}(\pi)  \hookrightarrow \zeta(\mathfrak m-\mathfrak m\langle e \rangle+\mathfrak m\langle e \rangle^-) .\]
The last embedding gives the Zelevinsky parameter $(D_e)^{\mathrm{mult}^b(\pi,e)}(\pi)$.
\end{proof}

\noindent
{\it Proof of Proposition \ref{prop construction highest from rho}.} Inductively, using Lemmas \ref{lem simple compute epsilon factor} and \ref{lem epsilon for rho der}, we have that $\mathrm{mult}^b(D_e^{k_e}\circ \ldots \circ D_c^{k_c}(\pi), e)=0$ and $D_e^{k_e}\circ \ldots \circ D_c^{k_c}(\pi)\neq 0$. By the explicit Zelevinsky parameter in Lemma \ref{lem simple compute epsilon factor}, we have that the Zelevinsky parameter $D_d^{k_d}\circ \ldots \circ D_{c+1}^{k_{c+1}} \circ D_{c}^{k_c}(\pi)$ is the multisegment $\mathfrak m^-$. Comparing with the description of the highest derivative in \cite[Theorem 8.1]{Ze80}, we have the desired isomorphism. \qed

\subsection{Left-right Bernstein-Zelevinsky derivatives} \label{sec left derivative}
Recall that the Bernstein-Zelevinsky derivative is defined in Section \ref{ss notion derivative}. We also define a left version (c.f. \cite{CS21, Ch21, Ch21+}), which one can use the transpose $R_i^t$ of $R_i$ to define the left BZ derivative as in (\ref{eqn def derivative}). One may further apply a conjugation on an antidiagonal element to obtain the following formulation:
\[ {}^{(i)}\pi = \delta_{\bar{R}_i}^{-1/2} \cdot \frac{\pi}{\langle x.v-\psi(x)v: x \in \bar{R}_i, v \in \pi \rangle},
\]
where $\bar{R}_i=aR_i^ta^{-1}$. Here $a$ is the matrix with $1$ in the antidiagonal entries and $0$ elsewhere.

Most results will only be stated and proved for the 'right' version, and the 'left' version can be formulated and proved similarly. 




\subsection{Properties of Bernstein-Zelevinsky derivatives} \label{ss prop derivatives}

From the multiplicity-one theorem \cite{AGRS10} (see \cite[Proposition 2.5]{Ch21}, \cite[Lemma 2.3]{CS21}) and a self-dual property (see \cite[Lemma 2.4]{CS21}), we deduce that:

\begin{lemma} \label{lem property derivative} \cite[Proposition 2.5]{Ch21}
Let $\pi \in \mathrm{Irr}(G_n)$. Then $\mathrm{soc}(\pi^{(i)})$ is multiplicity-free. The same holds for $\mathrm{soc}({}^{(i)}\pi)$, $\mathrm{cosoc}(\pi^{(i)})$ and $\mathrm{cosoc}({}^{(i)}\pi)$. 
\end{lemma}

Using the stronger multiplicity-one theorem in \cite{Ch21+}, we have the following statement:

\begin{lemma} \label{lem multiplicity free st repn}
Let $\pi$ be a standard representation of $G_n$. Then $\mathrm{cosoc}(\pi^{(i)})$ is multiplicity-free. The same holds for $\mathrm{cosoc}({}^{(i)}\pi)$. 
\end{lemma}

The proof of Lemma \ref{lem multiplicity free st repn} is similar to the one of \cite[Proposition 2.5]{Ch21} and so we omit the details. 


\begin{lemma} \label{lem dual property} \cite[Lemma 2.4]{CS21}
Let $\pi \in \mathrm{Irr}(G_n)$. Then, for any $i$ with $\pi^{(i)}\neq 0$, $\mathrm{soc}(\pi^{(i)}) \cong \mathrm{cosoc}(\pi^{(i)})$.
\end{lemma}

Let $\mathfrak m=\left\{ \Delta_1, \ldots, \Delta_r\right\} \in \mathrm{Mult}$ and let $\Delta_p=[a_p,b_p]_{\rho_p}$ for each $p$. For any non-negative integer $i$,
\begin{align} \label{eqn bz derivative of multisegment}
  \mathfrak m^{(i)} = \left\{  \Delta_1^{\#_1}+ \ldots +\Delta_r^{\#_r}  : \#_p=0  \mbox{ or } -, \quad \sum_{p: \#_p=- } \mathrm{deg}(\rho_p)=i  \right\}.
\end{align}
The notion ${}^{(i)}\mathfrak m$ is defined similarly by using ${}^{-}\Delta$ and ${}^{0}\Delta$.

\begin{lemma} \label{lem socle cosocle coarse}
Let $\mathfrak m \in \mathrm{Mult}$ and let $\pi =\langle \mathfrak m \rangle$. For any simple quotient or submodule $\tau$ of $\pi^{(i)}$ (resp. ${}^{(i)}\pi$), $\tau \cong \langle \mathfrak n \rangle $ for some $\mathfrak n \in \mathfrak m^{(i)}$ (resp. $\mathfrak n \in {}^{(i)}\mathfrak m$).
\end{lemma}

\begin{proof}
The proof is similar to \cite[Lemma 7.3]{Ch21} and \cite[Proposition 2.3]{Ch20}, and we provide the details for the convenience of the reader. By Lemma \ref{lem dual property}, it suffices to prove the statement for simple submodules. We shall label all segments $\Delta_1, \ldots, \Delta_r$ in $\mathfrak m$ such that 
\begin{align} \label{eqn labelling in m- lemma}
  a(\Delta_i) \not < a(\Delta_j)
\end{align}
for $i<j$ (see notations in Section \ref{ss notations multi}). Then, by the Zelevinsky classification,
\[  \pi  \hookrightarrow \langle \Delta_1 \rangle \times \ldots \times \langle \Delta_r \rangle
\]
and so
\[  \pi^{(i)} \hookrightarrow (\langle \Delta_1 \rangle \times \ldots \times \langle \Delta_r \rangle)^{(i)} .
\]
Now, let $\tau$ be a simple submodule of $\pi^{(i)}$. Hence,
\[  \tau \hookrightarrow (\langle \Delta_1 \rangle \times \ldots \times \langle \Delta_r \rangle)^{(i)} 
\]
and so $\tau$ embeds to one of the layer in the filtration of $(\langle \Delta_1 \rangle \times \ldots \times \langle \Delta_r \rangle)^{(i)}$ from the geometric lemma. Then,
\[  \tau \hookrightarrow \langle \Delta_1^{\#_1} \rangle \times \ldots \times \langle \Delta_r^{\#_r} \rangle ,
\]
where $\#_1, \ldots, \#_r$ precisely satisfy the condition defining (\ref{eqn bz derivative of multisegment}). But using (\ref{eqn labelling in m- lemma}), $\langle \Delta_1^{\#_1} \rangle \times \ldots \times \langle \Delta_r^{\#_r} \rangle $ is the Zelevinsky standard module $\zeta(\left\{ \Delta_1^{\#_1}, \ldots, \Delta_r^{\#_r} \right\})$. In other words, 
\[  \tau \hookrightarrow \zeta(\mathfrak n)
\]
for some $\mathfrak n \in \mathfrak m^{(i)}$. This implies that $\tau \cong \langle \mathfrak n \rangle$, as desired.
\end{proof}

\subsection{Submodules of derivatives from Jacquet functor} \label{ss derivatives jacquet functor}

\begin{lemma} \label{lem quotient of ps to generic}
Let $\omega$ be an indecomposable generic representation of $G_n$ of finite length. Let $\omega'$ be the unique (up to isomorphisms) generic simple composition factor in $\omega$. Let $\mathrm{csupp}(\omega')=\left\{ \rho_1, \ldots, \rho_r\right\}$ for some $\rho_1, \ldots, \rho_r \in \mathrm{Irr}^c$. Assume the labelling satisfies that $\rho_i \not>\rho_j$ for any $i<j$. Then there exists a non-zero map from $\rho_1 \times \ldots \times \rho_r$ to $\omega'$.
\end{lemma}

\begin{proof}
By the Zelevinsky classification, the irreducible generic representation $\omega'$ admits the following surjection:
\[  \rho_1\times \ldots \times \rho_r \twoheadrightarrow \omega'
\]
and 
\[  \mathrm{Hom}_{G_n}(\rho_1 \times \ldots \times \rho_r, \widetilde{\omega}) =0
\]
for any non-generic irreducible representation $\widetilde{\omega}$. By the second adjointness of parabolic induction and Lemma \ref{lem twist the isomorphism}, 
\[  \mathrm{Hom}_{G'}(\rho_r\boxtimes \ldots \boxtimes \rho_1, \widetilde{\omega}_{N'})=0 ,
\]
where $G'=G_{\mathrm{deg}(\rho_r)}\times \ldots \times G_{\mathrm{deg}(\rho_1)}$ and $N'$ is the unipotent radical of $P_{\mathrm{deg}(\rho_r),\ldots, \mathrm{deg}(\rho_1)}$. But, this indeed also implies that $\rho_r \boxtimes \ldots \boxtimes \rho_1$ is not a simple composition factor of $\widetilde{\omega}_{N'}$. Thus, a standard long exact sequence argument implies that
\[  \mathrm{Ext}^k_{G'}(\rho_r \boxtimes \ldots \boxtimes \rho_1, \widetilde{\omega}_{N'}) =0 
\]
for all $k$, and so by the second adjointness and Lemma \ref{lem twist the isomorphism} again, for all $k$,
\begin{align} \label{eqn ext vanishing in generic lemma}
  \mathrm{Ext}^k_{G_n}(\rho_1 \times \ldots \times \rho_r, \widetilde{\omega})=0
\end{align}

If $\widetilde{\omega}$ is a submodule of $\omega$, then the natural map
\[ \mathrm{Hom}_{G_n}(\rho_1\times \ldots \times \rho_r, \widetilde{\omega}) \rightarrow \mathrm{Hom}_{G_n}(\rho_1 \times \ldots \times \rho_r, \omega)
\]
is injective and so $\mathrm{Hom}_{G_n}(\rho_1\times \ldots \times \rho_r, \omega)\neq 0$ as desired.

If $\widetilde{\omega}$ is not a submodule of $\omega$, then we consider a simple submodule $\omega'$ of $\widetilde{\omega}$. Then, we have a short exact sequence:
\[  0 \rightarrow \omega' \rightarrow \omega \rightarrow \omega/\omega' \rightarrow 0 .
\]
Then, a standard long exact sequence argument with the Ext-vanishing (for $k=0,1$) in (\ref{eqn ext vanishing in generic lemma}) gives that
\[  \mathrm{Hom}_{G_n}(\rho_1\times \ldots \times \rho_r, \omega) \cong \mathrm{Hom}_{G_n}(\rho_1 \times \ldots \times \rho_r, \omega/\omega') .
\]
Since $\omega$ is of finite length and $\omega/\omega'$ is still generic, one inductively has that $\mathrm{Hom}_{G_n}(\rho_1\times \ldots \times \rho_r, \omega/\omega')\neq 0$. Hence, we also have $\mathrm{Hom}_{G_n}(\rho_1\times \ldots \times \rho_r, \omega)\neq 0$ as desired.
\end{proof}

The author would like to thank G. Savin for a discussion on the following proposition.

\begin{proposition} \label{prop embedding deri sub}
Let $\pi$ be a representation of $G_n$ of finite length. Let $\tau$ be a simple submodule or quotient of $\pi^{(i)}$. Then there exists $\rho_k$ of $\mathrm{Irr}^c(G_{n_k})$ ($k=1,\ldots, r$) such that $\rho_i \not > \rho_j$ for any $i <j$, and 
\[   \tau \boxtimes \rho_r \boxtimes \rho_{r-1} \boxtimes \ldots \boxtimes \rho_1 \hookrightarrow \pi_{N'},
\]
where $N'=N_{n-n_1-\ldots-n_r, n_r, \ldots, n_1}$.
\end{proposition}

\begin{proof}
By Lemma \ref{lem dual property}, it suffices to show for the submodule statement. Using the Hecke algebra realization \cite{CS19} of Bernstein-Zelevinsky derivatives (see  Section \ref{ss bz derivatives aha}), there is a submodule of $\pi_{N'}$ of the form 
\[  \tau \boxtimes \omega ,
\]
with $\omega$ to be an admissible generic representation of $G_i$. We may and shall further assume that $\omega$ is indecomposable.

The module $\omega$ determines a set of cuspidal representations $\rho_1, \ldots, \rho_r$ such that 
\[  \rho_1+\ldots +\rho_r =\mathrm{csupp}(\omega') 
\]
for any simple composition factor $\omega'$ in $\omega$. We shall relabel $\rho_1, \ldots, \rho_r$ such that $\rho_i \not> \rho_j$ for $i<j$. 

Now we consider the embedding:
\[  \tau \boxtimes \omega \hookrightarrow \pi_{N_i} \]
as $G_{n-i}\times G_i$-modules. By Lemma \ref{lem quotient of ps to generic}, we then have a non-zero map:
\[   \tau \boxtimes (\rho_1\times \ldots \times \rho_r) \hookrightarrow \pi_{N_i} .
\]
By the second adjointness and Lemma \ref{lem twist the isomorphism}, we then have a non-zero map:
\[   \tau \boxtimes \rho_r\boxtimes \ldots \boxtimes \rho_1 \hookrightarrow \pi_{N'} .
\]
Since $\tau\boxtimes \rho_r\boxtimes \ldots \boxtimes \rho_1$ is irreducible, the last non-zero map must be injective.
\end{proof}

We shall also prove a kind of converse of the above statement in Proposition \ref{prop unique as segments}. 

\begin{definition} (c.f. \cite[Theorem 6.1]{Ze80})
A sequence of segments $\Delta_1, \ldots, \Delta_k$ is said to be {\it ascending} or in an {\it ascending order} if for any $i<j$, $\Delta_j \not < \Delta_i$. This is opposite to the ordering which usually defines a standard representation, which means that $\mathrm{St}(\Delta_1)\times \ldots \times \mathrm{St}(\Delta_k)$ is isomorphic to the smooth dual of a standard representation. This also coincides with the one defined in Section \ref{ss main results} when all $\Delta_i$ are in $\mathrm{Seg}_{\rho}$ for a fixed $\rho$. 

\end{definition}

\begin{proposition} \label{prop unique as segments}
Let $\pi \in \mathrm{Irr}(G_n)$. Let $\Delta_1, \ldots, \Delta_k$ be an ascending sequence of segments. Let $n_1, \ldots, n_k$ be the absolute lengths of $\Delta_1, \ldots, \Delta_k$ respectively. Suppose $n_1+\ldots +n_k\leq n$. Let $N$ be the unipotent radical of $P_{n-n_1-\ldots-n_k, n_k, \ldots, n_1}$. Let $n'=n_1+\ldots+n_k$. Then,
\begin{enumerate}
\item  For any $\tau \in \mathrm{Irr}(G_{n-n'})$, 
\[  \mathrm{dim}~ \mathrm{Hom}_{G}( \tau \boxtimes \mathrm{St}(\Delta_k) \boxtimes \ldots \boxtimes \mathrm{St}(\Delta_1), \pi_{N}) \leq 1 ,
\]
where $G=G_{n-n'} \times G_{n_k}\times \ldots \times G_{n_1}$. 
\item For any $\tau \in \mathrm{Irr}(G_{n-n'})$, 
\[  \mathrm{dim}~ \mathrm{Hom}_{G'}(\tau \boxtimes (\mathrm{St}(\Delta_1)\times \ldots \times \mathrm{St}(\Delta_k)), \pi_{N'}) \leq 1,
\]
where $G'=G_{n-n'}\times G_{n'}$ and $N'=N_{n'}$ in $G_n$. 
\item If one of the dimensions above is non-zero, then $\tau$ is a simple submodule of $\pi^{(n')}$ and also a simple quotient of $\pi^{(n')}$. 
\end{enumerate}
\end{proposition}

\begin{proof}
Note that (1) and (2) are equivalent by Frobenius reciprocity. (We remark that the ordering of segments in (1) and (2) is switched, followed by Lemma \ref{lem twist the isomorphism}.) We now prove (2). Suppose 
\[  \mathrm{dim}~\mathrm{Hom}_{G'}(\tau \boxtimes (\mathrm{St}(\Delta_1)\times \ldots \times \mathrm{St}(\Delta_k)), \pi_{N'}) \geq 2
\]
and we shall derive a contradiction. We begin with a claim. \\
\ \\

\noindent
{\it Claim}: Let $f, f' \in \mathrm{Hom}_{G'}(\tau \boxtimes (\mathrm{St}(\Delta_1)\times \ldots \times \mathrm{St}(\Delta_k)), \pi_{N'})$. Suppose $f,f'$ are non-zero and $f \neq cf'$ for some non-zero scalar $c$. Then $\mathrm{im}~f\neq \mathrm{im}~f'$. \\

\noindent
{\it Proof of Claim:} Suppose $\mathrm{im}~f=\mathrm{im}~f'$. Then there is a quotient $\kappa$ of $\mathrm{St}(\Delta_1)\times \ldots \times \mathrm{St}(\Delta_k)$ such that 
\[   \tau \boxtimes \kappa \cong \mathrm{im}~f =\mathrm{im}~f'\hookrightarrow \pi_{N'} 
\]
Thus, it suffices to show $\mathrm{End}_{G'}(\tau \boxtimes \kappa)\cong \mathbb C$. 

To this end,  \[\mathrm{End}_{G'}(\tau \boxtimes \kappa)=\mathrm{End}_{G_{n-n'}}(\tau)\otimes \mathrm{End}_{G_{n'}}(\kappa).\]
Now $\mathrm{End}_{G_{n-n'}}(\tau)$ is isomorphic to $\mathbb C$ since $\tau$ is irreducible and $\mathrm{End}_{G_{n'}}(\kappa)$ is isomorphic to $\mathbb C$ by the cosocle-irreducible property of $\kappa$ (see \cite[Proposition 3.2]{JS83} and \cite[Proposition 2.3]{Ch21}). Here the cosocle irreducible property means that $\kappa$ has unique simple quotient and such simple quotient appears with multiplicity one in the Jordan-H\"older factors of $\kappa$. Hence, we have:
\[  \mathrm{End}_{G'}(\tau \boxtimes \kappa) \cong \mathbb C .
\]

\ \\

Now we return to the proof. Let $f$ and $f'$ be two maps in the claim. Since both $\mathrm{im}~f$ and $\mathrm{im}~f'$ are quotients of
\[  \tau \boxtimes (\mathrm{St}(\Delta_1)\times \ldots \times \mathrm{St}(\Delta_k)) ,
\]
the discussion in the claim shows that $\mathrm{im}~f$ and $\mathrm{im}~f'$ are both cosocle-irreducible. This implies that the situation that $\mathrm{im}~f\subset \mathrm{im}~f'$ or $\mathrm{im}~f' \subset \mathrm{im}~f$ can happen only if $\mathrm{im}~f=\mathrm{im}~f'$. However, the claim says that $\mathrm{im}~f=\mathrm{im}~f'$ is impossible. Hence, we cannot have $\mathrm{im}~f\subset \mathrm{im}~f'$ and $\mathrm{im}~f'\subset \mathrm{im}~f$ as well. In other words, $(\mathrm{im}~f) \cap (\mathrm{im}~f')$ is a proper submodule of both $\mathrm{im}~f$ and $\mathrm{im}~f'$


Indeed, the results of \cite[Proposition 3.2]{JS83} and \cite[Proposition 2.3]{Ch21} say that the unique simple quotient of $\mathrm{St}(\Delta_1)\times \ldots \times \mathrm{St}(\Delta_k)$ is generic; and the cosocle irreducible property then further implies that any simple composition factor in $(\mathrm{im}~f) \cap (\mathrm{im}~f')$ cannot contain a simple composition factor of the form $\tau_1\boxtimes \tau_2$ with $\tau_2$ to be generic. Now taking the (exact) $(U_{n'},\psi_{n'})$-twisted Jacquet functor on the second factor of $G_{n-n'}\times G_{n'}$-representation for the following short exact sequence:
\[ 0 \rightarrow  (\mathrm{im}~f)\cap (\mathrm{im}~f') \rightarrow \mathrm{im}~f \oplus \mathrm{im}~f' \rightarrow \mathrm{im}~f+\mathrm{im}~f' \rightarrow 0 , \]
we have 
\[ \tau \oplus \tau \hookrightarrow  (\mathrm{im}~f+\mathrm{im}~f')_{U_{n'},\psi_{n'}} .  \]

By the exactness of  $(U_{n'},\psi_{n'})$-twisted Jacquet functor, we then have:
\[  \tau \oplus \tau \hookrightarrow  (\mathrm{im}~f+\mathrm{im}~f')_{U_{n'},\psi_{n'}}  \hookrightarrow \pi^{(n')} .\]
This contradicts Lemma \ref{lem property derivative}. Hence, this implies (2).

We now prove (3). As mentioned before, the two dimensions are the same by Frobenius reciprocity. Thus we can only consider the dimension in (2) is non-zero. Then, a quotient of
\[  \tau \boxtimes (\mathrm{St}(\Delta_1)\times \ldots \times \mathrm{St}(\Delta_k))
\]
embeds to $\pi_{N_{n'}}$. Since any quotient is isomorphic to 
\[  \tau \boxtimes \kappa
\]
for some quotient $\kappa$ of $\mathrm{St}(\Delta_1)\times \ldots \times \mathrm{St}(\Delta_k)$, we have that
\begin{align*} \label{eqn embedding from the given map}
   \tau \boxtimes \kappa \hookrightarrow \pi_{N_{n'}}
\end{align*}
for some quotient $\kappa$ of $\mathrm{St}(\Delta_1)\times \ldots \times \mathrm{St}(\Delta_k)$. Thus, 
\[   \tau \boxtimes \kappa_{U_{n'}, \psi_{n'}}  \hookrightarrow (\pi_{N_{n'}})_{U_{n'}, \psi_{n'}} .
\]
as $G_{n-n'}$-representations. Here the subscript ${}_{U_{n'}, \psi_{n'}}$ means to take the $(U_{n'}, \psi'_{n'})$-twisted Jacquet functor on the second factor for $G_{n-n'}\times G_{n'}$. 

By the transitivity of (twisted) Jacquet functors, last space is just $\pi^{(n')}$. Since $\kappa$ has a generic simple subquotient with multiplicity one, $\kappa_{U_{n'}, \psi_{n'}}$ is $1$-dimensional. Hence, we have $\tau \hookrightarrow \pi^{(n')}$, which proves the submodule part. The quotient part of (3) then follows from Lemma \ref{lem dual property}.
\end{proof}



\subsection{Reduction to $\mathrm{Irr}_{\rho}$ case} \label{ss reduction to cuspidal case}

Let $\pi \in \mathrm{Irr}$. By \cite[Proposition 8.5]{Ze80}, there exist cuspidal representations $\rho_1, \ldots, \rho_r \in \mathrm{Irr}$ such that 
\begin{itemize}
\item for $j\neq k$, $\rho_j \not\cong \nu^c \rho_k$ for any $c \in \mathbb Z$;
\item $\pi \cong \pi_1 \times \ldots \times \pi_r$ for some $\pi_i \in \mathrm{Irr}_{\rho_i}$.
\end{itemize}
 For any non-negative integer $i$, the geometric lemma gives that
\begin{align} \label{eqn separate derivative by cuspidal}
  \pi^{(i)}\cong \bigoplus_{i_1+\ldots +i_r=i} \pi_1^{(i_1)}\times \ldots \times \pi_r^{(i_r)} ,
\end{align}
where the direct sum is guaranteed by the vanishing of Ext-groups through a comparison of cuspidal supports. 

\begin{proposition} \label{prop reduction to irr rho}
We use the above notations. Suppose $\tau$ is a simple quotient of $\pi^{(i)}$. Then there exist non-negative integers $i_1, \ldots,i_r$ with $i_1+\ldots +i_r=i$ such that $\tau \cong \tau_1 \times \ldots \times \tau_r$ for some simple quotients $\tau_j$ of $\pi_j^{(i_j)}$. 
\end{proposition}

\begin{proof}
By (\ref{eqn separate derivative by cuspidal}), $\tau$ is a simple quotient of $\pi_1^{(i_1)}\times \ldots \times \pi_r^{(i_r)}$ for some integers $i_1+\ldots +i_r=i$. Now, applying Frobenius reciprocity, we have a non-zero map from $\pi_1^{(i_1)}\boxtimes \ldots \boxtimes \pi_r^{(i_r)}$ to $\tau_{N_{i_1, \ldots, i_r}^-}$, where $N_{i_1, \ldots, i_r}^-$ is the unipotent radical of the parabolic subgroup opposite to $P_{i_1, \ldots, i_r}$. Thus we have representations $\tau_j \in \mathrm{Irr}_{\rho_j}$ ($j=1, \ldots, r$) such that there is a non-zero map from $\tau_1\boxtimes \ldots \boxtimes \tau_r$ to $\tau_{N_{i_1, \ldots, i_r}^-}$ and so a non-zero map from $\tau_1 \times \ldots \times \tau_r$ onto $\tau$. Since $\tau_1 \times \ldots \times \tau_r$ is irreducible (by \cite[Proposition 8.5]{Ze80}), we have $\tau\cong \tau_1 \times \ldots \times \tau_r$. 

Now, we have:
\[  \pi_1^{(i_1)}\times \ldots \times \pi_r^{(i_r)} \twoheadrightarrow \tau_1 \times \ldots \times \tau_r  .\]
Then, we apply Frobenius reciprocity to obtain:
\begin{align} \label{eqn surjection separate cuspidal}
  (\pi_1^{(i_1)} \times \ldots \times \pi_r^{(i_r)})_N \twoheadrightarrow \tau_1 \boxtimes \ldots \boxtimes \tau_r,
\end{align}
where $N=N_{n_1, \ldots, n_r}$ with each $n_k=\mathrm{deg}(\pi_k)-i_k$. Now, by using the cuspidal support condition in the first bullet before this proposition, the only layer in the geometric lemma contributing (\ref{eqn surjection separate cuspidal}) takes the form $\pi_1^{(i_1)}\boxtimes \ldots \boxtimes \pi_r^{(i_r)}$. Hence, we have:
 \[    \pi_1^{(i_1)}\boxtimes \ldots \boxtimes \pi_r^{(i_r)} \twoheadrightarrow \tau_1 \boxtimes \ldots \boxtimes \tau_r .
\]
This implies, by K\"unneth formula, that $\tau_j$ is a simple quotient of $ \pi_j^{(i_j)}$ as desired.
\end{proof}

\section{Maximal multisegment at a right point} \label{s multi at a point}


\subsection{Socle and cosocle of some parabolically induced modules} \label{ss socle and cosocle}

For a representation $\pi$ of finite length in $\mathrm{Alg}(G_n)$, we denote by $\mathrm{soc}(\pi)$ and $\mathrm{cosoc}(\pi)$ the socle (i.e. the maximal semisimple subrepresentation) and cosocle (i.e. the maximal semisimple quotient) of $\pi$ respectively. 

We now discuss some extension of the results in Section \ref{ss rho derivatives}. We refer to \cite{LM14, LM16} for a definition of a ladder representation. The particular example of ladder representations, which we shall need, is that $\mathrm{St}(\left\{ \Delta_1, \Delta_2 \right\})$ for two linked segments $\Delta_1, \Delta_2$. 

\begin{lemma} \label{lem socle cosocle ladder} 
Let $\pi \in \mathrm{Irr}_{\rho}(G_n)$ be a ladder representation or a generic representation. Let $\tau_1 \in \mathrm{Irr}_{\rho}(G_{k})$ and let $\tau_2 \in \mathrm{Irr}_{\rho}(G_{k})$. Then 
\begin{itemize}
\item[(1)] $\mathrm{soc}(\pi \times \tau_i)$ and $\mathrm{cosoc}(\pi \times \tau_i)$ are irreducible ($i=1,2$);
\item[(2)] $\mathrm{soc}(\pi \times \tau_1) \cong \mathrm{soc}(\pi \times \tau_2)$ if and only if $\tau_1 \cong \tau_2$;
\item[(3)] $\mathrm{cosoc}(\pi \times \tau_1) \cong \mathrm{cosoc}(\pi \times \tau_2)$ if and only if $\tau_1 \cong \tau_2$.
\end{itemize}
The above statements also hold if one replaces $\pi \times \tau_1$ and $\pi \times \tau_2$ with $\tau_1 \times \pi$ and $\tau_2 \times \pi$ respectively.
\end{lemma}

\begin{proof}
The socle part of (1) is \cite[Corollary 4.10, Proposition 6.15]{LM16}. (2) is \cite[Corollary 6.8]{LM16} and \cite[Theorem 4.1D]{LM22}. The cosocle statements follow from the corresponding socle ones and the fact that the smooth dual functor sends a ladder (resp. generic) representation to a ladder (resp. generic) representation. The remaining assertions that switch from $\pi \times \tau_i$ to $\tau_i \times \pi$ follow from applying $\theta$.
\end{proof}

\begin{remark} \label{rmk unique derivative}
By Frobenius reciprocity, Lemma \ref{lem socle cosocle ladder} (2) (resp. Lemma \ref{lem socle cosocle ladder} (3)) implies that for any ladder or generic representation $\tau$ of $G_r$ and any $\pi \in \mathrm{Irr}_{\rho}(G_n)$, there exists at most one $\omega \in \mathrm{Irr}_{\rho}(G_{n-r})$ such that there exists a surjection $\pi_{N_r} \twoheadrightarrow \omega \boxtimes \tau$ (resp. $\pi_{N_{n-r}} \twoheadrightarrow \tau \boxtimes \omega$). By Proposition \ref{prop sub quo Jacquet}, if such $\omega$ exists, one also has an injection $\omega \boxtimes \tau \hookrightarrow \pi_{N_r}$ (resp. $\tau \boxtimes \omega \hookrightarrow \pi_{N_{n-r}}$).
\end{remark}

\subsection{$\varepsilon_{\Delta}$-invariant} \label{ss epsilon invariant}

Let $\Delta \in \mathrm{Seg}_{\rho}$ and let $i=l_{abs}(\Delta)$. Let $n \geq i$ and let $\pi \in \mathrm{Irr}_{\rho}(G_n)$. Recall that $D_{\Delta}(\pi)$ is defined in Section \ref{ss notion derivative}. 
Let $  \varepsilon_{\Delta}(\pi)$ be the maximal integer $k$ such that $D_{\Delta}^k(\pi):= \overbrace{D_{\Delta} \circ \ldots \circ D_{\Delta}}^{k \mbox{ times }}(\pi) \neq 0$. When $\Delta= [a]_{\rho}$, $\varepsilon_{\Delta}$ coincides with $\varepsilon_a$ defined in Notation \ref{notn derivatives epsilon}. 


\subsection{First property of $D_{\Delta}$ from Frobenius reciprocity}

We record the following first property of the derivative $D_{\Delta}$ by Frobenius reciprocity, in the same spirit of Remark \ref{rmk unique derivative}:

\begin{lemma} \label{lem reformulate frobenius}
Let $\pi \in \mathrm{Irr}_{\rho}$ and let $\Delta \in \mathrm{Seg}_{\rho}$. Suppose $D_{\Delta}(\pi)\neq 0$. Then 
\[  \pi \hookrightarrow D_{\Delta}(\pi) \times \mathrm{St}(\Delta) ,
\]
and $\pi_{N_{i}} \twoheadrightarrow D_{\Delta}(\pi)\boxtimes \mathrm{St}(\Delta)$, where $i=l_{abs}(\Delta)$.
\end{lemma}



\subsection{Commutativity for unlinked segments}

\begin{lemma} \label{lem comm derivative 1}
Let $\pi \in \mathrm{Irr}_{\rho}$. Let $\Delta_1$ and $\Delta_2$ be unlinked segments. Then 
\[ D_{\Delta_1}\circ D_{\Delta_2}(\pi) \cong D_{\Delta_2}\circ D_{\Delta_1}(\pi). \]
\end{lemma}

\begin{proof}
By definitions of $D_{\Delta_1}$ and $D_{\Delta_2}$, $\pi$ is a submodule of 
\[  D_{\Delta_2}\circ D_{\Delta_1}(\pi) \times (\mathrm{St}(\Delta_2)\times \mathrm{St}(\Delta_1)) 
\]
Similarly, $\pi$ is a submodule of 
\[  D_{\Delta_1}\circ D_{\Delta_2}(\pi) \times (\mathrm{St}(\Delta_1)\times \mathrm{St}(\Delta_2)) .
\]

Since $\Delta_1$ and $\Delta_2$ are unlinked, $\mathrm{St}(\Delta_1)\times \mathrm{St}(\Delta_2)$ is a generic representation and so $\mathrm{St}(\Delta_1)\times \mathrm{St}(\Delta_2) \cong \mathrm{St}(\Delta_2)\times \mathrm{St}(\Delta_1)$ by (\ref{eqn product commut st}). Now, combining with Lemma \ref{lem socle cosocle ladder}(2), one has $D_{\Delta_2}\circ D_{\Delta_1}(\pi) \cong D_{\Delta_1}\circ D_{\Delta_2}(\pi)$.
\end{proof}

\subsection{Derivatives and Jacquet functors}

We first prove some useful results on how to pass some information of derivatives under Jacquet functors.

\begin{lemma} \label{lem zero derivative reduce}
Let $\pi \in \mathrm{Irr}_{\rho}(G_n)$ and let $\Delta=[a,c]_{\rho}$ be a segment. Suppose $D_{[a',c]_{\rho}}(\pi)= 0$ for any $c\geq a'\geq a$. Then, for any $1 \leq i \leq n$, and for any simple composition factor $\tau_1 \boxtimes \tau_2$ of $\pi_{N_i}$ ($\tau_1 \in \mathrm{Irr}_{\rho}(G_{n-i})$ and $\tau_2 \in \mathrm{Irr}_{\rho}(G_i)$), $D_{[a',c]_{\rho}}(\tau_2)=0$ for any $c \geq a' \geq a$.
\end{lemma}

\begin{proof}
Suppose not. Then $\pi_{N_i}$ has a simple composition factor of the form $\tau_1 \boxtimes \tau_2$ for some $\tau_1 \in \mathrm{Irr}_{\rho}(G_{n-i})$ and some $\tau_2 \in \mathrm{Irr}_{\rho}(G_i)$ with $D_{[a',c]_{\rho}}(\tau_2)\neq 0$ for some $a \leq a' \leq c$. This implies that $(\tau_2)_{N_j}$ has the simple composition factor isomorphic to
\[  D_{[a',c]_{\rho}}(\tau_2)\boxtimes \mathrm{St}([a',c]_{\rho}) .
\]
Thus, $\pi_{N_{n-i,i-j,j}}$ has a simple composition factor of the form $\tau \boxtimes \omega \boxtimes \mathrm{St}([a',c]_{\rho})$ ($\tau, \omega \in \mathrm{Irr}_{\rho}$). Then $\pi_{N_{n-i,i-j,j}}$ has a simple quotient of the form $\tau' \boxtimes \omega' \boxtimes \kappa$ for some $\tau', \omega' \in \mathrm{Irr}_{\rho}$ and some $\kappa \in \mathrm{Irr}_{\rho}$ satisfying 
\begin{align}
\mathrm{csupp}(\kappa) = [a',c]_{\rho}
\end{align}
By the Langlands classification, we then must have
\[  \kappa \hookrightarrow  \kappa'\times  \mathrm{St}([a'',c]_{\rho})
\]
for some $\kappa' \in \mathrm{Irr}_{\rho}$ and $c \geq a'' \geq a'$. 

From Frobenius reciprocity,  
\[   \pi  \hookrightarrow \tau' \times \omega' \times \kappa
\]
and thus
\[  \pi \hookrightarrow \tau' \times \omega' \times \kappa' \times \mathrm{St}([a'',c]_{\rho}) .
\]
Thus, by Frobenius reciprocity, $D_{[a'',c]_{\rho}}(\pi)\neq 0$ and this gives a contradiction.
\end{proof}




\begin{proposition} \label{prop non composition factor at right pt} (c.f. \cite[Proposition 7.3]{LM22})
Let $\pi \in \mathrm{Irr}_{\rho}(G_n)$. Let $c \in \mathbb Z$. Suppose $D_{[a,c]_{\rho}}(\pi)=0$ for any integer $a \leq c$. Then, for any $1\leq i \leq n$, and for any simple composition factor $\tau_1\boxtimes \tau_2$ in $\pi_{N_i}$ ($\tau_1 \in \mathrm{Irr}_{\rho}(G_{n-i})$ and $\tau_2 \in \mathrm{Irr}_{\rho}(G_i)$), if any $\rho' \in \mathrm{csupp}(\tau_2)$ satisfies that $\rho' \leq \nu^c\rho$, then $\nu^c\rho \notin \mathrm{csupp}(\tau_2)$.
\end{proposition}

\begin{proof}
Suppose $\nu^c\rho \in \mathrm{csupp}(\tau_2)$. By the Langlands classification, 
\[  \tau_2 \hookrightarrow \tau_2' \times \mathrm{St}([a',c]_{\rho}) 
\]
for some $\tau_2' \in \mathrm{Irr}$ and some $a' \leq c$. In other words, $D_{[a',c]_{\rho}}(\tau_2) \neq 0$. However, this then contradicts to Lemma \ref{lem zero derivative reduce}.
\end{proof}

\subsection{ Maximal multisegment at a right point} \label{ss right max multisegment}

Recall $\mathrm{Mult}_{\rho,c}^b$ is defined in Section \ref{ss notations multi}. We shall call the multisegments in $\mathrm{Mult}_{\rho,c}^b$ are the {\it multisegments at the right point $\nu^c\rho$}.

For a multisegment $\mathfrak m$ in $\mathrm{Mult}^b_{\rho,c}$ and for any $\pi \in \mathrm{Irr}_{\rho}(G_n)$, recall that $D_{\mathfrak m}(\pi)$ is defined in (\ref{eqn derivatives of multisegment}). We shall adapt the convention that: $D_{\emptyset}(\pi)=\pi$.

\begin{lemma} \label{lem uniuq max b element}
Let $\pi \in \mathrm{Irr}_{\rho}$ and let $c \in \mathbb Z$. Then there exists a unique $\leq^b_c$-maximal multisegment $\mathfrak m$ in $\mathrm{Mult}_{\rho,c}^b$ such that $D_{\mathfrak m}(\pi)\neq 0$. 
\end{lemma}

\begin{proof}
Let $\mathfrak m$ and $\mathfrak m'$ be $\leq^b_c$-maximal multisegments such that $D_{\mathfrak m}(\pi)\neq 0$ and $D_{\mathfrak m'}(\pi)\neq 0$. Write 
\[  \mathfrak m=\left\{ [a_1,c]_{\rho}, \ldots, [a_r, c]_{\rho} \right\} , \quad \mathfrak m'=\left\{ [a_1',c]_{\rho}, \ldots, [a_s',c]_{\rho}\right\}
 \]
with $a_1 \geq a_2\geq \ldots \geq a_r$ and $a_1'\geq a_2'\geq \ldots \geq a_s'$.

We first show that $a_1 \neq a_1'$ cannot happen. By reindexing if necessary, we may assume $a_1>a_1'$. In such case, we consider the embedding:
\[   \pi \hookrightarrow D_{\mathfrak m'}(\pi) \times \mathrm{St}(\mathfrak m') .
\]
Combining with 
\[        D_{[a_1,c]_{\rho}}(\pi)\boxtimes \mathrm{St}([a_1,c]_{\rho})\hookrightarrow     \pi_{N_i}, \quad \mbox{where $i=l_{abs}([a_1,c]_{\rho})$,}
\]
we have
\[   D_{[a_1,c]_{\rho}}(\pi)\boxtimes \mathrm{St}([a_1,c]_{\rho}) \hookrightarrow (D_{\mathfrak m'}(\pi) \times \mathrm{St}(\mathfrak m'))_{N_i} .
\]
By the geometric lemma,  $D_{[a_1,c]_{\rho}}(\pi)\boxtimes \mathrm{St}([a_1,c]_{\rho})$ is a simple submodule $\tau_1 \times \omega_1 \boxtimes \tau_2 \times \omega_2$ with
\begin{enumerate}
\item[(a)] $\tau_1 \boxtimes \tau_2$ being a composition factor of $(D_{\mathfrak m'}(\pi))_{N_{i_1}}$;
\item[(b)] $\omega_1 \boxtimes \omega_2$ being a composition factor of $\mathrm{St}(\mathfrak m')_{N_{i_2}}$;
\item[(c)] $i_1+i_2=i$.
\end{enumerate}

We now separate into two cases:
\begin{enumerate}
\item[(1)] $i_1=0$. In such case, we have: $\mathrm{csupp}(\omega_2)=[a_1,c]_{\rho}$ and $\mathrm{csupp}(\omega_2) \subset \mathrm{csupp}(\mathfrak m')$. However, $\nu^{a_1}\rho \in [a_1,c]_{\rho}$, but $\nu^{a_1}\rho \notin \mathrm{csupp}(\mathfrak m')$ by our assumption. This gives a contradiction.
\item[(2)] $i_1>0$. In such case, we consider
\[  \mathrm{St}([a_1,c]_{\rho})\hookrightarrow \tau_2 \times \omega_2
\] 
and so by Frobenius reciprocity, we have a non-zero map:
\[  \mathrm{St}(\mathfrak m')_{N_{i_2}} \rightarrow \tau_2 \boxtimes \omega_2.
\]
Since $i_1>0$, $\nu^c\rho \in \mathrm{csupp}(\tau_2)$. This contradicts Proposition \ref{prop non composition factor at right pt} since $D_{[a,c]_{\rho}}(D_{\mathfrak m}(\pi))=0$ for any $a \leq c$.
\end{enumerate}

Hence, we now must have $a_1=a_1'$. Then $\mathfrak m-[a_1,c]_{\rho}$ and $\mathfrak m'-[a_1',c]_{\rho}$ are $\leq^b_c$-maximal multisegments such that $D_{[a_1,c]_{\rho}}(\pi)$ and so one can then proceed inductively.
\end{proof}

\begin{definition} \label{def max R multiseg}
For $c \in \mathbb Z$, denote the unique $\leq^b_c$-maximal multisegment $\mathfrak n$ in $\mathrm{Mult}_{\rho,c}^b$ such that $D_{\mathfrak n}(\pi)\neq 0$ by $\mathfrak{mxpt}^{b}(\pi, c)$. 
\end{definition}

\subsection{$\mathfrak{mxpt}^b$ in terms of $\varepsilon$}

\begin{lemma} \label{lem irr submodule under reduced}
 Let $\pi \in \mathrm{Irr}_{\rho}$ and let $c \in \mathbb Z$. Let $\mathfrak m \in \mathrm{Mult}_{\rho,c}^b$. Let $\omega$ be the irreducible submodule of $\pi \times \mathrm{St}(\mathfrak m)$. If  $\mathfrak{mxpt}^b(\pi, c)=\emptyset$, then
\[   \mathfrak{mxpt}^b(\omega, c)=\mathfrak m .
\]
\end{lemma}

\begin{proof}
We have:
\[   \omega \hookrightarrow \pi \times \mathrm{St}(\mathfrak m) .
\]
This implies that $\mathfrak m \leq^b_c \mathfrak{mxpt}^b(\omega, c)$ and so 
\begin{align} \label{eqn bound on product max b}
 |\mathfrak m| \leq |\mathfrak{mxpt}^b(\omega, c)|.
\end{align}

Let $i=l_{abs}(\mathfrak{mxpt}^b(\omega, c))$. Hence, we have:
\[  D_{\mathfrak{mxpt}^b(\omega,c)}(\pi)\boxtimes \mathrm{St}(\mathfrak{mxpt}^b(\omega,c)) \hookrightarrow (\pi \times \mathrm{St}(\mathfrak m))_{N_i}.
\]
Then there exists non-negative integers $i_1,i_2$ and $\tau_1, \tau_2, \omega_1, \omega_2 \in \mathrm{Irr}_{\rho}$ such that
\begin{enumerate}
\item $i_1+i_2=i$
\item $\tau_1 \boxtimes \tau_2$ is a simple composition factor of $\pi_{N_{i_1}}$ and $\omega_1\boxtimes \omega_2$ is a simple composition factor of $\mathrm{St}(\mathfrak m)_{N_{i_2}}$;
\item $D_{\mathfrak{mxpt}^b(\omega, c)}(\pi) \hookrightarrow \tau_1\times \omega_1$ and $\mathrm{St}(\mathfrak{mxpt}^b(\pi,c))\hookrightarrow \tau_2\times \omega_2$.
\end{enumerate}

If $i_1>0$, one obtains a contradiction by the same argument in the Case (2) of the proof of Lemma \ref{lem uniuq max b element}. Thus $i_1=0$. If $i_2<l_{abs}(\mathfrak m)$, then the multiplicity of $\nu^c\rho$ in $\mathrm{csupp}(\omega_2)$ is less than $|\mathfrak m|$. This contradicts to (\ref{eqn bound on product max b}). Thus $i_2=i$, and so we must have that $\omega_2=\mathrm{St}(\mathfrak m)$ and $\mathrm{St}(\mathfrak{mxpt}^b(\omega,c))=\mathrm{St}(\mathfrak m)$. This shows the lemma.
\end{proof}


\begin{corollary} \label{cor mxpt b product submod}
Let $\pi \in \mathrm{Irr}_{\rho}$ and let $c \in \mathbb Z$. Let $\mathfrak m \in \mathrm{Mult}_{\rho,c}^b$. Let $\omega$ be the irreducible submodule of $\pi \times \mathrm{St}(\mathfrak m)$. Then 
\[  \mathfrak{mxpt}^b(\omega, c) =\mathfrak{mxpt}^b(\pi,c) +\mathfrak m.
\]
\end{corollary}

\begin{proof}
We have
\[  \omega \hookrightarrow \pi \times \mathrm{St}(\mathfrak m) \hookrightarrow D_{\mathfrak{mxpt}^b(\pi,c)} \times \mathrm{St}(\mathfrak{mxpt}^b(\pi,c)) \times \mathrm{St}(\mathfrak m).
\]
Since $\mathrm{St}(\mathfrak{mxpt}^b(\pi,c)) \times \mathrm{St}(\mathfrak m)=\mathrm{St}(\mathfrak{mxpt}^b(\pi,c)+\mathfrak m)$, the corollary now follows from Lemma \ref{lem irr submodule under reduced}.
\end{proof}

\begin{corollary} \label{cor mxpt b epsilon}
Let $\pi \in \mathrm{Irr}_{\rho}$ and let $c \in \mathbb Z$. Then, for any integer $a \leq c$, $[a,c]_{\rho}$ appears in $\mathfrak{mxpt}^b(\pi,c)$ with multiplicity $\varepsilon_{[a,c]_{\rho}}(\pi)$. 
\end{corollary}

\begin{proof}
Let $\mathfrak m=\mathfrak{mxpt}^b(\pi,c)$. For any integer $a$, let $k_a$ be the multiplicity of $[a,c]_{\rho}$ in $\mathfrak m$. We have:
\[  \pi \hookrightarrow D_{\mathfrak m}(\pi)\times \mathrm{St}(\mathfrak m) \cong D_{\mathfrak m}(\pi)\times \mathrm{St}(\mathfrak m-k_a\cdot [a,c]_{\rho}) \times \mathrm{St}([a,c]_{\rho})^{\times k_a} .
\]
Hence, by Frobenius reciprocity, $(D_{[a,c]_{\rho}})^{k_a}(\pi)\neq 0$ and so $\varepsilon_{[a,c]_{\rho}}(\pi) \geq k_a$.

On the other hand, let $k_a'=\varepsilon_{[a,c]_{\rho}}(\pi)$. We also have the embedding:
\[  \pi \hookrightarrow (D_{[a,c]_{\rho}})^{k_a'}(\pi) \times \mathrm{St}( k_a' \cdot [a,c]_{\rho}) .
\]
Let $\mathfrak m'=\mathfrak{mxpt}^b(D_{[a,c]_{\rho}}^{k_a'}(\pi), c)$. We have:
\[  D_{[a,c]_{\rho}}^{k_a'}(\pi) \hookrightarrow D_{\mathfrak m'}(D_{[a,c]_{\rho}}^{k_a'}(\pi)) \times \mathrm{St}(\mathfrak m') 
\]
and so combining the above embeddings, we have:
\[  \pi  \hookrightarrow D_{\mathfrak m'}(D_{[a,c]_{\rho}}^{k_a'}(\pi)) \times \mathrm{St}(\mathfrak m') \times  \mathrm{St}( k_a' \cdot [a,c]_{\rho}) .
\]
Since $\mathfrak{mxpt}^b(D_{[a,c]_{\rho}}^{k_a'}(\pi), c)=\emptyset$, Lemma \ref{lem irr submodule under reduced} implies that 
\[  \mathfrak{mxpt}^b(\pi,c)= \mathfrak m'+k_a'\cdot [a,c]_{\rho} .
\]
Thus, $k_a \geq k_a'$.

Combining the above two paragraphs, we have $k_a=k_a'$.
\end{proof}

One may ask how to use $\mathfrak{mxpt}^b(\pi,c)$ to characterize multisegments $\mathfrak m$ in $\mathrm{Mult}^b_{\rho,c}$ such that $D_{\mathfrak m}(\pi)\neq 0$. We will postpone to do this in Corollary \ref{cor multisegment at right point}.

\subsection{$\mathfrak{mxpt}^b$ under parabolic induction}

We first have the following generalization of Lemma \ref{lem epsilon for rho der}:

\begin{lemma} \label{lem cuspidal support on irr pair}
 Let $\pi \in \mathrm{Irr}_{\rho}$. Let $c \in \mathbb Z$. Let $\mathfrak m=\mathfrak{mxpt}^b(\pi, c)$. Let $l=l_{abs}(\mathfrak m)$. Then, the followings hold:
\begin{enumerate}
\item $D_{\mathfrak m}(\pi)\boxtimes \mathrm{St}(\mathfrak m)$ is a direct summand in $\pi_{N_l}$;
\item for any other simple composition factor $\tau_1\boxtimes \tau_2$ in $\pi_{N_l}$ satisfying that $\rho' \leq \nu^c\rho$ for all $\rho'\in \mathrm{csupp}(\tau_2)$, the multiplicity of $\nu^c \rho$ in $\mathrm{csupp}(\tau_2)$ is strictly less than the number of segments in $\mathrm{mult}^b(\pi, c)$;
\item for $l'<l$ and for any simple composition factor $\tau_1 \boxtimes \tau_2$ in $\pi_{N_{l'}}$ satisfying that $\rho' \leq \nu^c\rho$ for all $\rho'\in \mathrm{csupp}(\tau_2)$, the multiplicity of $\nu^c\rho$ in $\mathrm{csupp}(\tau_2)$ is strictly less than the number of segments in $\mathrm{mult}^b(\pi,c)$;
\item for $l'\geq l$ and for any simple composition factor $\tau_1 \boxtimes \tau_2$ in $\pi_{N_{l'}}$ satisfying that $\rho' \leq \nu^c\rho$ for all $\rho'\in \mathrm{csupp}(\tau_2)$, the multiplicity of $\nu^c\rho$ in $\mathrm{csupp}(\tau_2)$ is at most the number of segments in $\mathrm{mult}^b(\pi,c)$. 
\end{enumerate}
\end{lemma}

\begin{proof}
The first assertion is proved in \cite[Proposition 11.1]{Ch22+}. The other assertions can be proved by considering the embedding 
\[   \pi_{N_l} \hookrightarrow (D_{\mathfrak m}(\pi) \times \mathrm{St}(\mathfrak m))_{N_l}
\]
and analyzing the geometric lemma on $(D_{\mathfrak m}(\pi) \times \mathrm{St}(\mathfrak m))_{N_l}$. The analysis is similar to the proofs of Lemmas \ref{lem uniuq max b element} and \ref{lem irr submodule under reduced} that utilize Lemma \ref{lem zero derivative reduce}.
\end{proof}

\begin{proposition} \label{prop delta reduced preserve under parabolic}
Let $\pi_1, \ldots, \pi_r \in \mathrm{Irr}_{\rho}$ and let $c \in \mathbb Z$. Fix another $d \in \mathbb Z$ with $d \leq c$. Suppose $D_{[e,c]_{\rho}}(\pi_k)=0$ for all $k$ and all $d\leq e \leq c$. Then, for any simple composition factor $\pi$ of $\pi_1 \times \ldots \times \pi_r$, 
\[  D_{[e,c]_{\rho}}(\pi) = 0  \]
for all $d \leq e \leq c$.
\end{proposition}

\begin{proof}
Let $\pi$ be a simple composition factor in $\pi_1 \times \ldots \times \pi_r$. Suppose $D_{[e,c]_{\rho}}(\pi)\neq 0$ for some $d \leq e \leq c$. Let $i=l_{abs}([e,c]_{\rho})>0$. Then $D_{[e,c]_{\rho}}(\pi)\boxtimes \mathrm{St}([e,c]_{\rho})$ is a simple composition factor of $\pi_{N_i}$ and so is also a simple composition factor of 
\[  (\pi_1\times \ldots \times \pi_r)_{N_i} .\]
By the geometric lemma, we can find non-negative integers $i_1, \ldots, i_r$ and $\tau_1, \ldots, \tau_r, \tau_1', \ldots, \tau_r'$ such that 
\begin{enumerate}
\item[(1)] $i_1+\ldots+i_r=i$;
\item[(2)] $\tau_k\boxtimes \tau_k'$ is a simple composition factor in $(\pi_k)_{N_{i_k}}$;
\item[(3)] $D_{[e,c]_{\rho}}(\pi)$ is a simple composition factor in $\tau_1\times \ldots \times \tau_r$ and $\mathrm{St}([e,c]_{\rho})$ is a simple composition factor in $\tau_1'\times \ldots \times \tau_r'$.
\end{enumerate}
This then implies that there exists an index $k$ such that $\nu^c\rho \in \mathrm{csupp}(\tau_k)$ and 
\[ \mathrm{csupp}(\tau_k) \subset [e,c]_{\rho} . \]
This then implies that $D_{[e',c]_{\rho}}(\tau_k) \neq 0$ for some $c \geq e'\geq e \geq d$. But, by Lemma \ref{lem zero derivative reduce}, $D_{[e',c]_{\rho}}(\pi_k) \neq 0$, giving a contradiction.
\end{proof}
\begin{proposition} \label{prop max b under product}
Let $\pi_1, \ldots, \pi_r \in \mathrm{Irr}_{\rho}$ and let $c \in \mathbb Z$. Suppose $\pi_1 \times \ldots \times \pi_r$ is irreducible. Then
\[ \mathfrak{mxpt}^b(\pi_1\times \ldots \times \pi_r, c) =\mathfrak{mxpt}^b(\pi_1, c)+\mathfrak{mxpt}^b(\pi_2, c)+\ldots +\mathfrak{mxpt}^b(\pi_r,c) . \]
\end{proposition}

\begin{proof}
Let $\mathfrak m_k=\mathfrak{mxpt}^b(\pi_k, c)$ for all $k$ and let $\mathfrak m=\mathfrak m_1+\ldots+\mathfrak m_r$. Let $l_k=l_{abs}(\mathfrak{m}_k)$ and let $l=l_1+\ldots+l_r$. Let $n_k=\mathrm{deg}(\pi_k)$. One of the layers in the geometric lemma of
\[  (\pi_1 \times \ldots \times \pi_r)_{N_l} 
\]
takes the form:
\begin{align} \label{eqn main layer contribute}
   \mathrm{Ind}_{P'}^{G_n} ((\pi_1)_{N_{l_1}} \boxtimes (\pi_2)_{N_{l_2}} \boxtimes \ldots \boxtimes (\pi_r)_{N_{l_r}})^{\phi} ,
\end{align}
where
\begin{enumerate}
\item $P'=P_{n_1-l_1,\ldots, n_r-l_r, l_1, \ldots, l_r}$; and
\item  $\phi$ is a natural map taking a $G_{n_1-l_1} \times G_{l_1} \times \ldots \times G_{n_r-l_r}\times G_{l_r}$-representation to a $G_{n_1-l_1}\times \ldots \times G_{n_r-l_r}\times G_{l_1}\times \ldots \times G_{l_r}$. 
\end{enumerate}

By Lemma \ref{lem cuspidal support on irr pair}(1),
\[    (D_{\mathfrak m_1}(\pi)\times \ldots \times D_{\mathfrak m_r}(\pi)) \boxtimes (\mathrm{St}(\mathfrak m_1)\times \ldots \times \mathrm{St}(\mathfrak m_k)) 
\]
is a composition factor in $(\pi_1 \times \ldots \times \pi_r)_{N_l}$. Lemma \ref{lem cuspidal support on irr pair} (1) and (2) guarantee that there are no other composition factors of the form $\tau_1 \boxtimes \tau_2$ in the layer (\ref{eqn main layer contribute}) with $\mathrm{csupp}(\tau_2)=\mathrm{csupp}(\mathfrak m)$.

Other than the layer (\ref{eqn main layer contribute}), we also have layers from the geometric lemma of the form:
\begin{align} \label{eqn main layer contribute 2}
   \mathrm{Ind}_{P''}^{G_n} ((\pi_1)_{N_{l_1'}} \boxtimes (\pi_2)_{N_{l_2'}} \boxtimes \ldots \boxtimes (\pi_r)_{N_{l_r'}})^{\phi'} ,
\end{align}
where $P''$, $l_1',\ldots, l_r'$ and $\phi'$ are similarly defined as in (\ref{eqn main layer contribute}).

By Lemma \ref{lem cuspidal support on irr pair}(3) and (4), there are no composition factors of the form $\tau_1 \boxtimes \tau_2$ in the layer (\ref{eqn main layer contribute 2}) with $\mathrm{csupp}(\tau_2)=\mathrm{csupp}(\mathfrak m)$.

Combining the above situations, we have that 
\[   (D_{\mathfrak m_1}(\pi)\times \ldots \times D_{\mathfrak m_r}(\pi)) \boxtimes (\mathrm{St}(\mathfrak m_1)\times \ldots \times \mathrm{St}(\mathfrak m_k)) 
\]
appears in a direct summand of $(\pi_1 \times \ldots \times \pi_r)_{N_l}$ and so
\[  (\pi_1 \times \ldots \times \pi_r)_{N_l} \twoheadrightarrow (D_{\mathfrak m_1}(\pi)\times \ldots \times D_{\mathfrak m_r}(\pi)) \boxtimes (\mathrm{St}(\mathfrak m_1)\times \ldots \times \mathrm{St}(\mathfrak m_k)) .
\]
Note that since $\mathfrak{mxpt}^b(D_{\mathfrak m_i}(\pi), c)=\emptyset$ for all $i=1,\ldots, r$, any simple composition factor $\omega$ in $D_{\mathfrak m_1}(\pi)\times \ldots \times D_{\mathfrak m_r}(\pi)$ satisfies $\mathfrak{mxpt}^b(\omega, c)=\emptyset$ by Proposition \ref{prop delta reduced preserve under parabolic}. Lemma \ref{lem irr submodule under reduced} in turn implies the proposition. 
\end{proof}

\section{Maximal multisegment at a left point} \label{s multisegment at a left pt}

\subsection{ Maximal multisegments at a (left) point} \label{ss max at a pt}
For $c \in \mathbb Z$, recall that $\mathrm{Mult}_{\rho, c}^a$ is defined in Section \ref{ss ordering mult}, whose elements will be referred as {\it multisegments at the point $\nu^c\rho$} (see Section \ref{ss main results}). An ordering $\leq_c^a$ on $\mathrm{Mult}^a_{\rho,c}$ is also defined in Section \ref{ss ordering mult}. 




Since $D_{\mathfrak m}(\pi) \times \mathrm{St}(\mathfrak m) $ has a unique submodule by Lemma \ref{lem socle cosocle ladder} , multiple use of Lemma \ref{lem reformulate frobenius} gives that
\begin{align} \label{eqn embedding derivatives induction form}
  \pi \hookrightarrow D_{\mathfrak m}(\pi) \times \mathrm{St}(\mathfrak m) 
\end{align}
and by Frobenius reciprocity and Proposition \ref{prop sub quo Jacquet}, we also have:
\begin{align} \label{eqn embedding of derivative}
   D_{\mathfrak m}(\pi) \boxtimes \mathrm{St}(\mathfrak m) \hookrightarrow \pi_{N_{l_{abs}(\mathfrak m)}} .
\end{align}

\begin{lemma} \label{lem max a positive}
Let $\pi \in \mathrm{Irr}_{\rho}$ and let $c \in \mathbb Z$. For each integer $d \geq c$, define
\[  k_d =  \varepsilon_{[c,d]_{\rho}}(\pi) -\varepsilon_{[c,d+1]_{\rho}}(\pi) .
\]
Then $k_d \geq 0$.
\end{lemma}

\begin{proof}
Let $p=\varepsilon_{[c,d]_{\rho}}(\pi)$. By definitions, 
\begin{align} \label{eqn embedding unique max}
  \pi  \hookrightarrow (D_{[c,d+1]_{\rho}})^{p}(\pi) \times \mathrm{St}([c,d+1]_{\rho})^{\times p} .
\end{align}

On the other hand, 
\[  \mathrm{St}([c,d+1]_{\rho}) \hookrightarrow \nu^{d+1}\rho \times \mathrm{St}([c,d]_{\rho}) .
\]
Using (\ref{eqn product commut st}), we then have:
\begin{align} \label{eqn embedding decompose generic}
  \mathrm{St}([c,d+1]_{\rho})^{\times \varepsilon_{[c,d+1]_{\rho}}(\pi)} \hookrightarrow (\nu^{d+1}\rho)^{\times p} \times \mathrm{St}([c,d]_{\rho})^{\times p}
\end{align}

Combining (\ref{eqn embedding unique max}) and (\ref{eqn embedding decompose generic}), we have:
\[   \pi   \hookrightarrow D_{[c,d+1]_{\rho}}^{p}(\pi) \times  (\nu^{d+1}\rho)^{\times p} \times \mathrm{St}([c,d]_{\rho})^{\times p}.
\]
By Frobenius reciprocity, we then have $D_{[c,d]_{\rho}}^{p}(\pi)\neq 0$. This shows the lemma.
\end{proof}

\begin{proposition} \label{prop define mxpt a}
We keep using notations in Lemma \ref{lem max a positive}. Define
\begin{align} \label{eqn mxpt a}
  \mathfrak{mxpt}^a(\pi,c): = \sum_{d \geq c} k_d \cdot [c,d]_{\rho} ,
\end{align}
where $k_d$ represents the multiplicity of $[c,d]_{\rho}$ in $\mathfrak{mxpt}^a(\pi,c)$. Then
 \[ D_{\mathfrak{mxpt}^a(\pi,c)}(\pi)\neq 0.
\]
\end{proposition}

\begin{proof}
 Let $\mathfrak m_e =\sum_{d\geq e} k_d\cdot [c,d]_{\rho}$. We shall prove by back induction on $e$ such that 
\[   D_{\mathfrak m_e}(\pi) \neq 0 .\]
When $e=c$, this gives the required statement.

We now choose $e^*$ to be the largest integer such that $\mathfrak m_{e^*}\neq \emptyset$. In such case, all segments in $\mathfrak m_{e^*}$ are $[c,e^*]_{\rho}$. By the definition of $\varepsilon$, $D_{\mathfrak m_{e^*}}(\pi) \neq 0$. 

We now proceed to prove $D_{\mathfrak m_e}(\pi)\neq 0$ for $c\leq e<e^*$. By the induction hypothesis, $D_{\mathfrak m_{e+1}}(\pi) \neq 0$. 

On the other hand, let $\mathfrak n$ be the submultisegment of $\mathfrak{mxpt}^b(D_{\mathfrak m_{e+1}}(\pi),e)$ precisely containing all segments in $\mathfrak{mxpt}^b(D_{\mathfrak m_{e+1}}(\pi),e)$ of the form $[a,e]_{\rho}$ with $a \geq c$. Then,
\[   \pi \hookrightarrow D_{\mathfrak m_{e+1}}(\pi) \times \mathrm{St}(\mathfrak m_{e+1}) \hookrightarrow D_{\mathfrak n}\circ D_{\mathfrak m_{e+1}}(\pi)\times \mathrm{St}(\mathfrak n)\times \mathrm{St}(\mathfrak m_{e+1}) .
\] 
By (\ref{eqn product commut st}), we then have:
\begin{align} \label{eqn max a inductive}
  \pi \hookrightarrow D_{\mathfrak n}\circ D_{\mathfrak m_{e+1}}(\pi) \times \mathrm{St}(\mathfrak m_{e+1}) \times \mathrm{St}(\mathfrak n) .
\end{align}

On the other hand, 
\begin{align} \label{eqn max a separate}
   \mathrm{St}(\mathfrak m_{e+1}) \hookrightarrow \mathrm{St}(\mathfrak m'_{e+1}) \times \mathrm{St}(l\cdot [c,e]_{\rho} ),
\end{align}
where $l=|\mathfrak m_{e+1}|=\varepsilon_{[c,e+1]_{\rho}}(\pi)$.

Combining (\ref{eqn max a inductive}) and (\ref{eqn max a separate}), we have:
\[ \pi \hookrightarrow D_{\mathfrak n}\circ D_{\mathfrak m_{e+1}}(\pi) \times \mathrm{St}(\mathfrak m_{e+1}) \times \mathrm{St}(\mathfrak n) \hookrightarrow D_{\mathfrak n}\circ D_{\mathfrak m_{e+1}}(\pi)\times \mathrm{St}(\mathfrak m'_{e+1}) \times \mathrm{St}(l\cdot [c,e]_{\rho} ) \times \mathrm{St}(\mathfrak n) .
\]
Hence,
\[ \pi \hookrightarrow \tau \times \mathrm{St}(l\cdot [c,e]_{\rho}+\mathfrak n) 
\]
for some simple composition factor in $D_{\mathfrak n}\circ D_{\mathfrak m_{e+1}}(\pi)\times \mathrm{St}(\mathfrak m'_{e+1})$. Moreover, by Corollary \ref{cor mxpt b product submod},
\begin{align} \label{eqn mxpt b sum in mxpt a}
   \mathfrak{mxpt}^b(\pi, e)=\mathfrak{mxpt}^b(\tau, e)+l\cdot[c,e]_{\rho}+\mathfrak n .
\end{align}

Note that
\[  D_{[f,e]_{\rho}}(D_{\mathfrak n}\circ D_{\mathfrak m_{e+1}}(\pi))=0, \quad D_{[f,e]_{\rho}}(\mathrm{St}(\mathfrak m'_{e+1})) =0
\]
for any $c \leq f \leq e$. Hence, by Proposition \ref{prop delta reduced preserve under parabolic},  $D_{[f,e]_{\rho}}(\tau)=0$ for any $c \leq f \leq e$. By Corollary \ref{cor mxpt b product submod}, there is no segment in $\mathfrak{mxpt}^b(\tau, e)$ of the form $[c,e]_{\rho}$, and $\mathfrak n$ contains precisely $\varepsilon_{[c,e]_{\rho}}(\pi)$ segments of the form $[c,e]_{\rho}$. Then, by Corollary \ref{cor mxpt b product submod} again and Corollary \ref{cor mxpt b epsilon},
\[  \varepsilon_{[c,e]_{\rho}}(\pi) =l+\varepsilon_{[c,e]_{\rho}}(D_{\mathfrak m_{e+1}}(\pi))
\]
and so 
\[  \varepsilon_{[c,e]_{\rho}}(D_{\mathfrak m_{e+1}}(\pi))=\varepsilon_{[c,e]_{\rho}}(\pi)-l=k_e .\]
Thus $D_{\mathfrak m_e}(\pi)=(D_{[c,e]_{\rho}})^{k_e}\circ D_{\mathfrak m_{e+1}}(\pi)\neq 0$, as desired.  
\end{proof}

\begin{proposition} \label{prop uniqueness max multi}
Let $\pi \in \mathrm{Irr}_{\rho}$. Let $c \in \mathbb Z$. Then $\mathfrak{mxpt}^a(\pi, c)$ defined in Proposition \ref{prop define mxpt a} is the unique $\leq^a_c$-maximal multisegment $\mathfrak m$ in $\mathrm{Mult}_{\rho, c}^a$ such that $D_{\mathfrak m}(\pi)\neq 0$. 
\end{proposition}

\begin{proof}
Let $\mathfrak m \in \mathrm{Mult}_{\rho,c}^a$ such that $D_{\mathfrak m}(\pi)\neq 0$. Write
\[  \mathfrak{mxpt}^a(\pi,c) =\left\{ [c,b_1]_{\rho}, \ldots, [c, b_r]_{\rho} \right\}, \quad \mathfrak{m}=\left\{ [c,b_1']_{\rho}, \ldots, [c,b_s']_{\rho} \right\} 
\]
with $b_1 \geq b_2\geq \ldots \geq b_r$ and $b_1'\geq b_2' \geq \ldots \geq b_s'$.

If $\mathfrak m \not\leq^a_c \mathfrak{mxpt}^a(\pi,c)$, then there exists an index $k$ such that $b_k' >b_k$. Let $k^*$ be such smallest integer. Let
\[ \mathfrak n = \left\{ [c,b_1']_{\rho}, \ldots, [c, b_{k^*}']_{\rho} \right\}.
\]
Then, it follows from the way to define $\mathfrak{mxpt}^a(\pi, c)$ in (\ref{eqn mxpt a}),
\begin{align} \label{eqn index greater epsilon}
  k^* >\varepsilon_{[c,b_{k^*}]_{\rho}}(\pi) .
\end{align}

Since $D_{\mathfrak n}(\pi)\neq 0$, we have 
\[   \pi \hookrightarrow D_{\mathfrak n}(\pi) \times \mathrm{St}(\mathfrak n) .
\]
We also have
\[  \mathrm{St}(\mathfrak n) \hookrightarrow \mathrm{St}(\mathfrak n')\times \mathrm{St}(k^*\cdot [c,b_{k^*}]_{\rho}) ,
\]
where $\mathfrak n'=[b_{k^*}+1, b_1']+\ldots +[b_{k^*}+1, b_{k^*-1}']_{\rho}$. Thus,
\[  \pi \hookrightarrow D_{\mathfrak n}(\pi) \times \mathrm{St}(\mathfrak n')\times \mathrm{St}(k^*\cdot [c,b_{k^*}]_{\rho})
\] 
and so $(D_{[c,b_{k^*}]_{\rho}})^{k^*}(\pi)\neq 0$. This contradicts to (\ref{eqn index greater epsilon}). 
\end{proof}

\begin{example}
Let $\pi=\mathrm{St}(\left\{ [0,5],[0,4],[1,4],[1,3],[2,3] \right\})$. Note that $\pi$ is a generic representation. In this case, $\mathfrak{mxpt}^a(\pi, 1)=\left\{ [1,3], [1,4] \right\}$. Then, $[1,3]$ and $[1,4]$ has multiplicity one in $\mathfrak{mxpt}^a(\pi, 1)$ and so one deduces from Proposition \ref{prop define mxpt a} that 
\[\varepsilon_{[1,3]}(\pi)=2, \quad \varepsilon_{[1,4]}(\pi) =1 .
\]

\end{example}


As we shall see later, the two notions of multisegments at a point, defined in Definition \ref{def max R multiseg} and Proposition \ref{prop define mxpt a}, have different applications. The maximal multisegment at a left point is primarily used to define the highest derivative multisegment. In contrast, the maximal multisegment at a right point is mainly employed in certain inductive arguments within proofs.

\subsection{Characterizing multisegments $\mathfrak n$ in $\mathrm{Mult}^a_{\rho,c}$ with $D_{\mathfrak n}(\pi)\neq 0$}

\begin{proposition} \label{prop smaller derivative}
 Let $\pi \in \mathrm{Irr}_{\rho}$ and let $\mathfrak n \in \mathrm{Mult}^a_{\rho,c}$. Let $c \in \mathbb Z$. Then $\mathfrak n \leq^a_c \mathfrak{mxpt}^a(\pi,c)$ if and only if $D_{\mathfrak n}(\pi) \neq  0$.

\end{proposition}

\begin{proof}
The if direction is clear by the maximality in Proposition \ref{prop uniqueness max multi}.

We now prove the only if direction. Let $\mathfrak n \leq_c^a \mathfrak{mxpt}^a(\pi, c)$. Let $i=l_{abs}(\mathfrak n)$. We first prove the following claim. \\

\noindent
{\it Claim:} There exists $\tau \in \mathrm{Irr}_{\rho}$ such that 
\[  \tau \boxtimes \mathrm{St}(\mathfrak n) \hookrightarrow \mathrm{St}(\mathfrak{mxpt}^a(\pi, c))_{N_i} .
\]

\noindent
{\it Proof of Claim:} We write the segments in $\mathfrak{mxpt}^a(\pi, c)$ as $\Delta_1=[c,b_1]_{\rho}, \Delta_2=[c,b_2]_{\rho}, \ldots, \Delta_r=[c,b_r]_{\rho}$ satisfying $b_1 \geq b_2 \geq \ldots \geq b_r$. Now one applies the geometric lemma on 
\[   (\mathrm{St}(\Delta_1) \times \ldots \times \mathrm{St}(\Delta_r))_{N_i} .
\]
Then, any layer takes the form 
\begin{align} \label{eqn layer in partial order lem}
(\mathrm{St}([b_1'+1,b_1]_{\rho})\times \ldots \times \mathrm{St}([b_r'+1,b_r]_{\rho}) \boxtimes  (\mathrm{St}([c,b_1']_{\rho})\times \ldots \times \mathrm{St}([c,b_r']_{\rho})),
\end{align}
for all $b_k' \leq b_k$.

In particular, any submodule of those layers takes the form $\tau \boxtimes (\mathrm{St}([c,b_1']_{\rho})\times \ldots \times \mathrm{St}([c,b_r']_{\rho}))$ for some $\tau \in \mathrm{Irr}_{\rho}$. This then implies that any submodule of $  (\mathrm{St}(\Delta_1) \times \ldots \times \mathrm{St}(\Delta_r))_{N_i} $ also takes the form. By the condition of $\mathfrak n \leq^a_c \mathfrak{mxpt}^a(\pi,c)$, we can find $b_1',\ldots, b_r'$ in (\ref{eqn layer in partial order lem}) such that $\mathrm{St}([c,b_1']_{\rho}\times \ldots \times \mathrm{St}([c,b_r']_{\rho}) \cong \mathrm{St}(\mathfrak n)$. Thus, $(\mathrm{St}(\Delta_1) \times \ldots \times \mathrm{St}(\Delta_r))_{N_i}$ has a submodule of the form $\tau \boxtimes \mathrm{St}(\mathfrak n)$, proved the claim.  \\

Now we return to prove the lemma. Let $\mathfrak m=\mathfrak{mxpt}^a(\pi, c)$. By (\ref{eqn embedding of derivative}), 
\[  D_{\mathfrak m}(\pi) \boxtimes \mathrm{St}(\mathfrak m) \hookrightarrow \pi_{N_l} .
\]
Thus we have a non-zero map:
\[  D_{\mathfrak{m}}(\pi) \boxtimes \tau \boxtimes \mathrm{St}(\mathfrak n) \hookrightarrow \pi_N,
\]
where $N=N_{\mathrm{deg}(\pi)-l_{abs}(\mathfrak{m}),l_{abs}(\mathfrak{m})-i ,i}$. By Frobenius reciprocity, we have 
\[\tau' \boxtimes \mathrm{St}(\mathfrak n) \hookrightarrow \pi_{N_{i}}\]
 for some $\tau' \in \mathrm{Irr}_{\rho}$. Then $D_{\mathfrak n}(\pi)\neq 0$, as desired.
\end{proof}






\section{Some results for commutativity of derivatives} \label{sec commutation of der}

\subsection{First commutativity result}

\begin{lemma} \label{lem non-zero composite}
Let $\pi \in \mathrm{Irr}_{\rho}$. Let $\Delta_1$ and $\Delta_2$ be linked segments with $\Delta_1<\Delta_2$. Suppose $D_{\Delta_1}(\pi)\neq 0$ and $D_{\Delta_2}(\pi)\neq 0$. Then 
\[  D_{\Delta_1}\circ D_{\Delta_2}(\pi)\neq 0, \quad D_{\Delta_2}\circ D_{\Delta_1}(\pi)\neq 0.
\]
\end{lemma}

\begin{proof}
Write $\Delta_1=[a_1,b_1]_{\rho}$ and $\Delta_2=[a_2,b_2]_{\rho}$. Let $\omega_1 =D_{\Delta_1}(\pi)$ and let $\omega_2=D_{\Delta_2}(\pi)$. Let $k=b_1-a_1+1$ and let $l=\mathrm{deg}(\rho)$. Then 
\[  \pi \hookrightarrow \omega_1\times \mathrm{St}(\Delta_1), \quad \pi \hookrightarrow \omega_2 \times \mathrm{St}(\Delta_2) .
\] 
We again use the argument of the geometric lemma and a comparison of cuspidal supports to prove both non-inequalities. The first one is indeed easier since the segment cannot contribute $\nu^{a_1}\rho$ for the cuspidal support of $\mathrm{St}(\Delta_1)$, and most arguments are similar to proving the second one, and so we omit the details.

Now we prove the second one. One applies the geometric lemma on 
\[   (\omega_2 \times \mathrm{St}(\Delta_2))_{N_{kl}} ,
\]
to reduce the possibility contributing to a factor of the form $ \tau \boxtimes \mathrm{St}(\Delta_1)$ to those of the form
\begin{align} \label{eqn geo lem 1}
   (\tau' \times \mathrm{St}([c+1,b_2]_{\rho})) \boxtimes (\mathrm{St}([a_1,a_2-1]_{\rho}) \times \mathrm{St}([c+1,b_1]_{\rho}) \times \mathrm{St}([a_2,c]_{\rho})) ,
\end{align}
for some $c \leq b_1-1$, or simply
\begin{align} \label{eqn geo lem 2}
 \tau'' \times \mathrm{St}([a_2,b_2 ]_{\rho}) \boxtimes \mathrm{St}([a_1,b_1]_{\rho}) .
\end{align}
Here one of the irreducible factor in (\ref{eqn geo lem 1}) has to take the form $\mathrm{St}([a_1,a_2-1]_{\rho})\times \mathrm{St}([c+1,b_1]_{\rho})$ by picking the unique generic representation with given cuspidal support, and another irreducible factor $\mathrm{St}([a_2,c]_{\rho})$ comes from the Jacquet module of $\mathrm{St}(\Delta_2)_{N_{(c-a_2+1)l}}$ in (\ref{eqn steinberg jacquet}).

Here $\tau'$ is an irreducible representation such that $\tau'\boxtimes  \mathrm{St}([a_1,a_2-1]_{\rho}) $ is a composition factor of $(\omega_2)_{N_{(a_2-a_1)l}}$; and $\tau''$ is an irreducible representation such that $\tau'' \boxtimes \mathrm{St}([a_1,b_1]_{\rho})$. However, for (\ref{eqn geo lem 1}), by Frobenius reciprocity, the $\mathrm{St}([a_1,b_1]_{\rho})$ does not appear in the submodule of $\mathrm{St}([a_1,a_2-1]_{\rho})\times \mathrm{St}([a_2,b_1]_{\rho})$. Thus only (\ref{eqn geo lem 2}) can contribute to a submodule of the form $\kappa \boxtimes \mathrm{St}(\Delta_1)$ in $\pi_{N_{(b_1-a_1+1)l}}$ (which we know such submodule exists by $D_{\Delta_1}(\pi)\neq 0$). Thus the socle of $(\omega_2)_{N_{(b_1-a_1+1)l}}$ has a factor of the form $\tau'' \boxtimes \mathrm{St}(\Delta_1)$. This shows $D_{\Delta_1}\circ D_{\Delta_2}(\pi)\neq 0$.
\end{proof}

\begin{proposition} \label{prop general commutativity}
Let $\pi \in \mathrm{Irr}_{\rho}$. Let $\Delta_1$ and $\Delta_2$ be linked segments with $\Delta_1< \Delta_2$. Suppose $D_{\Delta_1}(\pi)\neq 0$ and $D_{\Delta_2}(\pi)\neq 0$. Suppose further
\[   D_{\Delta_2}\circ D_{\Delta_1}(\pi) \not\cong D_{\Delta_1\cap \Delta_2}\circ D_{\Delta_1\cup \Delta_2}(\pi) ,
\]
and
\[  D_{\Delta_1}\circ D_{\Delta_2}(\pi) \not\cong D_{\Delta_1\cap \Delta_2}\circ D_{\Delta_1\cup \Delta_2}(\pi) .
\]
Then $D_{\Delta_2}\circ D_{\Delta_1}(\pi) \cong D_{\Delta_1}\circ D_{\Delta_2}(\pi)$.
\end{proposition}

\begin{proof}
Write $\Delta_1=[a_1,b_1]_{\rho}$ and $\Delta_2=[a_2,b_2]_{\rho}$. By Lemma \ref{lem non-zero composite}, we have $D_{\Delta_2}\circ D_{\Delta_1}(\pi)\neq 0$ and $D_{\Delta_1}\circ D_{\Delta_2}(\pi)\neq 0$. By Frobenius reciprocity, 
\[  \pi \hookrightarrow D_{\Delta_2}\circ D_{\Delta_1}(\pi) \times \mathrm{St}(\Delta_2) \times \mathrm{St}(\Delta_1) .
\]
There are two composition factors in $\mathrm{St}(\Delta_2)\times \mathrm{St}(\Delta_1)$ \cite[Proposition 4.6]{Ze80}: $\mathrm{St}( \Delta_1+\Delta_2)$ and $\mathrm{St}(\Delta_1\cup \Delta_2+ \Delta_1\cap \Delta_2 )$. Thus we have either:
\[  \pi \hookrightarrow D_{\Delta_2}\circ D_{\Delta_1}(\pi) \times \mathrm{St}( \Delta_1\cap \Delta_2+\Delta_1\cup \Delta_2 ) ,
\]
or
\[  \pi \hookrightarrow D_{\Delta_2}\circ D_{\Delta_1}(\pi) \times \mathrm{St}( \Delta_1+ \Delta_2 ) .
\]
In the former case, Lemma \ref{lem socle cosocle ladder} gives that $D_{\Delta_1\cap \Delta_2}\circ D_{\Delta_1\cup \Delta_2}(\pi) \cong D_{\Delta_2}\circ D_{\Delta_1}(\pi)$, giving a contradiction. Thus we must be in the latter case.

We similarly have that 
\[  \pi \hookrightarrow D_{\Delta_1}\circ D_{\Delta_2}(\pi)\times \mathrm{St}(\Delta_1)\times \mathrm{St}(\Delta_2) .
\]
With a similar argument as above, we have that 
 \[  \pi \hookrightarrow D_{\Delta_1}\circ D_{\Delta_2}(\pi) \times \mathrm{St}( \Delta_1+ \Delta_2 ) .
\]
Now the ladder representation case of Lemma \ref{lem socle cosocle ladder} gives that
\[  D_{\Delta_2}\circ D_{\Delta_1}(\pi) \cong D_{\Delta_1}\circ D_{\Delta_2}(\pi). 
\]
\end{proof}



We shall need the following lemma in Section \ref{ss isomorphic quotients}. We also remark that dropping the condition that $D_{\Delta_1}(\pi)\neq 0$ or $D_{\Delta_2}(\pi) \neq 0$ in Corollary \ref{cor commut derivative} will make the statement fail in general (e.g. considering some derivatives on a Speh representation).

\begin{corollary} \label{cor commut derivative}
Let $\pi \in \mathrm{Irr}_{\rho}$. Let $\Delta_1$ and $\Delta_2$ be linked segments with $\Delta_1<\Delta_2$. Suppose $D_{\Delta_1}(\pi)\neq 0$ and $D_{\Delta_2}(\pi)\neq 0$. If $D_{\Delta_1 \cup \Delta_2}(\pi)=0$, then 
\[ D_{\Delta_1}\circ D_{\Delta_2}(\pi) \cong D_{\Delta_2}\circ D_{\Delta_1}(\pi) .  \] 
\end{corollary}

\begin{proof}
This is a special case of Proposition \ref{prop general commutativity}.
\end{proof}

\subsection{Commutations in another form}

As mentioned before, Corollary \ref{cor commut derivative} requires the assumption that $D_{\Delta_1}(\pi)\neq 0$ and $D_{\Delta_2}(\pi)\neq 0$. It is not convenient for the purpose of some applications. We now prove another version of commutativity, and one may compare with the proof of Corollary \ref{cor commut derivative}.

\begin{lemma} \label{lem derivative transfer large}
Let $\pi \in \mathrm{Irr}_{\rho}$. Let $c \in \mathbb Z$. Let $\tau =D_{\mathfrak{mxpt}^a(\pi, c)}(\pi)$. Let $d>c$ be an integer. Let $\mathfrak n \in \mathrm{Mult}_{\rho,d}^a$. If $D_{\mathfrak n}(\tau)\neq 0$, then $D_{\mathfrak n}(\pi) \neq 0$. 
\end{lemma}

\begin{proof}

Let $\tau'=D_{\mathfrak n}(\tau)$. By Lemma \ref{lem reformulate frobenius}, we have embeddings:
\[   \pi \hookrightarrow \tau \times \mathrm{St}(\mathfrak{mxpt}^a(\pi, c)),\quad  \tau \hookrightarrow \tau' \times \mathrm{St}(\mathfrak n) .\]
Hence,
\[ \pi \hookrightarrow \tau'\times \mathrm{St}(\mathfrak n)\times \mathrm{St}(\mathfrak{mxpt}^a(\pi, c)) .
\]
The lemma will follow from the following claim. The main idea in proving the following claim is to switch a pair of segments respectively in $\mathfrak n$ and $\mathfrak{mxpt}^a(\pi,c)$ each time by using the maximality of $\mathfrak{mxpt}^a(\pi,c)$.
\\

\noindent
{\it Claim:} $\pi \hookrightarrow \tau' \times \mathrm{St}(\mathfrak{mxpt}^a(\pi, c))\times \mathrm{St}(\mathfrak n)$. \\
\noindent
{\it Proof of Claim:} We shall write the segments in $\mathfrak{mxpt}^a(\pi, c)$ as:
\begin{align} \label{eqn order in proof comm}
  \Delta_k \leq_c^a \ldots \leq_c^a \Delta_1 ,
\end{align}
and write the segments in $\mathfrak n$ as:
\[  \overline{\Delta}_1 \leq_d^a \ldots \leq_d^a \overline{\Delta}_l . 
\]
Note that the order is opposite to (\ref{eqn order in proof comm}).

We shall inductively show that:
\[ \pi \hookrightarrow \tau' \times A_{ij} \times \mathrm{St}(\overline{\Delta}_j)\times B_{ij} ,
\]
where 
\[  A_{ij} =(\mathrm{St}(\overline{\Delta}_l)\times \ldots \times \mathrm{St}(\overline{\Delta}_{j+1}))\times (\mathrm{St}(\Delta_{k})\times \ldots \times \mathrm{St}(\Delta_i))
\]
and
\[  B_{ij} =\mathrm{St}(\Delta_{i-1})\times \ldots \times \mathrm{St}(\Delta_1)\times \mathrm{St}(\overline{\Delta}_{j-1})\times \ldots \times \mathrm{St}(\overline{\Delta}_1) .
\]
The basic case has been given before the claim. Suppose the case is proved for $i=i^*$ and $j=j^*$. To prove the case that $i=i^*-1$ and $j=j^*$ (if $i^*=1$, then we proceed $i=k+1$ and $j=j^*+1$ and the argument is similar), we consider two cases:
\begin{enumerate}
\item $\overline{\Delta}_{j^*}\subset \Delta_{i^*-1}$. Then it follows from $(i^*, j^*)$ case and the fact that $\mathrm{St}(\overline{\Delta}_{j^*})\times \mathrm{St}(\Delta_{i^*-1})\cong \mathrm{St}(\Delta_{i^*-1})\times \mathrm{St}(\overline{\Delta}_{j^*})$. 
\item $\overline{\Delta}_{j^*} \not\subset \Delta_{i^*-1}$. Then there are exactly two simple composition factors in $\mathrm{St}(\overline{\Delta}_{j^*})\times \mathrm{St}(\Delta_{i^*-1})$, and those are 
\[  R= \mathrm{St}(\overline{\Delta}_{j^*}\cup \Delta_{i^*-1}) \times \mathrm{St}(\overline{\Delta}_{j^*}\cap \Delta_{i^*-1}) ,
\]
and a non-generic factor denoted by $S$. 
\end{enumerate}

Now, by induction hypothesis, 
\[  \pi \hookrightarrow \tau' \times A_{i^*, j^*} \times \mathrm{St}(\overline{\Delta}_{j^*})\times \mathrm{St}(\Delta_{i^*-1}) \times B_{i^*-1, j^*} \]
and so the above discussion implies that
\begin{align}\label{eqn generic case}
\pi \hookrightarrow \tau' \times A_{i^*, j^*}\times R \times B_{i^*-1, j^*} ,
\end{align}
or 
\begin{align} \label{eqn nongeneric case}
 \pi \hookrightarrow \tau' \times A_{i^*, j^*} \times S \times B_{i^*-1,j^*} .
\end{align}

We first prove the former case (\ref{eqn generic case}) is impossible. Suppose the former case happens. We write $\underline{\Delta}=\overline{\Delta}_{j^*} \cup \Delta_{i^*-1}$. We choose all the segments $\Delta_1, \ldots, \Delta_{p}$ in $\mathfrak{mxpt}^a(\pi, c)$ such that those $ \underline{\Delta} \subset \Delta_x$ ($x=1, \ldots, p$). Then, by the ordering above, we also have
\[  \overline{\Delta}_{j^*-1}\subset \ldots \subset \overline{\Delta}_1 \subset \underline{\Delta} \subset \Delta_p\subset \ldots \subset \Delta_1 .
\]
We also further have that $\underline{\Delta}$ is unlinked to $\Delta_y$ for any $y$. Thus, using (\ref{eqn product commut st}) several times, we have that: 
\[ A_{i^*, j^*}\times R \times B_{i^*-1,j^*} \cong A_{i^*, j^*} \times \widetilde{R}  \times \mathrm{St}(\underline{\Delta}) \times \mathrm{St}(\Delta_{p})\times \ldots \times \mathrm{St}(\Delta_1) ,
\]
where $\widetilde{R}$ is the product of those $\mathrm{St}(\Delta')$s for the remaining segments.

By Frobenius reciprocity, we have that
\[   D_{\left\{\overline{\Delta}, \Delta_1, \ldots, \Delta_p\right\}}(\pi) \neq 0 .
\]
However, due to the extra $\overline{\Delta}$, $\left\{\overline{\Delta}, \Delta_1, \ldots, \Delta_p\right\} \not\leq^a_c \mathfrak{mxpt}^a(\pi,c)$, contradicting the maximality of $\mathfrak{mxpt}^a(\pi,c)$.


Thus we must lie in the case of (\ref{eqn nongeneric case}). Now combining with
\[  S \hookrightarrow \mathrm{St}(\Delta_{i^*-1}) \times \mathrm{St}(\overline{\Delta}_{j^*}) ,
\]
we obtain the case that $i=i^*-1$ and $j=j^*$, as desired.
\end{proof}

\section{Highest derivative multisegments} \label{s highest der multisegment}

In this section, we construct the highest derivative of an irreducible representation by a sequence of St-derivatives. One may compare with the construction using $\rho$-derivatives in Proposition \ref{prop construction highest from rho}. The two situations give two extreme cases: the minimal (shown in \cite{Ch22+d}) and maximal one (Proposition \ref{prop construction highest from rho}) in the set $\mathcal S(\pi, \pi^-)$ under $\leq_Z$.

\subsection{A computation on maximal multisegments at a point}

\begin{lemma} \label{lem max multisegment equal}
Let $\pi \in \mathrm{Irr}_{\rho}$. Let $c \in \mathbb Z$ with $\varepsilon_c(\pi)\neq 0$. Then, for any $d >c$, 
\[ \mathfrak{mxpt}^a(D_{\mathfrak{mxpt}^a(\pi, c)}(\pi),d)=\mathfrak{mxpt}^a(\pi, d) . \]
\end{lemma}

\begin{proof}
Let $\mathfrak h_c=\mathfrak{mxpt}^a(\pi,c)$ and let $\mathfrak h_d=\mathfrak{mxpt}^a(\pi,d)$. Let  $n_1=l_{abs}(\mathfrak m_c)$ and $n_2=l_{abs}(\mathfrak m_d)$. By (\ref{eqn embedding derivatives induction form}),
\[   \pi \hookrightarrow  D_{\mathfrak h_c}(\pi) \times \mathrm{St}(\mathfrak h_c) .
\]
This implies that
\[  D_{\mathfrak h_d}(\pi) \boxtimes \mathrm{St}(\mathfrak h_d) \hookrightarrow (D_{\mathfrak h_c}(\pi) \times \mathrm{St}(\mathfrak h_c))_{N_{n_2}}.
\]

Note that, by (\ref{eqn steinberg jacquet}) and the geometric lemma, for $k>0$, any simple composition factor of $\mathrm{St}(\mathfrak h_c)_{N_k}\neq 0$ with the form $\tau_1\boxtimes \tau_2$ must satisfy that $\nu^c\rho \in \mathrm{csupp}(\tau_2)$. Now, by the geometric lemma and comparing cuspidal support at $\nu^c\rho$ , $D_{\mathfrak h_d}(\pi) \boxtimes \mathrm{St}(\mathfrak h_d)$ can only appear in the layer of the form
\[  \mathrm{Ind}^{G_{n-n_2+n_1}\times G_{n_2}}_{P_{n-n_2, n_1}\times G_{n_2}}  (D_{\mathfrak h_c}(\pi)_{N_{n_2}} \boxtimes \mathrm{St}(\mathfrak h_c))^{\phi} ,
\]
where $n=\mathrm{deg}(\rho)$ and $\phi$ twists a $G_{n-n_2}\times G_{n_2}\times G_{n_1}$-representation to a $G_{n-n_2}\times G_{n_1}\times G_{n_2}$-representation. This implies that $D_{\mathfrak h_d}\circ D_{\mathfrak h_c}(\pi)\neq 0$. Hence, by Proposition \ref{prop smaller derivative},
 \[ \mathfrak h_d \leq_d^a \mathfrak{mxpt}^a(D_{\mathfrak{mxpt}^a(\pi, c)}(\pi),d) .\] 

Now the opposite inequality follows from Lemma \ref{lem derivative transfer large} and Proposition \ref{prop uniqueness max multi}, and hence we are done.
\end{proof}

\subsection{Counting cuspidal representations from derivatives of an ascending sequence of segments}

We now prove a result on how an admissible sequence of derivatives of segments is controlled by the multiplicity of the endpoints of segments in the Zelevinsky multisegments. The main idea is to combine Proposition \ref{prop unique as segments} (construction of submodules of BZ derivatives from sequences of $D_{\Delta}$) and Lemma \ref{lem socle cosocle coarse} (Zelevinsky multisegments for submodules of BZ derivatives).

Recall that the notion $\mathrm{mult}^b(\pi, c)$ is defined in (\ref{eqn mult b values}).

\begin{lemma} \label{lem bound on ascending seq}
Let $\pi \in \mathrm{Irr}_{\rho}$. Let $\Delta_1, \ldots, \Delta_r$ in $\mathrm{Seg}_{\rho}$ be an ascending sequence of segments such that $D_{\Delta_r}\circ \ldots \circ D_{\Delta_1}(\pi)\neq 0$. Then, for any integer $c$,
\[  |\left\{i \in \left\{1, \ldots, r\right\} :  \nu^c\rho \in \Delta_i \right\}|\leq \mathrm{mult}^b(\pi, c) .
\] 
\end{lemma}

\begin{proof}
Let $\tau =D_{\Delta_r}\circ \ldots \circ D_{\Delta_1}(\pi)$. By the embedding
\[  \pi \hookrightarrow \tau \times \mathrm{St}(\Delta_r)\times \ldots \times \mathrm{St}(\Delta_1) ,
\]
\begin{align} \label{eqn cuspidal support two sides}
  \mathrm{csupp}(\pi) = \mathrm{csupp}(\tau) \cup \Delta_1 \cup \ldots \cup \Delta_r
\end{align}
as multisets.

On the other hand, by Proposition \ref{prop unique as segments}, $\tau$ is a simple submodule of $\pi^{(i)}$, where $i=\sum_{k=1}^r l_{abs}(\Delta_k)$. Now, let $\mathfrak m=\left\{\Delta_1^{\pi}, \ldots, \Delta_p^{\pi} \right\}$ be the Zelevinsky multisegment for $\pi$ i.e. $\pi \cong \langle \mathfrak m \rangle$. By Lemma \ref{lem socle cosocle coarse}, $\tau \cong \langle \mathfrak n \rangle$ for some $\mathfrak n \in \mathfrak m^{(i)}$ and so, in particular, 
\[  \mathfrak n = \Delta_1'+\ldots +\Delta_p'
\]
with each $\Delta_k'=\Delta_k^{\pi}$ or $(\Delta_k^{\pi})^-$. 

Now the formula $\mathrm{csupp}(\mathfrak n)=\mathrm{csupp}(\tau)$ with (\ref{eqn cuspidal support two sides}) implies that, as multisets,
\begin{align} \label{eqn multiplicity bound}
  \Delta_1\cup \ldots \cup \Delta_r= \bigcup_{k} b(\Delta_k^{\pi}),
\end{align}
where $k$ runs for all numbers from $1$ to $p$ such that $\Delta_k'=(\Delta_k^{\pi})^-$.

Moreover,
\begin{align} \label{eqn multiplicity bound 2}
  \bigcup_{k=1,\ldots p,\ \Delta_k'=(\Delta_k^{\pi})^-} b(\Delta_k^{\pi} ) \subset \bigcup_{k=1}^p b(\Delta_k^{\pi} ) .
\end{align}
Now the multiplicity of $\nu^c\rho$ on the LHS of (\ref{eqn multiplicity bound}) is 
\[  |\left\{i \in \left\{1, \ldots, r\right\} :  \nu^c\rho \in \Delta_i \right\}|, \]
 and the multiplicity of $\nu^c\rho$ on the RHS of (\ref{eqn multiplicity bound 2}) is $\mathrm{mult}^b(\pi,c)$. One combines the equations to obtain the lemma.
\end{proof}

\subsection{Highest derivatives by St-derivatives}

Recall that from Proposition \ref{prop unique as segments} and Lemma \ref{lem dual property}, we have that derivatives of an ascending sequence of segments can be used to construct a simple quotient of BZ derivatives. On the other hand, the highest derivative of an irreducible representation is known to be irreducible. Thus, the strategy of a proof of the following result is to show an ascending sequence of segments has the sum of the absolute lengths of segments equal to the level of that representation.

\begin{theorem} \label{thm highest derivative multiseg}
Let $\mathfrak m \in \mathrm{Mult}_{\rho}$ and let $\pi=\langle \mathfrak m \rangle$. Let 
\[   c=\mathrm{min}\left\{ e \in \mathbb Z : \nu^e\rho \cong b(\Delta) \quad \mbox{for some $\Delta$ in $\mathfrak m$ } \right\},
\]
and let 
\[ d =\mathrm{max}\left\{ f \in \mathbb Z: \nu^f\rho \cong b(\Delta)\quad \mbox{for some $\Delta$ in $\mathfrak m$} \right\}.\]
 Then 
\[ D_{\mathfrak{mxpt}^a(\pi, d)}\circ \ldots \circ D_{\mathfrak{mxpt}^a(\pi, c)}(\pi) \cong \pi^-. \]

\end{theorem}


\begin{proof}
For simplicity, let $\mathfrak h_e=\mathfrak{mxpt}^a(\pi, e)$. \\

\noindent
{\bf Step 1:} {\it Claim:} Suppose the followings:
\begin{itemize}
\item $D_{\mathfrak h_d}\circ \ldots \circ D_{\mathfrak h_c}(\pi)\neq 0$; and
\item $l_{abs}(\mathfrak h_c) +l_{abs}(\mathfrak h_{c+1})+\ldots+ l_{abs}(\mathfrak h_d) $ is equal to the level of $\pi$;
\end{itemize}
hold. Then the theorem also holds.  \\

\noindent
{\it Proof of the claim:} Let 
\[  \tau = D_{\mathfrak h_d}\circ \ldots \circ D_{\mathfrak h_c}(\pi) .\]
The first bullet and Lemma \ref{lem reformulate frobenius} give that 
\[   \pi \hookrightarrow \tau \times \mathrm{St}(\mathfrak h_d)\times \ldots \times \mathrm{St}(\mathfrak h_c), \]
and so, by Frobenius reciprocity again,
\[  \pi_{N_k} \hookrightarrow \tau \boxtimes (\mathrm{St}(\mathfrak h_d)\times \ldots \times \mathrm{St}(\mathfrak h_c) ),
\]
where $k$ is the sum of the absolute lengths of all multisegments $\mathfrak h_d, \ldots, \mathfrak h_c$. By Propositions \ref{prop sub quo Jacquet} and \ref{prop unique as segments}, $\tau$ is a submodule of $\pi^{(k)}$. By the second bullet, $\pi^{(k)}$ is the highest derivative and so it is irreducible. Thus $\tau \cong \pi^{(k)}$.  \\

\noindent
{\bf Step 2:} We first prove the first bullet in the claim. Indeed, by Lemma \ref{lem max multisegment equal} several times, we have: for any $e\geq c$,
\begin{align*}
 &\mathfrak{mxpt}^a(D_{\mathfrak{mxpt}^a(\pi,e-1)}\circ \ldots \circ D_{\mathfrak{mxpt}^a(\pi, c)}(\pi), e) \\
=& \mathfrak{mxpt}^a(D_{\mathfrak{mxpt}^a(\pi,e-2)}\circ \ldots \circ D_{\mathfrak{mxpt}^a(\pi,c)}(\pi), e) \\
=&\ldots \\
=&\mathfrak{mxpt}^a(\pi, e) 
\end{align*}
This implies that 
\[  D_{\mathfrak{mxpt}^a(\pi,e)}\circ \ldots \circ D_{\mathfrak{mxpt}^a(\pi,c)}(\pi)\neq 0 .
\]
Inductively, we prove the first bullet.

\ \\
{\bf Step 3:} We now prove the second bullet in the claim. Let $x_e$ be the total number of segments in $\mathfrak h_{c}+\ldots +\mathfrak h_{e-1}$  containing $\nu^{e}\rho$. Recall that $\mathrm{mult}^b(\pi,e)=|\mathfrak m\langle e \rangle|$.

{\it Claim 2:} For all $e$, 
\begin{align}\label{eqn key equation}
\quad  x_e+\mbox{ number of segments in $\mathfrak h_e$ = $\mathrm{mult}^b(\pi, e)$ }  .
\end{align}
\ \\

Suppose Claim 2 holds in the meanwhile. Note that the number in LHS of (\ref{eqn key equation}) is indeed the multiplicity of $\nu^e\rho$ appearing in $\cup_{\Delta \in \mathfrak m} \Delta$, where the union is taken as multisets. Hence, Claim 2 would imply that $l_{abs}(\mathfrak h_c)+\ldots +l_{abs}(\mathfrak h_d)$ is precisely
\[   (\mathrm{mult}^b(\pi,c)+\ldots +\mathrm{mult}^b(\pi,d))\cdot \mathrm{deg}(\rho) .
\]
On the other hand, by \cite[Theorem 8.1]{Ze80}, the last number is equal to the level of $\pi$. This proves the second bullet.  \\

\ \\

Now we proceed to prove Claim 2. When $e=c$, one can compute quite directly by Lemma \ref{lem simple compute epsilon factor}. Now again let
\[  \tau := D_{\mathfrak h_{e-1}} \circ \ldots \circ D_{\mathfrak h_c}(\pi) .
\]
By Proposition \ref{prop unique as segments}, 
\[ \tau \hookrightarrow \pi^{(j)}, \]
where 
\[ j=l_{abs}(\mathfrak h_c)+\ldots +l_{abs}(\mathfrak h_{e-1}). \]

Now, by  Lemma \ref{lem socle cosocle coarse}, $\tau \cong \langle \mathfrak n \rangle$ for some $\mathfrak n \in \mathfrak m^{(j)}$. By a comparison of cuspidal support,  
\[  \mathrm{csupp}(\mathfrak n) \cup \bigcup_{i=c}^{e-1}\mathrm{csupp}(\mathfrak h_i) = \mathrm{csupp}(\pi)
\]
The induction hypothesis implies that $\mathrm{mult}^b(\pi,e-1)$ (i.e. all) segments in $\mathfrak m$ with $b(\Delta) \cong \nu^{e-1}\rho$ are truncated to obtain $\mathfrak n$. In other words, 
\[  \mathrm{mult}^b(\tau,e-1)=0. \] 

By Lemma \ref{lem simple compute epsilon factor} again, 
\[  \{ \overbrace{\nu^e\rho, \ldots , \nu^e\rho }^{\mathrm{mult}^b(\tau,e) \mbox{ times}} \} \leq^a_{e}  \mathfrak{mxpt}^a(\tau, e) . 
\]
Since $D_{\mathfrak{mxpt}^a(\tau,e)}(\tau)\neq 0$, Lemma \ref{lem bound on ascending seq} implies that $\mathfrak{mxpt}^a(\tau,e)$ contains precisely $\mathrm{mult}^b(\tau,e)$-segments. Now, by Lemma \ref{lem max multisegment equal}, $\mathfrak h_e$ also contains exactly $\mathrm{mult}^b(\tau, e)$-number of segments.

Now, by  Lemma \ref{lem socle cosocle coarse} again, $D_{\mathfrak h_e}(\tau)\cong \langle \mathfrak n' \rangle$ for some $\mathfrak n' \in \mathfrak m^{(j')}$, where $j'=l_{abs}(\mathfrak h_c)+\ldots +l_{abs}(\mathfrak h_e)$. We again have:
\[  \mathrm{csupp}(D_{\mathfrak h_e}(\tau))\cup \bigcup_{i=c}^e \mathrm{csupp}(\mathfrak h_e) = \mathrm{csupp}(\pi) ,
\]
and so discussion in the previous paragraph and comparing the cuspidal support at $\nu^e\rho$, we have:
\[  x_e+|\mathfrak h_e|=\mathrm{mult}^b(\pi, e) 
\]
as desired.
\end{proof}

\section{Combinatorial derivatives via removal process} \label{sec effect hd}

This section defines the removal process and studies its various properties.

\subsection{Highest derivative multisegment $\mathfrak{hd}$} \label{ss define highest derivative}

For $\pi \in \mathrm{Irr}_{\rho}$ and  for $c \in \mathbb Z$, recall that $\mathfrak{mxpt}^a(\pi, c)$ is the $\leq^a_c$-maximal multisegment $\mathfrak m$ at $\nu^c\rho$ such that $D_{\mathfrak m}(\pi)\neq 0$ (see Proposition \ref{prop define mxpt a}). Define $\mathfrak{hd}(\pi)$ to be the multisegment
\begin{align} \label{eqn define highest derivative mult}
  \mathfrak{hd}(\pi) = \sum_{c \in \mathbb Z}  \mathfrak{mxpt}^a(\pi, c) ,
\end{align}
which is called the {\it highest derivative multisegment} for $\pi$. Note that there are only finitely many integers $c$ such that $\mathfrak{mxpt}^a(\pi, c) \neq \emptyset$. 

By definitions, we have that 
\[  \mathfrak{hd}(\pi)[c]=\mathfrak{mxpt}^a(\pi, c) .
\]
However, $\mathfrak{hd}(\pi)\langle c \rangle$ is not necessarily equal to $\mathfrak{mxpt}^b(\pi, c)$ (see Corollary \ref{cor right multi} for a precise description of $\mathfrak{mxpt}^b(\pi, c)$ in terms of $\mathfrak{hd}(\pi)$).



\subsection{Removal process} 

\begin{definition}
Given a multisegment $\mathfrak h \in \mathrm{Mult}_{\rho}$, a segment $\Delta=[a,b]_{\rho}$ is said to be {\it admissible} to $\mathfrak h$ if there exists a segment in $\mathfrak h$ of the form $[a,c]_{\rho}$ for some $c\geq b$.
\end{definition}

Suppose $\mathfrak h=\mathfrak{hd}(\pi)$ for some $\pi \in \mathrm{Irr}_{\rho}$. Then $\Delta$ is admissible to $\mathfrak h$ if and only if $D_{\Delta}(\pi)\neq 0$. This explains the above terminology of admissibility.



\begin{definition} \label{def removal process}
Let $\mathfrak h \in \mathrm{Mult}_{\rho}$. Let $\Delta=[a,b]_{\rho}$ be a segment admissible to $\mathfrak h$. The {\it removal process} on $\mathfrak h$ by $\Delta$ is a way to obtain a new multisegment, denoted by $\mathfrak r(\Delta, \mathfrak h)$, which is given by the following steps:
\begin{enumerate}
\item Choose a segment $\Delta_1$ in $\mathfrak h$ which has the shortest absolute length among all segments of the form $[a,b']_{\rho}$ for some $b'\geq b$. (In particular, $\nu^b\rho \in \Delta_1$.)
\item (Minimality condition and nesting condition) For $i\geq 2$, choose recursively segments $\Delta_i=[a_i,b_i]_{\rho}$ such that $[a_i,b_i]_{\rho}$ is the $\prec^L$-minimal segment (see Section \ref{ss ordering}) in $\mathfrak h$ satisfying $a_{i-1}< a_i$ and $b_{i-1}>b_i$. This step terminates when no further such segment can be found. Let $r$ be the index such that $\Delta_r$ is the last segment in the process. 
\item Obtain new segments $\Delta_1^{tr}, \ldots, \Delta_r^{tr}$ defined as:
\begin{itemize}
 \item for $1 \leq i \leq r-1$, $\Delta_i^{tr}=[a_{i+1}, b_i]_{\rho}$;
 \item $\Delta_r^{tr}=[b+1,b_r]_{\rho}$ (possibly an empty set).
\end{itemize} 
\item The new multisegment $\mathfrak r( \Delta, \mathfrak h)$ is defined as:
\begin{align} \label{eqn removal seq def}
 \mathfrak r(\Delta, \mathfrak h)=\mathfrak h -\sum_{i=1}^r \Delta_i+\sum_{i=1}^r\Delta_i^{tr} .
\end{align}
\end{enumerate} 

\end{definition}

\begin{remark}
\begin{enumerate}
\item[(a)] $\Delta_1$ in Step (1) above is guaranteed to exist by the assumption that $\Delta_1$ is admissible to $\mathfrak h$.
\item[(b)] If a segment $\Delta$ is not admissible to a multisegment $\mathfrak h$, we simply set $\mathfrak r(\Delta, \mathfrak h)=\infty$, where $\infty$ is just a symbol called the infinity multisegment to represent a non-admissibility situation. We also set $\mathfrak r(\Delta, \infty)=\infty$. 
\item[(c)] We also set $\mathfrak r(\emptyset, \mathfrak h)=\mathfrak h$ for any $\mathfrak h \in \mathrm{Mult}_{\rho}$. 
\end{enumerate}
\end{remark}

\begin{definition}
\begin{enumerate}
\item In the notation of Definition \ref{def removal process}, we shall call that $ \Delta_1, \ldots, \Delta_r$ form a {\it removal sequence} for $(\Delta, \mathfrak h)$. The {\it nesting condition} refers to the condition that $\Delta_{i} \subsetneq \Delta_{i-1}$ for any $i$. The {\it minimality condition} refers to the minimal choice of $\Delta_i$ in Step (2). 
\item Define $\Upsilon(\Delta, \mathfrak h)=\Delta_1$, the first segment in the removal sequence in Step (1) of Definition \ref{def removal process}.
\end{enumerate}
\end{definition}

\begin{example}
Let $\mathfrak h=\left\{ [0,4], [2,5], [2,3], [2] \right\}$. (The blue points in the graph represent those 'removed' to give $\mathfrak r(\Delta, \mathfrak h)$.)  
\begin{enumerate}
\item $\mathfrak r([0,2], \mathfrak h)=\left\{ [2,4], [2,5], [2,3] \right\}$; 
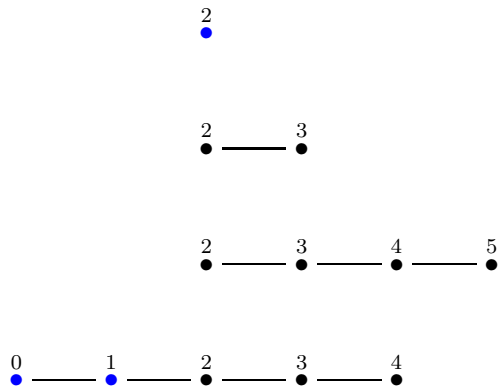
\[ \xymatrix{   &     & \stackrel{2}{{\color{blue} \bullet}}  &   &  &   \\
 &     & \stackrel{2}{\bullet} \ar@{-}[r]  & \stackrel{3}{\bullet}   &    &   \\
     &    & \stackrel{2}{\bullet} \ar@{-}[r]  & \stackrel{3}{\bullet} \ar@{-}[r]  & \stackrel{4}{\bullet} \ar@{-}[r]  & \stackrel{5}{\bullet}       \\
 \stackrel{0}{{\color{blue} \bullet}} \ar@{-}[r] & \stackrel{1}{{\color{blue}{\bullet}}} \ar@{-}[r]  & \stackrel{2}{\bullet} \ar@{-}[r]  & \stackrel{3}{\bullet} \ar@{-}[r]  & \stackrel{4}{\bullet} &   }
\]
\item $\mathfrak r([0,3], \mathfrak h)=\left\{ [2,4],[2,5], [2] \right\}$;
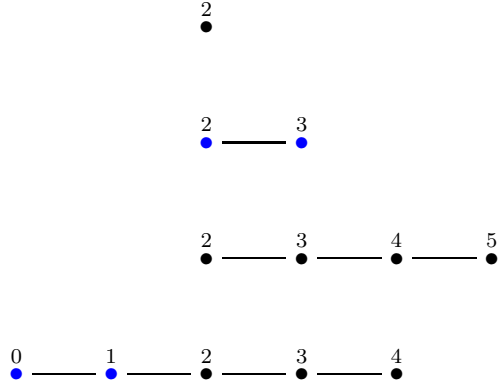
\[ \xymatrix{   &     & \stackrel{2}{\bullet}  &   &  &   \\
 &     & \stackrel{2}{{\color{blue} \bullet}} \ar@{-}[r]  & \stackrel{3}{{\color{blue} \bullet}}   &    &   \\
     &    & \stackrel{2}{\bullet} \ar@{-}[r]  & \stackrel{3}{\bullet} \ar@{-}[r]  & \stackrel{4}{\bullet} \ar@{-}[r]  & \stackrel{5}{\bullet}       \\
 \stackrel{0}{{\color{blue} \bullet}} \ar@{-}[r] & \stackrel{1}{{\color{blue}{\bullet}}} \ar@{-}[r]  & \stackrel{2}{\bullet} \ar@{-}[r]  & \stackrel{3}{\bullet} \ar@{-}[r]  & \stackrel{4}{\bullet} &   }
\]
\item $\mathfrak r([0,5], \mathfrak h)=\infty$ since $[0,5]$ is not admissible to $\mathfrak h$.
\end{enumerate}
\end{example}

\begin{example} 
\begin{enumerate}
\item Let $\mathfrak h=\left\{ [0,7], [1,4], [1,6] \right\}$. Let $\Delta=[0,5]$ and let $\Delta'=[1,4]$. The removal sequence for $(\Delta, \mathfrak h)$ is $[0,7], [1,6]$. The removal sequence for $(\Delta', \mathfrak h)$ is $[1,4]$. 
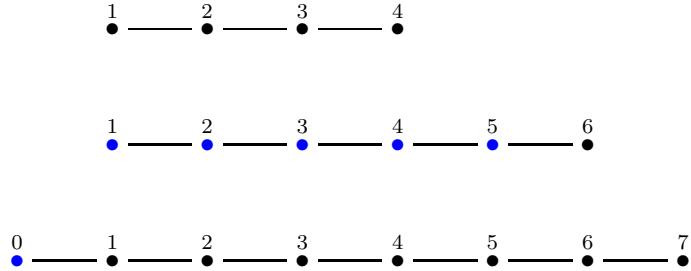
\[ \xymatrix{   
     & \stackrel{1}{\bullet} \ar@{-}[r]    & \stackrel{2}{\bullet} \ar@{-}[r]  & \stackrel{3}{\bullet} \ar@{-}[r]  & \stackrel{4}{\bullet}         &           &    \\
     & \stackrel{1}{{\color{blue} \bullet}} \ar@{-}[r]    & \stackrel{2}{{\color{blue} \bullet}} \ar@{-}[r]  & \stackrel{3}{{\color{blue} \bullet}} \ar@{-}[r]  & \stackrel{4}{{\color{blue} \bullet}}  \ar@{-}[r]      &     \stackrel{5}{{\color{blue} \bullet}}  \ar@{-}[r]     &   \stackrel{6}{\bullet}  \\
 \stackrel{0}{{\color{blue} \bullet}} \ar@{-}[r] & \stackrel{1}{ \bullet} \ar@{-}[r]  & \stackrel{2}{\bullet} \ar@{-}[r]  & \stackrel{3}{\bullet} \ar@{-}[r]  & \stackrel{4}{\bullet}  \ar@{-}[r]&  \stackrel{5}{\bullet} \ar@{-}[r] &  \stackrel{6}{\bullet} \ar@{-}[r] &  \stackrel{7}{\bullet}  }
\]
The blue points represent the points removed for $\mathfrak r(\Delta, \mathfrak h)$ and hence give the corresponding removal sequence for $(\Delta, \mathfrak h)$. In this case, we have $\mathfrak r(\Delta, \mathfrak h)=\left\{[1,4],[1,7],[6]\right\}$.
\item Let $\mathfrak h=\left\{ [0,7], [1,5], [1,6] \right\}$. Let $\Delta=[0,5]$ and let $\Delta'=[1,4]$. The removal sequence for $(\Delta, \mathfrak h)$ is $[0,7], [1,5]$. The removal sequence for $(\Delta', \mathfrak h)$ is $[1,5]$.  
\item Let $\mathfrak h=\left\{ [0,7],[1,5],[1,8] \right\}$. Let $\Delta=[0,5]$ and let $\Delta'=[1,4]$. The removal sequence for $(\Delta, \mathfrak h)$ is $[0,7], [1,5]$, and the removal sequence for $(\Delta', \mathfrak h)$ is $[1,5]$. 
\end{enumerate}
\end{example}

\subsection{Properties of removal process}

A simple but useful computation is the following:

\begin{lemma} \label{lem step step compute} (Removal of a cuspidal point at one time)
Let $\mathfrak h \in \mathrm{Mult}_{\rho}$. Let $\Delta$ be a non-empty segment admissible to $\mathfrak h$. Let
\[ \mathfrak h^*=\mathfrak h-\left\{  \Upsilon(\Delta, \mathfrak h) \right\}+\left\{ {}^-\Upsilon(\Delta, \mathfrak h) \right\}.\] Then 
\[  \mathfrak r(\Delta, \mathfrak h) = \mathfrak r({}^-\Delta, \mathfrak h^*) .
\]
\end{lemma}

\begin{proof}
Let $\widetilde{\Delta}=\Upsilon(\Delta, \mathfrak h)$. Write $\Delta=[a,b]_{\rho}$. Suppose $\mathfrak h[a+1]$ has no segment $\overline{\Delta}$ satisfying $b(\Delta) \leq b(\overline{\Delta})$. Then $\Upsilon({}^-\Delta, \mathfrak h^*)={}^-\widetilde{\Delta}_1$. Since the only difference between $\mathfrak h$ and $\mathfrak h^*$ is on that one segment, one checks that the remaining segments in the sequences are picked in the same way by the minimality and nesting conditions. This gives $\mathfrak r({}^-\Delta, \mathfrak h^*)=\mathfrak r(\Delta, \mathfrak h)$.

Suppose $\mathfrak h[a+1]$ has some segments $\overline{\Delta}$ satisfying $b(\Delta)\leq b(\overline{\Delta})$, and let $\overline{\Delta}^*$ be a shortest such segment. We further divide into two cases:
\begin{itemize}
\item Case 1: Suppose $\overline{\Delta}^* \subsetneq {}^-\widetilde{\Delta}$. We have that $\Upsilon({}^-\Delta, \mathfrak h^*)$ is $\overline{\Delta}^*$, coinciding with the second segment in the removal sequence for $(\Delta, \mathfrak h)$. By the minimality and nesting conditions, the subsequent segments in the removal sequence for $({}^-\Delta, \mathfrak h^*)$ are the same as those starting from the third one in the removal sequence for $(\Delta, \mathfrak h)$. Thus, from (\ref{eqn removal seq def}), $\mathfrak r(\Delta, \mathfrak h)=\mathfrak r({}^-\Delta, \mathfrak h^*)$. 
\item Case 2: Suppose $\overline{\Delta}^* \not\subset {}^-\widetilde{\Delta}$ or $\overline{\Delta}^*={}^-\widetilde{\Delta}$.  $\Upsilon({}^-\Delta ,\mathfrak h^*)={}^-\widetilde{\Delta}_1$, and, by the minimality and nesting conditions, the segments in the removal sequence for $({}^-\Delta, \mathfrak h^*)$ are those starting from the second one in the removal sequence for $(\Delta, \mathfrak h)$. Thus, from (\ref{eqn removal seq def}), $\mathfrak r(\Delta, \mathfrak h)=\mathfrak r({}^-\Delta, \mathfrak h^*)$. 
\end{itemize}
\end{proof}

We prove some further properties in Lemmas \ref{lem does not remove lower segment} to \ref{lem independent of order}. One may compare with properties in the derivative side such as Lemmas \ref{lem comm derivative 1} and \ref{lem non-zero composite}.

\begin{lemma} \label{lem does not remove lower segment} (No effect on previous segments)
Let $\mathfrak h \in \mathrm{Mult}_{\rho}$. Let $\Delta=[a,b]_{\rho} \in \mathrm{Seg}_{\rho}$ be a non-empty segment admissible to $\mathfrak h$. Then for any $a'<a$, $\mathfrak h[a']=\mathfrak r(\Delta, \mathfrak h)[a']$. 
\end{lemma}

\begin{proof}
This follows directly form Definition \ref{def removal process} since those segments do not involve in the removal process.  
\end{proof}

\begin{lemma} (Removing a whole segment in $\mathfrak h$) \label{lem remove whole seg}
Let $\mathfrak h \in \mathrm{Mult}_{\rho}$. Let $\Delta \in \mathfrak h$. Then
\[  \mathfrak r(\Delta, \mathfrak h) =\mathfrak h - \Delta  .
\]
\end{lemma}

\begin{proof}
Write $\Delta=[a,b]_{\rho}$. Note that $\Upsilon(\Delta, \mathfrak h)=\Delta$. The nesting property guarantees that there is no other segment in the removal sequence for $(\Delta, \mathfrak h)$.
\end{proof}

\begin{lemma} \label{lem first segment commute} (First segments of removal sequences from two segments of the same beginning point)
Let $\mathfrak h \in \mathrm{Mult}_{\rho}$. Let $\Delta, \Delta' \in \mathrm{Seg}_{\rho}$ with $a(\Delta) \cong a(\Delta')$. Then 
\[  \left\{ \Upsilon(\Delta, \mathfrak h) + \Upsilon(\Delta', \mathfrak r(\Delta, \mathfrak h)) \right\} = \left\{ \Upsilon(\Delta', \mathfrak h)+\Upsilon(\Delta, \mathfrak r(\Delta', \mathfrak h)) \right\} .
\]
\end{lemma}
The above lemma is straightforward from definitions and we omit a proof.

\begin{lemma}(Removal sequence involving the largest end point) \label{lem removal seq largest end pt}
Let $\mathfrak h \in \mathrm{Mult}_{\rho}$. Let $\Delta \in \mathrm{Seg}_{\rho}$ be non-empty and admissible to $\mathfrak h$. Let $c$ be the largest integer such that $\mathfrak h\langle c \rangle \neq 0$. If one of the segments in the removal sequence for $(\mathfrak n,\mathfrak h)$ is in $\mathfrak h\langle c \rangle$, then $\Upsilon(\Delta, \mathfrak h) \in \mathfrak h\langle c \rangle$. 
\end{lemma}

\begin{proof}
This follows from the nesting property in the removal sequence.
\end{proof}

\begin{lemma} \label{lem independent of order} (Commutativity for unlinked segments)
Let $\mathfrak h \in \mathrm{Mult}_{\rho}$. Let $\Delta, \Delta' \in \mathrm{Seg}_{\rho}$ be unlinked segments. Suppose $\mathfrak r(\Delta, \mathfrak h) \neq \infty$ and $\mathfrak r(\Delta', \mathfrak r(\Delta, \mathfrak h)) \neq \infty$. Then 
\[  \mathfrak r(\Delta', \mathfrak r(\Delta, \mathfrak h)) = \mathfrak r(\Delta, \mathfrak r(\Delta', \mathfrak h)) .
\]
\end{lemma}

\begin{proof}

We shall prove by an induction on the sum of absolute lengths of all segments in $\mathfrak h$. If one of $\Delta, \Delta'$ is an empty set, it is straightforward from definitions. We now assume both are non-empty sets. By switching the labellings if necessary, we may and shall assume that $a(\Delta')\geq a(\Delta)$. Let $\widetilde{\Delta}_1=\Upsilon(\Delta, \mathfrak h)$. 

{\bf Case 1:} $a(\Delta)\not\cong a(\Delta'), \nu^{-1}a(\Delta')$. Hence $a(\Delta')>a(\Delta)$. Now we consider 
\[  \mathfrak h^*=\mathfrak h-\left\{ \widetilde{\Delta}_1 \right\}+\left\{{}^-\widetilde{\Delta}_1 \right\} .
\]
and so, by Lemma \ref{lem step step compute}, $\mathfrak r(\Delta, \mathfrak h)=\mathfrak r({}^-\Delta , \mathfrak h^*)$. Thus, 
\begin{align} \label{eq removal two 1}
 \quad \mathfrak r(\Delta', \mathfrak r( \Delta, \mathfrak h))=\mathfrak r(\Delta', \mathfrak r({}^-\Delta, \mathfrak h^*)) .
\end{align}
Now using Lemma \ref{lem does not remove lower segment} and the assumption in this specific case, we still have $\Upsilon(\Delta, \mathfrak r(\Delta', \mathfrak h))=\widetilde{\Delta}_1$. Hence, we have that: by Lemma \ref{lem step step compute} again,
\begin{align} \label{eqn segment equals}
  \mathfrak r(\Delta, \mathfrak r(\Delta',\mathfrak h))=\mathfrak r({}^-\Delta, \mathfrak r(\Delta', \mathfrak h^*)) ,
\end{align}
where we also use $\mathfrak r(\Delta', \mathfrak h)-\left\{ \widetilde{\Delta}_1 \right\}+\left\{ {}^-\widetilde{\Delta}_1\right\} =\mathfrak r(\Delta', \mathfrak h^*)$ by Lemma \ref{lem does not remove lower segment}.

Now,
\[  \mathfrak r(\Delta', \mathfrak r(\Delta, \mathfrak h))=\mathfrak r(\Delta', \mathfrak r( {}^-\Delta, \mathfrak h^*)) =\mathfrak r({}^-\Delta, \mathfrak r( \Delta', \mathfrak h^*)) =\mathfrak r(\Delta, \mathfrak r(\Delta', \mathfrak h)),
\]
where the middle equality follows from the inductive case. Hence, we are done.

{\bf Case 2:} $a(\Delta) \cong a(\Delta')$. Let
\[  \mathfrak h^{**}=\mathfrak h-\Upsilon(\Delta, \mathfrak h)-\Upsilon(\Delta',\mathfrak r(\Delta, \mathfrak h))+{}^-\Upsilon(\Delta, \mathfrak h)+{}^-\Upsilon(\Delta',\mathfrak r(\Delta, \mathfrak h)).
\]
We use Lemma \ref{lem step step compute} twice and combine with Lemma \ref{lem first segment commute} to obtain:
\[  \mathfrak r(\Delta', \mathfrak r(\Delta, \mathfrak h))=\mathfrak r({}^-\Delta', \mathfrak r({}^-\Delta, \mathfrak h^{**}))
\]
and
\[  \mathfrak r(\Delta, \mathfrak r(\Delta', \mathfrak h))=\mathfrak r({}^-\Delta, \mathfrak r({}^-\Delta', \mathfrak h^{**})) .
\]
Then the equality follows from the induction.

{\bf Case 3:} $a(\Delta)\cong \nu^{-1} a(\Delta')$. We further divide into two more cases:
\begin{enumerate}
\item There is a segment $\widehat{\Delta}$ in $\mathfrak h$ such that $a(\widehat{\Delta})\cong \nu a(\Delta)$ and $\Delta' \subset \widehat{\Delta} \subsetneq {}^-\widetilde{\Delta}_1$. In such case, one has the equalities (\ref{eq removal two 1}) and (\ref{eqn segment equals}) as in Case 1.


\item There is no such segment in the above case. Let $\overline{\Delta}_1=\Upsilon(\Delta', \mathfrak h)$. Let $\mathfrak h^*=\mathfrak h-\left\{ \overline{\Delta}_1 \right\}+\left\{{}^-\overline{\Delta}_1 \right\}$. Then, by Lemma \ref{lem step step compute}, we have that
\[  \mathfrak r(\Delta', \mathfrak h)=\mathfrak r({}^-\Delta', \mathfrak h^*) .
\]
Hence, 
\begin{align} \label{eqn removal truncate first}
  \mathfrak r(\Delta, \mathfrak r(\Delta', \mathfrak h))=\mathfrak r(\Delta, \mathfrak r({}^-\Delta', \mathfrak h^*)) . 
\end{align}
On the other hand, by the nesting property for the removal sequence for $(\Delta, \mathfrak h)$ and the assumption in this case, $\overline{\Delta}_1$ cannot be involved in the removal sequence for $(\Delta, \mathfrak h)$. Hence,
\[  \mathfrak r(\Delta, \mathfrak h)-\left\{\overline{\Delta}_1\right\}+\left\{{}^-\overline{\Delta}_1\right\}= \mathfrak r(\Delta, \mathfrak h^*) .
\]
By using the assumption in this case, we still have that 
\[ \Upsilon(\Delta', \mathfrak r(\Delta, \mathfrak h^*))=\overline{\Delta}_1 .\]
 Thus, by Lemma \ref{lem step step compute} again,
\begin{align} \label{eqn removal truncate later}
  \mathfrak r(\Delta', \mathfrak r(\Delta, \mathfrak h))=\mathfrak r({}^-\Delta', \mathfrak r(\Delta, \mathfrak h^*)). 
\end{align}
Now, using (\ref{eqn removal truncate first}) and (\ref{eqn removal truncate later}) and the inductive case, we have that:
\[   \mathfrak r(\Delta', \mathfrak r(\Delta, \mathfrak h))=\mathfrak r(\Delta, \mathfrak r(\Delta', \mathfrak h)) .
\]

\end{enumerate}

\end{proof}

\subsection{Derivative resultant multisegments}

\begin{definition} \label{def admissible multiseg}Let $\mathfrak h \in \mathrm{Mult}_{\rho}$. Let $\mathfrak n=\left\{ \Delta_1, \ldots, \Delta_k \right\} \in \mathrm{Mult}_{\rho}$ with segments in an ascending order.
\begin{enumerate}
\item  We say that $\mathfrak n$ is {\it admissible} to $\mathfrak h$, if $\Delta_i$ is admissible to 
\[  \mathfrak r(\Delta_{i-1}, \ldots \mathfrak r(\Delta_1, \mathfrak h)\ldots )
\]
for all $i=1, \ldots, k$. By Lemma \ref{lem independent of order}, it is independent of the choice of an ascending order. We also consider $\emptyset$ to be admissible to any $\mathfrak h \in \mathrm{Mult}_{\rho}$.
\item Define
\[ \mathfrak r(\mathfrak n, \mathfrak h)= \mathfrak r(\left\{\Delta_1, \ldots, \Delta_k \right\}, \mathfrak h) := \mathfrak r(\Delta_k, \mathfrak r(\Delta_{k-1}, \ldots \mathfrak r(\Delta_1, \mathfrak h)\ldots). \]
\end{enumerate}
For any multisegment $\mathfrak n$ admissible to $\mathfrak h$, we say $\mathfrak r(\mathfrak n, \mathfrak h)$ to be a {\it derivative resultant multisegment} for $\mathfrak h$.
\end{definition}

\subsection{Constructing some derivative resultant multisegments}


The following technical lemma constructs some derivative resultant multisegments from some known ones, which is an essential part in the proof of Theorem \ref{thm isomorphic derivatives} below.

\begin{lemma} \label{lem shrinking mutli 0} 
Let $\mathfrak h \in \mathrm{Mult}_{\rho}$.  Let $\mathfrak n \in \mathrm{Mult}_{\rho}$ be admissible to $\mathfrak h$. Let $c \in \mathbb Z$ such that, for any $i\geq 1$, $\nu^{c+i}\rho$ is not in any segment of $\mathfrak n$. Let $\mathfrak s=\mathfrak r(\mathfrak n, \mathfrak h)$.  Recall that $\mathfrak h\langle c\rangle$ (resp. $\mathfrak s\langle c\rangle$) is the submultisegment of $\mathfrak h$ (resp. $\mathfrak s$) containing all the segments $\Delta$ in $\mathfrak h$ (resp. $\mathfrak s$) satisfying $b(\Delta)\cong \nu^c \rho$. We further assume that 
\[\mathfrak s\langle e \rangle =\mathfrak h\langle e \rangle\]
 for any $e\geq c+1$. Then 
\[  \mathfrak s - \mathfrak s\langle c \rangle + \mathfrak h\langle c\rangle
\]
is also a derivative resultant multisegment for $\mathfrak h$. 
\end{lemma}

\begin{proof}
We shall prove by an induction on $l_{abs}(\mathfrak h)$. We write the segments in $\mathfrak n$ in the following ascending order (see Section \ref{ss ordering} for the notion $\preceq^L$): 
\begin{align} \label{eqn order seg in shrink}
\quad \Delta_1 \preceq^L \ldots \preceq^L \Delta_p. 
\end{align}
 When $l_{abs}(\mathfrak h)=0$, it is clear.

Let $\widetilde{\Delta}=\Upsilon(\Delta_1, \mathfrak h)$. We first consider the case that $\widetilde{\Delta} \notin \mathfrak h\langle c \rangle$. Note that, by the definition of $\mathfrak r$,
\[  \mathfrak r(\mathfrak n-\Delta_1, \mathfrak r(\Delta_1, \mathfrak h)) =\mathfrak r(\mathfrak n, \mathfrak h) .
\]
Let $\mathfrak h'=\mathfrak r(\Delta_1, \mathfrak h)$. The assumptions in the lemma imply that $\widetilde{\Delta} \in \mathfrak h\langle e \rangle$ for some $e<c$ and then the nesting property with $\widetilde{\Delta} \notin \mathfrak h\langle c \rangle$ imply that $\mathfrak h\langle c\rangle =\mathfrak h'\langle c\rangle$.  Now the induction hypothesis with (\ref{eqn order seg in shrink}) gives that
\[ \mathfrak s-\mathfrak s\langle c\rangle +\mathfrak h\langle c\rangle=\mathfrak s -\mathfrak s\langle c \rangle +\mathfrak h'\langle c\rangle 
\]
is still a derivative resultant multisegment for $\mathfrak h'$ i.e. for some multisegment $\widetilde{\mathfrak n}$,
\[  \mathfrak r(\widetilde{\mathfrak n}, \mathfrak h')=\mathfrak s-\mathfrak s\langle c\rangle +\mathfrak h\langle c\rangle .
\]
It remains to observe from (\ref{eqn order seg in shrink}) that we still have
\[  \mathfrak r(\widetilde{\mathfrak n}+\Delta_1, \mathfrak h)=\mathfrak r(\widetilde{\mathfrak n}, \mathfrak h') 
\]
since $\Delta_1$ can still be the first segment for an ascending order for $\widetilde{\mathfrak n}+\Delta_1$ and so $\mathfrak r(\widetilde{\mathfrak n}+\Delta_1, \mathfrak h)=\mathfrak r(\widetilde{\mathfrak n}, \mathfrak r(\Delta_1, \mathfrak h))$. 

We now consider the case that $\widetilde{\Delta} \in \mathfrak h\langle c \rangle$. Let 
\[  \mathfrak h^*=\mathfrak h-\left\{ \widetilde{\Delta} \right\}+\left\{ {}^-\widetilde{\Delta} \right\} ,
\] 
\[ \mathfrak n^* =\mathfrak n-\left\{ \Delta_1 \right\}+\left\{ {}^-\Delta_1 \right\} .\]
By Lemma \ref{lem step step compute},
\[  \mathfrak r(\Delta_1, \mathfrak h)=\mathfrak r({}^-\Delta_1, \mathfrak h^*)
\]
and so, by Lemma \ref{lem independent of order}, 
\[   \mathfrak r(\mathfrak n, \mathfrak h)=\mathfrak r(\mathfrak n^*, \mathfrak h^*) .
\]

 Now, by the induction hypothesis, we have that $\mathfrak s- \mathfrak s\langle c \rangle +\mathfrak h^*\langle c \rangle$ is a derivative resultant multisegment for $\mathfrak h^*$, say there exists a multisegment $\widetilde{\mathfrak n}$ such that 
\begin{align}\label{eqn induction shrinking derivative}
  \mathfrak r(\widetilde{\mathfrak n}, \mathfrak h^*)=\mathfrak s-\mathfrak s\langle c \rangle+\mathfrak h^*\langle c \rangle .
\end{align}

We now claim: \\
{\it Claim 1:}  $\mathfrak r(\widetilde{\mathfrak n}, \mathfrak h^*)[a]=\mathfrak h^*[a]$, where $a$ is the integer such that $a(\widetilde{\Delta})\cong \nu^a\rho$. Moreover $\widetilde{\mathfrak n}[a]=\emptyset$.  \\
{\it Proof of claim 1:} In the removal sequences for $(\widetilde{\mathfrak n}, \mathfrak h^*)$, all the first segments involved are $\widetilde{\Delta}$ by the minimality in the removal process and the assumption that $\mathfrak s\langle e\rangle=\mathfrak h\langle e\rangle$ for all $e \geq c+1$. Thus, we have that,
\[ \mathfrak s=\mathfrak r(\mathfrak n, \mathfrak h)=\mathfrak r(\mathfrak n-\mathfrak n[a]+{}^-(\mathfrak n[a]), \mathfrak h-\left\{\widetilde{\Delta}, \ldots, \widetilde{\Delta} \right\}+\left\{ {}^-\widetilde{\Delta}, \ldots, {}^-\widetilde{\Delta} \right\}) ,
\]
where both $\widetilde{\Delta}$ and ${}^-\widetilde{\Delta}$ appear $|\mathfrak n[a]|$-times. Thus, by Lemma \ref{lem does not remove lower segment}, 
\[ \mathfrak s[a]=\mathfrak h[a]-\overbrace{\{ \widetilde{\Delta}, \ldots, \widetilde{\Delta} \}}^{\mbox{$|\mathfrak n[a]|$-times}}=\mathfrak h^*[a]-\overbrace{\{ \widetilde{\Delta}, \ldots, \widetilde{\Delta}\} }^{\mbox{$|\mathfrak n[a]|-1$ times} }.
\]
Hence, $\mathfrak s[a]-\mathfrak s[a]\langle c \rangle=\mathfrak h[a]-\mathfrak h[a]\langle c\rangle=\mathfrak h^*[a]-\mathfrak h^*[a]\langle c \rangle$ and so $(\mathfrak s-\mathfrak s\langle c \rangle +\mathfrak h^*\langle c \rangle)[a]=\mathfrak h^*[a]$. This proves the claim by (\ref{eqn induction shrinking derivative}). This proves the first assertion in the claim.

For the second assertion, since $\Delta_1$ is minimal under $\preceq^L$, all the segments in $\mathfrak n$ do not contain $\nu^d\rho$ for some $d < a$. Thus, all the segments in $\widetilde{\mathfrak n}$ do not contain $\nu^d\rho$ for some $d<a$. Now, combining with the first assertion, we have the second assertion.  \\

\noindent	
{\it Claim 2:} Any segment in the removal sequence for $(\widetilde{\mathfrak n}, \mathfrak h^*)$ does not involve ${}^-\widetilde{\Delta}$. \\

\noindent
{\it Proof of Claim 2:} This follows from (\ref{eqn induction shrinking derivative}) that $\mathfrak r(\widetilde{\mathfrak n}, \mathfrak h^*)\langle c \rangle =\mathfrak h^*\langle c\rangle$. Now the claim follows from the definition of the removal process. \\

\noindent
We now return to the proof. By Lemma \ref{lem does not remove lower segment} and Claim 1, any segment in the removal sequence for $(\widetilde{\mathfrak n}, \mathfrak h)$ does not involve $\widetilde{\Delta}$. Thus, with Claim 2, we now have that 
\[ \mathfrak r(\widetilde{\mathfrak n}, \mathfrak h)=\mathfrak r(\widetilde{\mathfrak n}, \mathfrak h^*)+\left\{ \widetilde{\Delta} \right\}-\left\{{}^- \widetilde{\Delta} \right\} . \]
Hence, by (\ref{eqn induction shrinking derivative}),
\[  \mathfrak r(\widetilde{\mathfrak n}, \mathfrak h)=\mathfrak s-\mathfrak{s}\langle c \rangle +\mathfrak h^*\langle c \rangle+\left\{ \widetilde{\Delta} \right\}-\left\{ {}^-\widetilde{\Delta} \right\}=\mathfrak s-\mathfrak s\langle c \rangle+\mathfrak h\langle c \rangle
\]
is still a derivative resultant multisegment as desired.
\end{proof}

\begin{example}
Let $\mathfrak h=\left\{ [1,5], [2,4], [4,5] \right\}$.
\begin{enumerate}
\item Let $\Delta=[1,3]$. Then $\mathfrak r(\Delta, \mathfrak h)=\left\{ [2,5], [4], [4,5] \right\}$. Then $\left\{ [1,5], [4],[4,5] \right\}$ is also a derivative resultant multisegment by Lemma \ref{lem shrinking mutli 0}. One can also check that $\mathfrak r([2,3],\mathfrak h)=\left\{ [1,5], [4],[4,5] \right\}$. This illustrates Lemma \ref{lem shrinking mutli 0}  for $c=5$.
\[ \xymatrix{  
 &     &   &  & \stackrel{4}{\bullet}  \ar@{-}[r]  & \stackrel{5}{\bullet}     \\
     &    & \stackrel{2}{{\color{blue}\bullet}} \ar@{-}[r]  & \stackrel{3}{{\color{blue} \bullet}} \ar@{-}[r]  & \stackrel{4}{\bullet}   &       \\
  & \stackrel{1}{{\color{red}{\bullet}}} \ar@{-}[r]  & \stackrel{2}{\bullet} \ar@{-}[r]  & \stackrel{3}{\bullet} \ar@{-}[r]  & \stackrel{4}{\bullet} \ar@{-}[r] &  \stackrel{5}{\bullet}  }
\]
Here those red and blue points are removed to obtain $\mathfrak r([1,3], \mathfrak h)$, and the red point is added to obtain a new derivative resultant multisegment, which come from $\mathfrak r([2,3],\mathfrak h)$.
\item Let $\mathfrak n=\left\{ [1,3], [2] \right\}$. Then $\mathfrak r(\mathfrak n, \mathfrak h)=\left\{ [3,5], [4], [4,5] \right\}$. Then $\left\{ [1,5], [4], [4,5] \right\}$ is a derivative resultant multisegment.  This again illustrates Lemma \ref{lem shrinking mutli 0}  for $c=5$.
\end{enumerate}
\end{example}


\begin{lemma} \label{lem shrinking muliseg} 
We use notations in the previous lemma. Let $\mathfrak n' \in \mathrm{Mult}_{\rho}$ such that 
\[  \mathfrak r(\mathfrak n', \mathfrak h)=\mathfrak s-\mathfrak s\langle c\rangle +\mathfrak h\langle c \rangle.\]
Then 
\[  \mathfrak r(\mathfrak h\langle c \rangle+\mathfrak n', \mathfrak h)= \mathfrak s - \mathfrak s\langle c \rangle .
\]
\end{lemma}

\begin{proof}
Note that 
\[ \mathfrak r(\mathfrak h\langle c \rangle+\mathfrak n', \mathfrak h)=\mathfrak r(\mathfrak h\langle c \rangle, \mathfrak r(\mathfrak n', \mathfrak h))=\mathfrak r(\mathfrak h\langle c \rangle, \mathfrak s-\mathfrak s\langle c \rangle+\mathfrak h\mathfrak \langle c\rangle) .\]
Then the lemma follows from Lemma \ref{lem remove whole seg}.
\end{proof}

\section{Comparison between $\mathfrak r(\Delta, \pi)$ and $D_{\Delta}(\pi)$} \label{s comparison theorem}

The main result of this section is Theorem \ref{thm effect of Steinberg}, which approximates $\mathfrak{hd}(D_{\Delta}(\pi))$ in terms of $\mathfrak{hd}(\pi)$ for $\pi \in \mathrm{Irr}_{\rho}$.

\subsection{Effect of St-derivatives}

We shall now compute the effect of St-derivatives on the invariant $\varepsilon_{\Delta}$. We need a preparation lemma first, which allows one to proceed inductively. For a segment $\Delta=[a,b]_{\rho}$, define ${}^+\Delta=[a-1,b]_{\rho}$. 

\begin{lemma} \label{lem change in epsilon}
Let $\pi \in \mathrm{Irr}_{\rho}$. Let $k=\varepsilon_c(\pi)$. Let $\widetilde{\pi}=D_c^k(\pi)$. For $\Delta=[a, b]_{\rho}$ with $b>0$,
\begin{enumerate}
\item if $a>c+1$, then 
\[ \varepsilon_{\Delta}(\widetilde{\pi})=\varepsilon_{\Delta}(\pi) ;\]
\item if $a=c+1$, then 
\[ \varepsilon_{\Delta}(\widetilde{\pi})=\varepsilon_{\Delta}(\pi)+\varepsilon_{{}^+\Delta}(\pi)
\]
\end{enumerate}
\end{lemma}
\begin{proof}
We consider (1). Let $p=\varepsilon_{\Delta}(\pi)$. Then we have 
\[                     \pi \hookrightarrow   D^p_{\Delta}(\pi) \times (\mathrm{St}(\Delta))^{\times p}
\]
This then gives a non-zero map:
\begin{align} \label{eqn in lem change epsilon}
  \widetilde{\pi} \boxtimes  (\nu^c\rho)^{\times k} \hookrightarrow \pi_{N_i}  \hookrightarrow  ( D^p_{\Delta}(\pi) \times (\mathrm{St}(\Delta))^{\times p})_{N_{i}} ,
\end{align}
where $i=k\cdot \mathrm{deg}(\rho)$. Since $\nu^c \rho$ is not in $\Delta$, $\nu^c\rho$ must come from the cuspidal support of $D^p_{\Delta}(\pi)$. Hence, the only one layer from the geometric lemma on $( D^p_{\Delta}(\pi) \times (\mathrm{St}(\Delta))^{\times p})_{N_i}$ that contributes (\ref{eqn in lem change epsilon}) takes the form:
\[  \mathrm{Ind}_P^{G_{n-i}\times G_i} (D^p_{\Delta}(\pi)_{N_{i}} \boxtimes \mathrm{St}(\Delta)^{\times p})^{\phi},
\]
where
\begin{itemize}
\item $P=P_{n-i-j,j} \times G_{i}$ and $j=p\cdot l_{abs}(\Delta)$, and $n=\mathrm{deg}(\pi)$;
\item $\phi$ is a twist sending $G_{n-i-j}\times G_i \times G_j$ representations $G_{n-i-j}\times G_j \times G_i$-representations.
\end{itemize}
 Applying Frobenius reciprocity, we have:
\[    \widetilde{\pi}_{N_j} \boxtimes (\nu^c\rho)^{\times k} \rightarrow  ((D^p_{\Delta}(\pi))_{N_{i}} \boxtimes \mathrm{St}(\Delta)^{\times p})^{\phi} .
\]
Thus, $\widetilde{\pi}_{N_j}$ must then have a simple quotient of the form $\tau \boxtimes \mathrm{St}(\Delta)^{\times p}$ for some $\tau \in \mathrm{Irr}$. This implies that $D_{\Delta}^p(\widetilde{\pi}) \neq 0$ and so $ \varepsilon_{\Delta}(\widetilde{\pi}) \geq p$.


Let $p' =\varepsilon_{\Delta}(\widetilde{\pi})$. We have an embedding:
\[  \pi \hookrightarrow  \widetilde{\pi}\times (\nu^c\rho)^{\times k} \hookrightarrow  D_{\Delta}^{p'}(\widetilde{\pi}) \times  \mathrm{St}(\Delta)^{\times p'} \times (\nu^c\rho)^{\times k} \cong D_{\Delta}^{p'}(\widetilde{\pi})  \times (\nu^c\rho)^{\times k} \times \mathrm{St}(\Delta)^{\times p'} ,
\]
where the last isomorphism follows from (\ref{eqn steinberg jacquet}). Hence, by Frobenius reciprocity, $D^{p'}_{\Delta}(\pi)\neq 0$ and hence $p=\varepsilon_{\Delta}(\pi) \geq p'$. Thus $p=p'$ and this proves (1).

We now consider (2). We shall write $\mathfrak h_c=\mathfrak{mxpt}^a(\pi, c)$ and $\mathfrak h_{c+1}=\mathfrak{mxpt}^a(\pi,c+1)$. Note that by the definition of a maximal multisegment and $\varepsilon_c$, $\mathfrak h_c$ has exactly $p$ segments. By Lemma \ref{lem max multisegment equal}, we have an embedding:
\[   \pi \hookrightarrow \omega \times \mathrm{St}(\mathfrak h_{c+1}) \times \mathrm{St}(\mathfrak h_c) 
\]
and so we also have:
\[ \pi \hookrightarrow \omega \times \mathrm{St}(\mathfrak h_{c+1}+{}^-(\mathfrak h_c)) \times (\nu^c \rho)^{\times k} .
\]
(see Section \ref{ss notation derivatives} for the notation ${}^-(\mathfrak h_c)$). Thus, by Corollary \ref{cor embeding derivative},
\[  \widetilde{\pi} \hookrightarrow \omega \times \mathrm{St}(\mathfrak h_{c+1}+{}^-(\mathfrak h_c)) .
\]
 Then, by Propositions \ref{prop define mxpt a} and \ref{prop smaller derivative},
\[  \varepsilon_{\Delta}(\widetilde{\pi}) \geq \varepsilon_{\Delta}(\pi)+\varepsilon_{{}^+\Delta}(\pi). 
\]

To prove the opposite inequality, suppose it fails for some $\Delta=[c+1, b]_{\rho}$. Let $p'=\varepsilon_{\Delta}(\widetilde{\pi})$ again. We have 
\[  \pi \hookrightarrow  D_{\Delta}^{p'}(\widetilde{\pi}) \times (\mathrm{St}(\Delta))^{\times p'} \times (\nu^c\rho)^{\times k} ,
\]
which implies that 
\[  \pi \hookrightarrow D_{\Delta}^{p'}(\widetilde{\pi}) \times \omega,
\]
for some composition factor $\omega$ in $\mathrm{St}(\Delta)^{\times p'} \times (\nu^c\rho)^{\times k}$. Since the only possible segments appearing in the multisegment for $\omega$ include $[c, b]_{\rho}$, $[c+1, b]_{\rho}$ or $[c]_{\rho}$, we must have that, via the Langlands classification for $\omega$,
\[  \pi \hookrightarrow D_{\Delta}^{p'}(\widetilde{\pi}) \times (\nu^c\rho)^r \times \mathrm{St}([c, b]_{\rho})^s \times \mathrm{St}([c+1, b]_{\rho})^t ,
\]
where $s+t=\varepsilon_{\Delta}(\widetilde{\pi})=p'$. Hence, by our assumption, we have either:
\[  s > \varepsilon_{\Delta}(\pi)  \quad \mbox{ or } \quad t > \varepsilon_{{}^+\Delta}(\pi) .\]
However, one applies Frobenius reciprocity and obtains a contradiction to the definition of $\varepsilon_{\Delta}(\pi)$ or $\varepsilon_{{}^+\Delta}(\pi)$, as desired.
\end{proof}




We now prove a lemma which allows one to prove some statements (e.g. Theorem \ref{thm effect of Steinberg}(2)) inductively.

\begin{lemma} \label{lem transitive embedding on max k}
Let $\pi \in \mathrm{Irr}_{\rho}$. Let $c$ be an integer such that $\varepsilon_c(\pi)\neq 0$. Let $\Delta=[c, b]_{\rho}$ be a segment with $D_{\Delta}(\pi)\neq 0$. Let $\widetilde{\pi}=D_{c}^k(\pi)$, where $k=\varepsilon_c(\pi)$. Then 
\[   D_{\Delta}(\pi) \hookrightarrow D_{{}^-\Delta}(\widetilde{\pi}) \times (\nu^c\rho)^{\times (k-1)} 
\]
and $\varepsilon_c(D_{\Delta}(\pi))=k-1$.
\end{lemma}

\begin{proof}

We have
\[   \pi \hookrightarrow \widetilde{\pi} \times (\nu^c\rho)^{\times k}  \]
and so
\[  D_{\Delta}(\pi)\boxtimes \mathrm{St}(\Delta) \hookrightarrow \pi_{N_{i}} \hookrightarrow (\widetilde{\pi} \times (\nu^c\rho)^{\times k})_{N_{i}} ,
\]
where $i=l_{abs}(\Delta)$.

We now again apply the geometric lemma on $(\widetilde{\pi} \times (\nu^c\rho)^{\times k})_{N_{i}}$. By Proposition \ref{prop non composition factor at right pt} and similar to several arguments in proofs before, $\widetilde{\pi}_{N_{i'}}$ (for any $i'$) cannot contribute to a factor of $\nu^c\rho$ for $\mathrm{St}(\Delta)$, and so a comparison of cuspidal support gives that the only layer in the geometric lemma that can contribute the submodule $ D_{\Delta}(\pi)\boxtimes \mathrm{St}(\Delta)$ takes the form:
\begin{align} \label{eqn term in geometric lemma reduction} 
 \mathrm{Ind}^{G_{n'}\times G_{i}}_{P_{n'-r,r}\times P_{i-s,s}} (\widetilde{\pi}_{N_r} \boxtimes ((\nu^c\rho)^{\times k})_{N_{\mathrm{deg}(\rho)}})^{\phi},
\end{align}
where 
\begin{itemize}
\item $n'=\mathrm{deg}(D_{\Delta}(\pi))$; and
\item $s=\mathrm{deg}(\rho), r=l_{abs}({}^-\Delta)$; and
\item $\phi$ is a twist to get a $G_{n'-r}\times G_r\times G_{i-s}\times G_s$-representation.
\end{itemize}

Thus we have an embedding $D_{\Delta}(\pi)\boxtimes \mathrm{St}(\Delta)$ to the representation (\ref{eqn term in geometric lemma reduction} ). Now, we have
\begin{align*}
0\neq   & \mathrm{Hom}_{G_{n'}\times G_{i}}( D_{\Delta}(\pi)\boxtimes \mathrm{St}(\Delta) ,  \mathrm{Ind}^{G_{n'}\times G_{i}}_{P_{n'-r,r}\times P_{i-s,s}} (\widetilde{\pi}_{N_r} \boxtimes ((\nu^c\rho)^{\times k})_{N_{\mathrm{deg}(\rho)}})^{\phi})  \\
	\cong & \mathrm{Hom}_{G_{n'}\times G_{i-s}\times G_s}(D_{\Delta}(\pi) \boxtimes \mathrm{St}(\Delta)_{N_s},   \mathrm{Ind}^{G_{n'}\times G_{i-s}\times G_s}_{P_{n'-r,r}\times G_{i-s}\times G_s} (\widetilde{\pi}_{N_r} \boxtimes ((\nu^c\rho)^{\times (k-1)} \boxtimes \nu^c\rho))^{\phi}) \\
	\cong&  \mathrm{Hom}_{G_{n'}\times G_{i-s}\times G_s}(D_{\Delta}(\pi)\boxtimes \mathrm{St}({}^-\Delta)\boxtimes \nu^c\rho, ,   \mathrm{Ind}^{G_{n'}\times G_{i-s}\times G_s}_{P_{n'-r,r}\times G_{i-s} \times  G_s} (\widetilde{\pi}_{N_r} \boxtimes ((\nu^c\rho)^{\times (k-1)} \boxtimes \nu^c\rho))^{\phi})
\end{align*}
where  the first isomorphism follows from Frobenius reciprocity and the second isomorphism follows from (\ref{eqn steinberg jacquet}).

Now, applying the adjointness of tensor product, we have:
\begin{align} \label{eqn reduction by rho derivative}
\quad D_{\Delta}(\pi)  \hookrightarrow \mathrm{Hom}_{G_r}(\mathrm{St}({}^-\Delta), \widetilde{\pi}_{N_r}) \times (\nu^c\rho)^{\times (k-1)},
\end{align}
where $\mathrm{Hom}_{G_r}(\mathrm{St}({}^-\Delta), \widetilde{\pi}_{N_r})$ is equipped with $G_{n'}$-action via: for $f \in \mathrm{Hom}_{G_r}(\mathrm{St}({}^-\Delta), \widetilde{\pi}_{N_r})$, and $x \in \mathrm{St}({}^-\Delta)$,
\[  (g.f)(x) = \mathrm{diag}(g, I_{r}).(f(x)). 
\]

By Proposition \ref{prop smaller derivative}, we have that $\varepsilon_c(D_{\Delta}(\pi))=k-1$. Hence, Corollary \ref{cor embeding derivative} and (\ref{eqn reduction by rho derivative}) give
\[  D_c^{k-1}\circ D_{\Delta}(\pi) \hookrightarrow \mathrm{Hom}_{G_r}(\mathrm{St}({}^-\Delta), \widetilde{\pi}_{N_r})
\]
and hence $D_c^{k-1} \circ D_{\Delta}(\pi) \cong D_{{}^-\Delta}(\widetilde{\pi})$, as desired.
\end{proof}

We now study the change of $\mathfrak{mxpt}^a$ under the derivatives. The proof is based on two special cases: Lemma \ref{lem change in epsilon} and Proposition \ref{prop smaller derivative}.

\begin{theorem} \label{thm effect of Steinberg}
Let $\pi \in \mathrm{Irr}_{\rho}$ and let $\Delta=[a,b]_{\rho}$ be a segment. 
\begin{enumerate}
 \item[(1)] Suppose $\Delta$ is not admissible to $\mathfrak{hd}(\pi)$. Then $D_{\Delta}(\pi)=0$.
 \item[(2)] Suppose $\Delta$ is admissible to $\mathfrak{hd}(\pi)$.  Then, the followings hold:
 \begin{itemize}
     \item For any integer $c \geq a$, 
		\[ \mathfrak{mxpt}^a(D_{\Delta}(\pi), c) = \mathfrak r( \Delta, \pi)[c] . \]
		(See Section \ref{ss notations multi} for the notations.)
		 \item For any integer $c <a$,  
		\[  \mathfrak{hd}(\pi)[c]= \mathfrak r( \Delta, \pi)[c] \leq^a_{c}  \mathfrak{mxpt}^a(D_{\Delta}(\pi), c) .
		\]
     \item For any $\Delta'$ unlinked to $\Delta$ and $a(\Delta')\not\cong a(\Delta)$, 
		\[  \varepsilon_{\Delta'}(D_{\Delta}(\pi))=\varepsilon_{\Delta'}(\pi) .
		\]
	\end{itemize}
\end{enumerate}
\end{theorem}

\begin{proof}
(1) follows from definitions. The second bullet of (2) can be proved by a similar manner to the first case in the proof of Lemma \ref{lem non-zero composite} and we omit the details.

We shall prove the first bullet of (2) by an induction on $\mathrm{deg}(\pi)$. When $\mathrm{deg}(\pi)=0,1$, it is trivial.  Let $\Delta=[a,b]_{\rho}$ be an admissible segment for $\pi$. Let $\widetilde{\pi}=D^k_a(\pi)$, where $k=\varepsilon_a(\pi)$.

Now we consider two cases:
\begin{enumerate}
\item[(i)] Suppose $c \geq a+2$. In such case, $D_{\Delta} \circ D^{k-1}_a(\pi) \cong D^{k-1}_a\circ D_{\Delta}(\pi)$ . Let $\mathfrak n \in \mathrm{Mult}_{\rho,c}^a$. Then 
\[  D_{\mathfrak n} \circ D_{\Delta}\circ D^{k-1}_a(\pi) \neq 0 \quad \Leftrightarrow \quad D^{k-1}_a \circ D_{\mathfrak n} \circ D_{\Delta}(\pi) \neq 0 \quad \Leftrightarrow D_{\mathfrak n}\circ D_{\Delta}(\pi) \neq 0 ,
\]
where the first 'if and only if' condition follows by applying Lemma \ref{lem comm derivative 1} twice, and for the second 'if and only if' condition, the 'only if' direction is straightforward while the 'if' direction follows by a similar proof for Lemma \ref{lem non-zero composite}. Now one deduces  $\mathfrak{mxpt}^a(D_{\Delta}(\pi),c)$ ($c\geq a$)  as follows: First,
\[  \mathfrak{mxpt}^a(D_{\Delta^-}(\widetilde{\pi}),c)=\mathfrak r(\Delta^-, \widetilde{\pi})[c]=\mathfrak r(\Delta, \pi)[c],
\]
where the first equality follows from the inductive case and the second equality follows from Lemma \ref{lem change in epsilon} and the definition of removal process. Then, by Lemma \ref{lem transitive embedding on max k}, $D_a^{k-1}(D_{\Delta}(\pi))=D_{\Delta}(\widetilde{\pi})$ and so by applying Lemma \ref{lem change in epsilon} on $D_{\Delta}(\pi)$, we have that $\mathfrak{mxpt}^a(D_{\Delta^-}(\widetilde{\pi}),c)=\mathfrak{mxpt}^a(D_{\Delta}(\pi),c)$. Combining all, we have
\[  \mathfrak{mxpt}^a(D_{\Delta}(\pi),c)=\mathfrak r(\Delta, \pi)[c]
\]
as desired.
\item[(ii)] Suppose $c=a+1$. Let 
\[  \mathfrak s= \mathfrak{mxpt}^a(\pi, a)- \Delta_0  ,
\]
where $\Delta_0$ is the shortest segment in $ \mathfrak{mxpt}^a(\pi, a)$ such that $\Delta \subset \Delta_0$.

By Lemma \ref{lem change in epsilon},
\[  \mathfrak n:=\mathfrak{mxpt}^a(\widetilde{\pi}, a+1)= {}^- (\mathfrak{mxpt}^a(\pi, a)) +\mathfrak{mxpt}^a(\pi, a+1) . 
\]
By Proposition \ref{prop smaller derivative}, 
\[  D_{\mathfrak s} \circ D_{\Delta}(\pi) \neq 0 
\]
and now by multiple uses of Lemma \ref{lem transitive embedding on max k}, 
\[  D_{\mathfrak s}\circ D_{\Delta}(\pi) \cong D_{{}^-\mathfrak s}\circ D_{{}^-\Delta}(\widetilde{\pi}) .
\]
Hence, by the inductive case,
\[ \mathfrak{mxpt}^a( D_{\mathfrak s}\circ D_{\Delta}(\pi), a+1)=\mathfrak{mxpt}^a(D_{{}^-\mathfrak s}\circ D_{{}^-\Delta}(\widetilde{\pi}), a+1)
\]
is precisely 
\[ \mathfrak r({}^-\mathfrak s+{}^-\Delta,\widetilde{\pi})[a+1]= \mathfrak r( \mathfrak s+\Delta, \pi)[a+1] .
\]
The last equality follows from the rules of the removal process and Lemma \ref{lem change in epsilon}.

Indeed we also have
\[ \mathfrak r(\mathfrak s+\Delta, \pi)[a+1] =\mathfrak r(\mathfrak s, \mathfrak r(\Delta, \pi))[a+1]=  \mathfrak r(\Delta, \pi)[a+1]  ,
\]
where the second equality follows from that applying $D_{\Delta'}$ for each $\Delta'\in \mathfrak s$  will simply remove the segment $\Delta'$ in $\mathfrak r(\Delta,\pi)$ by Lemma \ref{lem remove whole seg}. 

Since $\mathfrak s=\mathfrak{mxpt}^a(D_{\Delta}(\pi),a)$, Lemma \ref{lem max multisegment equal} implies that 
\[ \mathfrak{mxpt}^a(D_{\Delta}(\pi), a+1)=\mathfrak{mxpt}^a(D_{\mathfrak s}\circ D_{\Delta}(\pi), a+1) .\]
Now, combining all the equations, we have that:
\[ \mathfrak{mxpt}^a(D_{\Delta}(\pi), a+1)=\mathfrak r(\Delta, \pi)[a+1] .
\]
\item[(iii)] Suppose $c=a$. Then $\mathfrak{mxpt}^a(D_{\Delta}(\pi), a)$ follows from Propositions  \ref{prop uniqueness max multi} and \ref{prop smaller derivative}.
\end{enumerate}

We now prove the third bullet of (2). The inequality
\[  \varepsilon_{\Delta'}(D_{\Delta}(\pi)) \leq \varepsilon_{\Delta'}(\pi)
\]
follows from Lemma \ref{lem comm derivative 1}. The opposite inequality follows from an application of the geometric lemma similar to the proof of Lemma \ref{lem max multisegment equal}. We omit the details. 
\end{proof}

\begin{definition} \label{def admissible derivatives}
Let $\pi \in \mathrm{Irr}_{\rho}$. Let $\mathfrak n=\left\{ \Delta_1, \ldots, \Delta_r \right\} \in \mathrm{Mult}_{\rho}$ with segments written in an ascending order. We say that $\mathfrak n$ is {\it admissible} to $\pi$ if $D_{\Delta_r}\circ \ldots \circ D_{\Delta_1}(\pi)\neq 0$. By Lemma \ref{lem comm derivative 1}, the admissibility is independent of the choice of an ascending order of segments for $\mathfrak n$. 
\end{definition}

\begin{corollary} \label{cor admissible by resultant}
Let $\pi \in \mathrm{Irr}_{\rho}$. Let $\mathfrak n \in \mathrm{Mult}_{\rho}$. Then $\mathfrak n$ is admissible to $\mathfrak \pi$ if and only if $\mathfrak n$ is admissible to $\mathfrak{hd}(\pi)$. 
\end{corollary}

\begin{proof}
Let $\mathfrak n=\left\{ \Delta_1, \ldots, \Delta_r \right\}$ with segments written in an ascending order. Write $\Delta_i=[a_i, b_i]_{\rho}$ for all $i$. We shall assume that $a_1\leq \ldots \leq a_r$. We consider the if direction. Let $k$ be the smallest integer such that $\Delta_k$ is not admissible to 
\[  \mathfrak r(\left\{ \Delta_1, \ldots, \Delta_{k-1} \right\}, \mathfrak{hd}(\pi)) .
\]
We have that
\begin{align} \label{eqn mxpt in addmisible}
 \mathfrak{mxpt}^a(  D_{\Delta_{k-1}}\circ \ldots \circ D_{\Delta_1}(\pi), a_k)= \mathfrak r(\left\{ \Delta_1, \ldots, \Delta_{k-1} \right\}, \mathfrak{hd}(\pi))[a_k]
\end{align}
by Theorem \ref{thm effect of Steinberg}(2). Thus the admissibility of $\Delta_k$ to $\mathfrak r(\left\{ \Delta_1, \ldots, \Delta_{k-1} \right\}, \mathfrak{hd}(\pi))$ implies that
\[   \left\{\Delta_k \right\} \leq_{a_k}  \mathfrak{mxpt}^a(  D_{\Delta_{k-1}}\circ \ldots \circ D_{\Delta_1}(\pi), a_k)
\]
by (\ref{eqn mxpt in addmisible}). Hence $D_{\Delta_k}\circ D_{\Delta_{k-1}}\circ \ldots \circ D_{\Delta_1}(\pi) \neq 0$ by Proposition \ref{prop smaller derivative}. 

The only if direction is similar by using (\ref{eqn mxpt in addmisible}) and Proposition \ref{prop smaller derivative} and we omit the details.
\end{proof}

\subsection{Multisegment at a right point (revised)}

\begin{corollary} \label{cor right multi}
Let $\pi \in \mathrm{Irr}_{\rho}$ and let $c \in \mathbb Z$. Then
\[   \mathfrak{mxpt}^b(\pi, c) = \sum_{[a,b]_{\rho} \in \mathfrak{hd}(\pi), a\leq c \leq b}  [a,c]_{\rho} .
\]
In other words, each $[a,c]_{\rho}$ is contributed from a segment in $\mathfrak{hd}(\pi)$ of the form $[a,b]_{\rho}$ with some $b\geq c$.
\end{corollary}

\begin{proof}
Let $\mathfrak k=\sum_{[a,b]_{\rho} \in \mathfrak{hd}(\pi), a\leq c \leq b}  [a,c]_{\rho}$. By Theorem \ref{thm effect of Steinberg}(2) several times, we see that $D_{\mathfrak k}(\pi)\neq 0$. On the other hand, $\mathfrak r(\mathfrak k, \mathfrak{hd}(\pi))$ does not have any segment $\Delta$ such that $\nu^c\rho \in \Delta$. Hence, by Theorem \ref{thm effect of Steinberg} (1) and (2), $\mathfrak k$ is $\leq^b_c$-maximal such that $\mathfrak r(\mathfrak k, \mathfrak{hd}(\pi))$ is admissible to $\pi$ and so $\mathfrak{mxpt}^b(\pi,c)=\mathfrak k$ by  Definition \ref{def max R multiseg}.
\end{proof}

\begin{example}
Let $\pi \in \mathrm{Irr}_{\rho}$ with $\mathfrak{hd}(\pi)$ taking the form:
\[ \xymatrix{ &  \stackrel{1}{{\color{red}\bullet}} \ar@{-}[r] & \stackrel{2}{{\color{red} \bullet}}  \ar@{-}[r]   &  \stackrel{3}{{\color{red} \bullet}} \ar@{-}[r]   &   \stackrel{4}{\bullet} \ar@{-}[r]    &   \stackrel{5}{\bullet}  & \\
             &   &     &    &  \stackrel{4}{\bullet} \ar@{-}[r]   &   \stackrel{5}{\bullet}   &   \\
              &    &      &  \stackrel{3}{{\color{red} \bullet}} \ar@{-}[r]    &   \stackrel{4}{\bullet} \ar@{-}[r]   &   \stackrel{5}{\bullet} \ar@{-}[r]  &    \stackrel{6}{\bullet} \\     
					   &    &     &  \stackrel{3}{{\color{red} \bullet}}   &     &    &     \\  \stackrel{0}{\bullet} 	\ar@{-}[r] 	      &  \stackrel{1}{\bullet} \ar@{-}[r]   &   \stackrel{2}{\bullet}   &    &     &    &     \\					}
\]
The red points contribute to $\mathfrak{mxpt}^b(\pi, 3)$ and so $\mathfrak{mxpt}^b(\pi, 3)=\left\{ [1,3], [3], [3] \right\}$. 
\end{example}

By Theorem \ref{thm effect of Steinberg}(1) and the third bullet of Theorem \ref{thm effect of Steinberg}(2), we also have:

\begin{corollary} \label{cor multisegment at right point}
Let $\pi \in \mathrm{Irr}_{\rho}$. Let $\mathfrak m \in \mathrm{Mult}_{\rho, c}^b$ for some $c \in \mathbb Z$. Then $D_{\mathfrak m}(\pi)\neq 0$ if and only if $\mathfrak m$ is a submultisegment of $\mathfrak{mxpt}^b(\pi, c)$.
\end{corollary}

\section{Isomorphic simple quotients of Bernstein-Zelevinsky derivatives} \label{ss isomorphic bz derivatives}

\subsection{Complementary sequence of St-derivatives}

For $\pi \in \mathrm{Irr}_{\rho}$ and $\mathfrak n \in \mathrm{Mult}_{\rho}$, define
\begin{align} \label{eqn removal for hd pi}
  \mathfrak r(\mathfrak n, \pi):=\mathfrak r(\mathfrak n, \mathfrak{hd}(\pi)) ,
\end{align}
where the latter term is defined in Definition \ref{def removal process}. We now prove a main property of the derivative resultant multisegment.


\begin{theorem} \label{lem segment derivative}
Let $\pi \in \mathrm{Irr}_{\rho}$. Let $\mathfrak n \in \mathrm{Mult}_{\rho}$ be admissible to $\pi$. Then 
\[  D_{\mathfrak r(\mathfrak n, \pi)}\circ D_{\mathfrak n}(\pi) \cong \pi^- .
\]
\end{theorem}

\begin{proof}

\noindent
{\bf Step 1: Inductive set-up and the basic case for the highest derivative} \\

Let $\mathfrak n=\left\{ \Delta_1, \ldots, \Delta_k\right\}$ with segments written in an ascending order. We shall prove by a backward induction on the sum of the absolute lengths of all those $\Delta_1, \ldots, \Delta_k$. Since the sequence is admissible, the sum must be not greater than the level of $\pi$ (by Proposition \ref{prop unique as segments}). If the sum is equal to the level of $\pi$, then $\omega \cong \pi^-$ by Proposition \ref{prop unique as segments}(3) and the irreducibility of the highest derivative of $\pi$. In this case, $\mathfrak r(\mathfrak n,\pi)=\emptyset$ and so it is immediate. \\

\noindent
{\bf Step 2: Rearrange unlinked segments according to certain cuspidal representation} \\

Let $c^*$ be the smallest integer such that $\mathfrak r(\left\{ \Delta_1, \ldots, \Delta_k \right\}, \pi)[c^*]\neq 0$. In other words, $\nu^{c^*}\rho$ is isomorphic to a $\leq$-minimal element in 
\[ \left\{ a(\Delta') : \Delta' \in \mathfrak r(\left\{ \Delta_1,\ldots, \Delta_k \right\},\pi) \right\}. \] 

We shall choose the ascending order of the sequence such that 
\[ a(\Delta_1) \leq \ldots \leq a(\Delta_k) .
\]
Let $r \leq k$ such that $\Delta_1, \ldots, \Delta_r$ be all the segments such that $\nu^{c^*} \rho \geq a(\Delta_i)$ for $i=1, \ldots, r$. We rearrange the segments $\Delta_{\delta(r+1)}, \ldots, \Delta_{\delta(k)}$ so that 
\[ b(\Delta_{\delta(k)}) \geq \ldots \geq b(\Delta_{\delta(r+1)}), \]
 where $\delta$ is a permutation on $\left\{ r+1, \ldots , k \right\}$. 

Let $\tau=D_{\Delta_r}\circ \ldots \circ D_{\Delta_1}(\pi)$. Note that the sequence $\Delta_{\delta(r+1)}, \ldots, \Delta_{\delta(k)}$ can be obtained from $\Delta_{r+1}, \ldots, \Delta_k$ by repeatedly switching two adjacent unlinked segments. Hence, by Lemma \ref{lem comm derivative 1},
\begin{align} \label{eqn rearrange segments main thm}
\quad  D_{\Delta_{\delta(k)}} \circ \ldots \circ D_{\Delta_{\delta(r+1)}}(\tau) \cong D_{\Delta_k} \circ \ldots\circ D_{\Delta_{r+1}}(\tau) \cong D_{\mathfrak n}(\pi) .
\end{align}

\ \\
\noindent
{\bf Step 3: Prove the theorem modulo a claim on commutation of derivatives}\\

Now let $\widetilde{\Delta}$ be a longest segment in $\mathfrak r(\left\{ \Delta_1, \ldots, \Delta_k \right\}, \pi)[c^*]$.

\noindent
{\it Claim:} For $i \geq 2$,
\[ D_{\widetilde{\Delta}} \circ D_{\Delta_{\delta(i)}} \circ \ldots \circ D_{\Delta_{\delta(1)}} (\tau) \cong D_{\Delta_{\delta(i)}}\circ D_{\widetilde{\Delta}} \circ D_{\Delta_{\delta(i-1)}} \circ \ldots \circ D_{\Delta_{\delta(1)}}(\tau) ,\]
and for $i=1$,
\[  D_{\widetilde{\Delta}} \circ D_{\Delta_{\delta(1)}}(\tau) \cong D_{\Delta_{\delta(1)}}\circ D_{\widetilde{\Delta}}(\tau)  .\]

Suppose the claim holds for the meanwhile. We have a new ascending sequence of segments, 
\begin{align} \label{eqn commutative sequence}
  \Delta_1, \ldots, \Delta_r, \widetilde{\Delta}, \Delta_{r+1}, \ldots, \Delta_k ,
\end{align}
which is admissible since the composition of their corresponding derivatives is non-zero by (\ref{eqn rearrange segments main thm}) and the claim. Now one applies the induction hypothesis to obtain that 
\[   D_{\mathfrak r(\mathfrak n+\widetilde{\Delta}, \pi)}\circ D_{\mathfrak n+\widetilde{\Delta}}(\pi)\cong \pi^- .
\]
By Lemma \ref{lem does not remove lower segment}, $\mathfrak r(\mathfrak n, \pi)=\mathfrak r(\mathfrak n+\widetilde{\Delta}, \pi)+\widetilde{\Delta}$. Thus, 
\[   D_{\mathfrak r(\mathfrak n, \pi)}\circ D_{\mathfrak n}(\pi) \cong D_{\mathfrak r(\mathfrak n+\widetilde{\Delta}, \pi)}\circ D_{\widetilde{\Delta}}\circ D_{\mathfrak n}(\pi)\cong D_{\mathfrak n+\widetilde{\Delta}}(\pi),
\] 
where the first isomorphism follows from definition and the second isomorphism follows from the claim.

Combining above isomorphisms, we have  $D_{\mathfrak r(\mathfrak n, \pi)}\circ D_{\mathfrak n}(\pi) \cong \pi^-$. \\

\noindent
{\bf Step 4: Prove the claim by Lemma \ref{lem comm derivative 1} or Corollary \ref{cor commut derivative}  } \\

It remains to prove the claim. Indeed, it will follow from Lemma \ref{lem comm derivative 1} or Corollary \ref{cor commut derivative} if we could check those conditions in the corresponding lemma. We use the notations in the claim. If $\widetilde{\Delta}$ and $\Delta_{\delta(i)}$ are unlinked, then we use Lemma \ref{lem comm derivative 1} and we are done. Now suppose $\widetilde{\Delta}$ and $\Delta_{\delta(i)}$ are linked. Note that,  by Lemma \ref{lem does not remove lower segment},
\begin{align} \label{eqn no effect on smallest}
\mathfrak r(\left\{\Delta_1, \ldots, \Delta_k \right\}, \pi)[c^*]=\mathfrak r(\left\{ \Delta_1, \ldots, \Delta_r, \Delta_{\delta(r+1)}, \ldots, \Delta_{\delta(i)}\right\}, \pi)[c^*] .
\end{align}
This implies that, by Theorem \ref{thm effect of Steinberg}(2) (multiple times),
 
\begin{align} \label{eqn main thm condition1}
\quad D_{\widetilde{\Delta}}(\kappa) \neq 0 ,
\end{align}
where 
\[  \kappa =D_{\Delta_{\delta(i-1)}} \circ \ldots \circ D_{\Delta_{\delta(r+1)}}\circ D_{\Delta_r}\circ \ldots \circ D_{\Delta_1}(\pi) .\]
Now let $\Delta'=\Delta_{\delta(i)} \cup \widetilde{\Delta}$ and we have to check that 
$D_{\Delta'}(\kappa)=0$. Note that $b(\Delta') >b(\widetilde{\Delta})$. Thus, we have:
\begin{enumerate}
\item by the longest choice of $\widetilde{\Delta}$ in $\mathfrak r(\left\{ \Delta_1, \ldots, \Delta_k \right\}, \pi)[c^*]$ and (\ref{eqn no effect on smallest}), $\Delta'$ is not admissible to $\mathfrak r(\left\{ \Delta_1, \ldots, \Delta_r \right\}, \pi)$. Hence, Theorem \ref{thm effect of Steinberg}(1) implies that
\[  D_{\Delta'}(\tau) = 0.  \]
\item our arrangement on $\Delta_{\delta(p)}$s gives that $\Delta'$ and $\Delta_{\delta(x)}$ are unlinked for $x=1, \ldots, i-1$. (Here we also use that $a(\Delta_{\delta(x)})>\nu^{c_*}\rho$. See the third paragraph of the proof.)
\end{enumerate}

Hence, (1) and (2) above give
\begin{align} \label{eqn main thm zero after intersect}
\quad  D_{\Delta'}(\kappa) = D_{\Delta'}\circ D_{\Delta_{\delta(i-1)}}\circ \ldots \circ D_{\Delta_{\delta(1)}}(\tau) =D_{\Delta_{\delta(i-1)}}\circ \ldots \circ D_{\Delta_{\delta(1)}}\circ D_{\Delta'}(\tau) =0 ,
\end{align}
where the second equality follows from (2) above with Lemma \ref{lem comm derivative 1} and the last equality follows from (1) above. 

Since $D_{\Delta_k} \circ \ldots \circ D_{\Delta_1}(\pi)\neq 0$, we also have 
\begin{align} \label{eqn main thm condition3}
 D_{\Delta_{\delta(i)}}(\kappa)\neq 0. 
\end{align}

Hence, the conditions $(\ref{eqn main thm condition1}), (\ref{eqn main thm zero after intersect}), (\ref{eqn main thm condition3})$ guarantee conditions in Corollary \ref{cor commut derivative} and this completes the proof of the claim.
\end{proof}

\subsection{Isomorphic simple quotients under St-derivatives} \label{ss isomorphic quotients}

We now prove a main result using the highest derivative multisegment and the removal process to determine when two sequences of St-derivatives give rise to isomorphic simple quotients of a Bernstein-Zelevinsky derivative. The strategy for the 'if' direction below of Theorem \ref{thm isomorphic derivatives} below is that we use Theorem \ref{lem segment derivative} to construct isomorphic modules by taking the same sequence of St-derivatives. The strategy for the 'only if' direction is to find some St-derivatives that kill one, but not another one. However, in order to do so, we need to do it on some other derivatives via some constructions.

\begin{theorem} \label{thm isomorphic derivatives}
Let $\pi \in \mathrm{Irr}_{\rho}$. Let $\Delta_1, \ldots , \Delta_k$ and $\Delta_1', \ldots, \Delta_l'$ be two ascending sequences of segments that are admissible to $\pi$. Then 
\[   D_{\Delta_k}\circ \ldots \circ D_{\Delta_1}(\pi) \cong D_{\Delta_l'}\circ \ldots \circ D_{\Delta_1'}(\pi)
\]
if and only if 
\[  \mathfrak r(\left\{ \Delta_1, \ldots, \Delta_k \right\}, \pi) = \mathfrak r(\left\{ \Delta_1', \ldots, \Delta_l' \right\}, \pi) .
\]
\end{theorem}

\begin{proof}

Let 
\[  \omega_1 =D_{\Delta_k} \circ \ldots \circ D_{\Delta_1}(\pi), \quad \omega_2=D_{\Delta_l'}\circ \ldots \circ D_{\Delta_1'}(\pi) .
\]

\ \\

\noindent
{\bf Step 1: Prove the if direction} \\

For the if direction, we write $\widetilde{\Delta}_1, \ldots, \widetilde{\Delta}_r$ to be all the segments in $\mathfrak r(\left\{ \Delta_1, \ldots, \Delta_k\right\}, \pi)$ in an ascending order. It follows from Theorem \ref{lem segment derivative} that 
\[  D_{\widetilde{\Delta}_r} \circ \ldots \circ D_{\widetilde{\Delta}_1}(\omega_1)\cong D_{\widetilde{\Delta}_r} \circ \ldots \circ D_{\widetilde{\Delta}_1}(\omega_2) \cong \pi^- .
\]
 Hence, both $\omega_1$ and $\omega_2$ are isomorphic to 
\[     I_{\widetilde{\Delta}_1}\circ \ldots \circ I_{\widetilde{\Delta}_r}(\pi^-) ,
\]
where $I_{\widetilde{\Delta}_s}(\tau)$ (for $\tau \in \mathrm{Irr}$ and $s=1,\ldots, r$) denotes the unique irreducible submodule of $\tau \times \mathrm{St}(\widetilde{\Delta}_s)$ (see Lemma \ref{lem socle cosocle ladder}, and for more discussions on 'integrals', see \cite{LM16, Ch22+, Ch22+c}). In particular, 
\[  \omega_1 \cong \omega_2 .
\]

\noindent
{\bf Step 2: Define the representations $\omega_{1,i}$ and $\omega_{2,i}$ for the only if direction }

We now consider the only if direction. We denote by 
\[  \mathfrak r_1=\mathfrak r( \left\{\Delta_1, \ldots, \Delta_k \right\}, \pi), \quad \mathfrak r_2=\mathfrak r( \left\{\Delta_1', \ldots, \Delta_l' \right\}, \pi) .
\]
For $p=1,2$, recall that $\mathfrak r_{p}\langle c \rangle$ is the sub-multisegment of $\mathfrak r_p$ exactly containing all the segments $\Delta$ satisfying $b(\Delta) \cong \nu^c \rho$. Let $c^*$ be the largest integer such that 
\[  \mathfrak r_{1}\langle c^*\rangle \neq \mathfrak r_{2}\langle c^*\rangle .
\]

Let $\overline{c}$ be the largest integer such that $\nu^{\overline{c}}\rho \in \mathrm{csupp}(\pi)$. Let $c_i=\overline{c}-i$ for $i \geq 0$, and let $z$ be the integer such that $c_z=c^*$. Set
\begin{align} \label{eqn derivative multi on pi}
 \omega_{1,0}=D_{\Delta_k}\circ \ldots \circ D_{\Delta_1}(\pi), \quad \omega_{2,0}=D_{\Delta_l'} \circ \ldots \circ D_{\Delta_1'}(\pi).
\end{align}
For $p=1, 2$, we inductively, for each $c_i$ ($i=0, \ldots, z$), define representations $\omega_{1,i}$ and $\omega_{2,i}$ as follows. First, define
\begin{align} \label{eqn kappa omega derivative}
 \kappa_{1,i}= D_{\mathfrak p_{i-1}} ( \omega_{1,i-1}) ,\quad \kappa_{2,i}= D_{\mathfrak p_{i-1}}(\omega_{2,i-1}) ,
\end{align}
where $\mathfrak p_{i-1}=\mathfrak r_1\langle c_{i-1} \rangle =\mathfrak r_2\langle c_{i-1} \rangle$ (possibly empty and the equality $\mathfrak r_1\langle c_{i-1} \rangle =\mathfrak r_2\langle c_{i-1} \rangle$ follows from our choice of $c^*$).

Suppose the following conditions hold in the meanwhile:
\begin{enumerate}
\item[(1)] $\omega_{1,i-1} \cong D_{\mathfrak n_{1,i-1}}(\pi)$ and $\omega_{2,i-1} \cong D_{\mathfrak n_{2,i-1}}(\pi)$ for some $\mathfrak n_{1,i-1}, \mathfrak n_{2,i-1} \in \mathrm{Mult}_{\rho}$ admissible to $\pi$, and 
\item[(2)] for any $e \geq c_{i-1}+1$, 
\[  \mathfrak r(\mathfrak n_{1,i-1}, \pi)\langle e \rangle = \mathfrak{hd}(\pi)\langle e \rangle, \quad \mathfrak r(\mathfrak n_{1,i-1}, \pi)\langle e \rangle =\mathfrak{hd}(\pi)\langle e \rangle ,
\]
and
\item[(3)] for any $e \leq c_{i-1}$, 
\[ \mathfrak r(\mathfrak n_{1,i-1}, \pi)\langle e \rangle =\mathfrak r_1\langle e \rangle, \quad \mathfrak r(\mathfrak n_{2,i-1}, \pi)\langle e \rangle = \mathfrak r_2\langle e \rangle. 
\]
\end{enumerate}

The conditions (1) and (2) above guarantee that one can apply Lemma \ref{lem shrinking mutli 0}. With (3), that means, there exist multisegments $\mathfrak m_{1,i-1}$ and $\mathfrak m_{2,i-1}$ such that 
\begin{align} \label{eqn inductive construct condition 2}
  \mathfrak r(\mathfrak m_{1,i-1}, \pi)=\mathfrak r(\mathfrak n_{1,i-1}+\mathfrak p_{i-1}, \pi)+\mathfrak q_{i-1}, \quad \mathfrak r(\mathfrak m_{2,i-1}, \pi)=\mathfrak r(\mathfrak n_{2,i-1}+\mathfrak p_{i-1}, \pi) +\mathfrak q_{i-1} ,
\end{align}
where $\mathfrak q_{i-1}=\mathfrak{hd}(\pi)\langle c_{i-1}\rangle$. (Note that $\mathfrak r(\mathfrak n_{1,i-1}+\mathfrak p_{i-1}, \pi)=\emptyset$ and $\mathfrak r(\mathfrak n_{2,i-1}+\mathfrak p_{i-1}, \pi)=\emptyset$ and so the above terms from  Lemma \ref{lem shrinking mutli 0} are dropped.)

Now we can define $\omega_{1,i}$ and $\omega_{2,i}$:
\begin{align} \label{eqn inductive construct condition 1}
  \omega_{1,i}=D_{\mathfrak m_{1,i-1}}(\pi), \quad \omega_{2,i}=D_{\mathfrak m_{2,i-1}}(\pi) 
\end{align}

\ \\

\noindent
{\bf Step 3: Check the conditions (1), (2) and (3) in Step 2} \\

To check the well-definedness of $\omega_{1,i}$ and $\omega_{2,i}$, we have to check the conditions (1), (2) and (3) above and we will again do inductively. Note that, (1) for $\omega_{1,0}$ and $\omega_{2,0}$ is automatic from (\ref{eqn derivative multi on pi}); (2) for $\omega_{1,0}$ and $\omega_{2,0}$ is automatic from our choice of $\bar{c}$ and (3) for $\omega_{1,0}$ and $\omega_{2,0}$ is automatic from definitions. These give the basis case. 

Now, (1) for general $\omega_{1,i}$ follows directly from (\ref{eqn inductive construct condition 1}). For (2) for general $\omega_{1,i}$, the induction hypothesis (specially using (3)) and Lemma \ref{lem remove whole seg} imply that
\begin{align} \label{eqn inductive formula 1}
 \mathfrak r(\mathfrak n_{1,i-1}+\mathfrak p_{i-1}, \pi)\langle e \rangle=\mathfrak r(\mathfrak n_{1,i-1}, \pi)\langle e\rangle
\end{align}
for any $e \neq c_{i-1}=c_i+1$, and so (\ref{eqn inductive formula 1}) with the formula (\ref{eqn inductive construct condition 2}) implies that (3) holds; and (2) holds for the cases $e\geq c_{i-1}+1=c_{i-1}+2$. For the remaining case of (2), the induction hypothesis (specifically using (2)) gives that
\[  \mathfrak r(\mathfrak n_{1,i-1}+\mathfrak p_{i-1}, \pi)\langle c_{i}+1\rangle =\emptyset
\]
and so the formula (\ref{eqn inductive construct condition 2}) also gives that
\[  \mathfrak r(\mathfrak m_{1,i-1}, \pi)\langle c_{i-1} \rangle = \mathfrak q_{i-1}=\mathfrak{hd}(\pi)\langle c_{i-1}\rangle , 
\]
as desired. The argument for $\omega_{2,i}$ is identical.
\ \\


\noindent
{\bf Step 4: Conclude the proof by applying $D_{\mathfrak r_1\langle c^*\rangle}$ or $D_{\mathfrak r_2\langle c^*\rangle}$}

Now, by (\ref{eqn kappa omega derivative}) and the (proved) if direction of this theorem,  
\[   \kappa_{1,i}\cong D_{\mathfrak q_{i-1}}(\omega_{1,i}), \quad \kappa_{2,i} \cong D_{\mathfrak q_{i-1}}(\omega_{2,i}) .\]
In other words,
\begin{align} \label{eqn kappa omega integral}
\omega_{1,i} \cong I_{\mathfrak q_{i-1}}(\kappa_{1,i}), \quad \omega_{2,i} \cong I_{\mathfrak q_{i-1}}(\kappa_{2,i}) .
\end{align}
Here $I_{\mathfrak q_{i-1}}(\kappa_{p,i})$ $(p=1,2$) is the unique simple submodule of $\kappa_{p,i} \times \mathrm{St}(\mathfrak q_{i-1})$ (see Lemma \ref{lem socle cosocle ladder}).


Since $\mathfrak r_1\langle c^* \rangle \neq \mathfrak r_2\langle c^* \rangle$, we have either
\[  D_{\mathfrak r_{2}\langle c^*\rangle}(\omega_{1,c^*}) =0, \quad \mbox{ or } D_{\mathfrak r_{1}\langle c^*\rangle}(\omega_{2,c^*})=0 .
\]
But $D_{\mathfrak r_2\langle c^*\rangle}(\omega_{2,c^*})\neq 0$ and $D_{\mathfrak r_2\langle c^* \rangle}(\omega_{1,c^*})\neq 0$. This implies that $\omega_{1,c^*} \not\cong \omega_{2,c^*}$. 

On the other hand, (\ref{eqn kappa omega derivative}) and (\ref{eqn kappa omega integral}) give that, for $p=1,2$,
\[ I_{\mathfrak q_{z-1}}\circ D_{\mathfrak p_{z-1}}\circ \ldots \circ  I_{\mathfrak q_0}\circ D_{\mathfrak p_0}( \omega_{p,0})  \cong\omega_{p,c^*} .
\]
This implies that $\omega_{1,0} \not\cong \omega_{2,0}$. 
\end{proof}

\subsection{Comparing with $\rho$-derivatives}

We give a simple quotient of a Bernstein-Zelevinsky derivative that cannot be obtained from an ascending sequence of derivatives from cuspidal representations.

\begin{example} \label{example rho not BZ}
Let $\mathfrak m=\left\{ [0,1]_{\rho}, [1]_{\rho}, [1,2]_{\rho} \right\}$ and let $\pi=\langle \mathfrak m \rangle$. Then $\mathfrak{hd}(\pi)=\left\{ [1]_{\rho}, [1,2]_{\rho} \right\}$. Hence, there is a simple quotient of $\pi^{(2)}$ obtained by applying $D_{[1,2]_{\rho}}$ and its derivative resultant multisegment is $\mathfrak r([1,2]_{\rho}, \pi)=\left\{ [1]_{\rho} \right\}$. But $D_{[2]_{\rho}}\circ D_{[1]_{\rho}}(\pi)=0$ and so $D_{[1,2]_{\rho}}(\pi)\neq D_{[2]_{\rho}}\circ D_{[1]_{\rho}}(\pi)$. 
\end{example}

\section{Examples of highest derivative multisegments} \label{ss examples}

\subsection{Generic representations}

An irreducible representation $\pi$ of $G_n$ is said to be generic if $\pi^{(n)}\neq 0$. According to \cite{Ze80}, for $\mathfrak m \in \mathrm{Mult}_{\rho}$, $\langle \mathfrak m \rangle$ is generic if and only if all the segments in $\mathfrak m$ are singletons. Equivalently, $\langle \mathfrak m \rangle \cong \mathrm{St}(\mathfrak n)$ for a multisegment $\mathfrak n$ whose all segments are unlinked. One can compute $\mathfrak n$, for example, by the M\oe glin-Waldspurger algorithm \cite{MW86}. In this case, $\mathfrak{hd}(\pi) =\mathfrak n$ (e.g. use (\ref{eqn product commut st})).

\subsection{Highest derivative multisegments for a product of representations}

\begin{proposition} \label{prop prod highest derivative 1}
Let $\pi_1, \ldots, \pi_r \in \mathrm{Irr}_{\rho}$. Suppose $\pi_1 \times \ldots \times \pi_r$ is still irreducible. Then 
\[  \mathfrak{hd}(\pi_1\times \ldots \times \pi_r) = \mathfrak{hd}(\pi_1)+\ldots +\mathfrak{hd}(\pi_r) .
\]
\end{proposition}

\begin{proof}

Note that one can recover the highest derivative multisegment $\mathfrak{hd}$ from all $\mathfrak{mxpt}^b$ by Corollary \ref{cor right multi}. Now it is a simple counting to obtain $\mathfrak{hd}(\pi_1\times \ldots \times \pi_r) = \mathfrak{hd}(\pi_1)+\ldots +\mathfrak{hd}(\pi_r)$ by Proposition \ref{prop max b under product}.
\end{proof}

\subsection{Arthur representations} \label{ss hd arthur rep}

We write 
\[  \Delta_{\rho}(d)=[-(d-1)/2, (d-1)/2]_{\rho}. 
\]
Let 
\[ u_{\rho}(d,m)=\langle \left\{ \nu^{(m-1)/2}\Delta_{\rho}(d), \ldots, \nu^{-(m-1)/2}\Delta_{\rho}(d) \right\} \rangle . \]
Let $Y(u_{\rho}(d,m))=\nu^{(d-m)/2}\rho$. The representations $u_{\rho}(d,m)$ are so-called Speh representations. To each Speh representation $u_{\rho}(d,m)$, one can attach a segment 
\[  \Delta(u_{\rho}(d,m)):= [(d-m)/2, (d+m-2)/2]_{\rho} .
\]
It follows from \cite[Theorem 14]{LM14} (also see \cite[Corollary 7.2]{CS19}) that 
\begin{align} \label{eqn highest derivative speh}
  \mathfrak{hd}(u_{\rho}(d,m))= \left\{ \Delta(u_{\rho}(d,m)) \right\}.
\end{align}

\begin{proposition} \label{prop hd arthur rep}
Let $\pi$ be an Arthur type representation in $\mathrm{Irr}_{\rho}$ i.e. 
\[  \pi = \pi_1 \times \ldots \times \pi_r ,
\]
for some Speh representations $\pi_1, \ldots, \pi_r$ (see e.g. \cite{GGP20} for discussions on Arthur type representations). Then
\[ \mathfrak{hd}(\pi) =  \Delta(\pi_1)+\ldots +\Delta(\pi_r) .
\]
\end{proposition}

\begin{proof}
It is known from \cite[Theorem 7.5]{Ta86} that $\pi$ is irreducible. The statement then follows from (\ref{eqn highest derivative speh}) and Proposition \ref{prop prod highest derivative 1}.
\end{proof}





\subsection{Ladder representations}

As we saw above, the highest derivative multisegment for a Speh representation is simply a segment. Let $\pi \in \mathrm{Irr}_{\rho}$ be a ladder representation. Then its associated multisegment $\mathfrak m=\left\{ \Delta_1, \ldots, \Delta_r\right\} \in \mathrm{Mult}_{\rho}$ satisfies the property:
\[  a(\Delta_1)<\ldots <a(\Delta_r), \quad b(\Delta_1)<\ldots <b(\Delta_r) .
\]
Note that there is a unique multisegment $\mathfrak n$ such that $\cup_{\Delta \in \mathfrak n} \Delta =\left\{ b(\Delta_1), \ldots, b(\Delta_r) \right\}$ and the segments in $\mathfrak n$ are mutually unlinked. We have that $\mathfrak{hd}(\pi)=\mathfrak n$.

\subsection{$\square$-irreducible representations}

An irreducible representation $\pi$ of $G_n$ is said to be {\it $\square$-irreducible} if $\pi \times \pi$ is still irreducible \cite{LM18}. For progress on characterizing such classes of modules, see, for example \cite{Le03, GLS11, KKKO18, LM18} and references therein.

\begin{proposition}
Let $\pi \in \mathrm{Irr}_{\rho}$. Suppose $\pi$ is $\square$-irreducible. Then, for any positive integer $k$,
\[ \mathfrak{hd}(\overbrace{\pi \times \ldots \times  \pi}^{k \mbox{ times}})=\overbrace{\mathfrak{hd}(\pi)+ \ldots +\mathfrak{hd}(\pi)}^{k \mbox{ times}} .\]
\end{proposition}

\begin{proof}
By \cite[Corollary 2.7]{LM18}, $\overbrace{\pi \times \ldots \times  \pi}^{k \mbox{ times}}$ is still irreducible. Then the proposition is a direct consequence of Proposition \ref{prop prod highest derivative 1}.
\end{proof}

\section{Appendix: Bernstein-Zelevinsky derivatives for affine Hecke algebras} \label{ss bz derivatives aha}

\subsection{BZ functor} \label{ss bz functor}
In this section, we explain how the results in this paper can be formulated in the affine Hecke algebra setting and we first give some background. We mainly follow \cite{CS19}, but we remark that we only need the Iwahori case to transfer results to the affine Hecke algebras using the Borel-Casselman's equivalence \cite{Bo76}, in which earlier work of Barbasch-Moy and Reeder \cite{BM94, Re02} shows that generic irreducible representations correspond to modules containing the sign module of the finite Hecke algebra, and is later used to study the unitary dual problem by Barbasch-Ciubotaru \cite{BC08}. Using the idea of finite Hecke algebra modules in characterizing modules also goes back to earlier work of Rogawski \cite{Ro85}.

A key to formulate the Bernstein-Zelevinsky derivative in \cite{CS19} is using an explicit affine Hecke algebra structure of the Iwahori component of the Gelfand-Graev representation in \cite{CS18}, and such expression is also obtained in Brubaker-Buciumas-Bump- Friedberg \cite{BBBF18}. Our realization of the Gelfand-Graev representation is obtained via viewing the affine Hecke algebra as the convolution algebra on Iwahori-biinvariant functions of $G_n$, and there is an alternate approach of describing affine Hecke algebra in terms of an endomorphism algebra due to Heiermann \cite{He11}. For Hecke algebras arising from other Bernstein components, see the work of Waldspurger \cite{Wa86} and Bushnell-Kutzko \cite{BK93}, and also the work of S\'echerre-Stevens \cite{SS12} for inner forms of general linear groups.

Let $q \in \mathbb{C}^{\times}$. Assume that $q$ is not a root of unity. The affine Hecke algebra $\mathcal H_n:= \mathcal H_n(q)$ is defined as an associative algebra over $\mathbb C$ with the generators $T_1, \ldots, T_{n-1}$ and $\theta_1, \ldots, \theta_n$ satisfying the following relations:
\begin{enumerate}
\item $T_kT_{k+1}T_k=T_{k+1}T_kT_{k+1}$ for $k=1, \ldots, n-2$;
\item $(T_k+1)(T_k-q)=0$ for all $k$;
\item $\theta_k\theta_l=\theta_l\theta_k$ for all $k,l$;
\item $T_k\theta_k-\theta_{k+1}T_k=(q-1)\theta_k$ for all $k$;
\item $T_k\theta_l=\theta_lT_k$ if $l\neq k, k+1$. 
\end{enumerate}

The subalgebra, denoted $\mathcal H_{S_n}$, generated by $T_1, \ldots, T_{n-1}$ is isomorphic to the finite Hecke algebra attached to the symmetric group $S_n$. For $k=1, \ldots, n-1$, define $s_k$ to be the transposition between $k$ and $k-1$. For $w \in S_n$ with a reduced expression $w=s_{k_1}\ldots s_{k_r}$, define $T_w=T_{k_1}\ldots T_{k_r}$. It is well-known that $T_w$ is independent of a choice of a reduced expression of $w$.

We now define the analogous Bernstein-Zelevinsky functor for $\mathcal H_n$ \cite{CS19}. There is a natural embedding from $\mathcal H_{n-i}\otimes \mathcal H_i$ to $\mathcal H_n$ explicitly given by: for $k=1, \ldots, n-i-1$, $m(T_k\otimes 1)=T_k$; for $k=1, \ldots, i$, $m(1\otimes T_k)=T_{n-i+k}$; for $k=1, \ldots, n-i$, $m(\theta_k\otimes 1)=\theta_k$; for $k=1, \ldots, i$, $m(1\otimes \theta_k)=\theta_{n-i+k}$. Define the sign projector:
\[ \mathbf{sgn}_i= \frac{1}{\sum_{w \in S_i} (1/q)^{l(w)}}\sum_{w \in S_n}(-1/q)^{l(w)}T_w \in \mathcal H_{S_i},
\]
Let $\mathbf{S}^n_i=m(1\otimes \mathbf{sgn}_i)$. The $i$-th Bernstein-Zelevinsky derivative for an $\mathcal H_n$-module $\sigma$ is defined as:
\[ \mathbf{BZ}_i(\sigma)=\mathbf{S}^n_i(\sigma) .
\]

The following result for $i=1$ is covered by \cite{GV01} by using $\rho$-derivatives. We remark that \cite{GV01} covers other cases such as the works of Kleshchev and Brundan \cite{Kl95} and \cite{Br98} for some positive characteristic algebras. See a survey of Kleshchev \cite{Kl10} for an overview of this problem. We remark that the branching law has deep connections with the theory of crystal bases. For instance, the decomposition matrix for restriction coincides with the coefficients of crystal bases in certain way, see the work of Lascoux-Leclerc-Thibon \cite{LLT96} and Ariki \cite{Ar96} and even the development for other classical types by Enomoto-Kashiwara, Miemietz, Varagnolo-Vasserot, Shan-Varagnolo-Vasserot \cite{EK08, Mie08, VV11, SVV11}. The following result generalizes part of \cite{GV01} and opens up some possibilities of connections with crystal theory:


\begin{theorem} \label{thm transfer multiplicity free}
Let $\sigma$ be an irreducible $\mathcal H_n$-module. Then the socle and cosocle of $\mathbf{BZ}_i(\sigma)$ are multiplicity-free.
\end{theorem}

\begin{proof}

We first assume that $q$ is a prime power. We choose $F$ to be a $p$-adic field with $|\mathcal O/\omega\mathcal O|=q$, where $\mathcal O$ is the ring of integers in $F$ and $\omega$ is the uniformizer. Then, by \cite[Theorem 4.2]{CS19} and Lemma \ref{lem property derivative}, we have the multiplicity-free result in such case. Note that in the case that $\sigma$ has a real central character (see \cite[Section 5.2]{CS19}, also see \cite[Section 2]{OS10}), we can further obtain that for the corresponding graded Hecke algebra $\mathbb H_n$, its analogous Bernstein-Zelevinsky derivative $\mathbf{gBZ}_i(\widetilde{\sigma})$ also has multiplicity-free socle and cosocle, where $\widetilde{\sigma}$ is the corresponding module under Lusztig's second reduction \cite{Lu89} (see \cite[Theorem 6.3]{CS19}). Here $\mathbf{gBZ}_i$ is defined in \cite[Section 6.3]{CS19}. Thus this implies that $\mathbf{gBZ}_i(\widetilde{\sigma}')$ has multiplicity-free socle and cosocle for any simple module $\widetilde{\sigma}'$ of $\mathbb H_n$ of real central character. (We remark that, by a rescaling argument, $\mathbb H_n$ can be defined as an associative algebra over $\mathbb C$ generated by the group algebra $\mathbb C[S_n]$ and the polynomial ring $S(\mathbb C^n)$ subject to some relations independent of $q$.)

We now consider arbitrary $q$ (which is not of root of unity). In such case, by using Lusztig's first reduction \cite{Lu89} (see \cite[Section 5.2]{CS19}) and \cite[Theorem 5.3]{CS19}, we can transfer to the problem of some affine Hecke algebra modules $\mathcal H_{n_1}\otimes \ldots \otimes \mathcal H_{n_k}$ with $n_1+\ldots +n_k=n$, and in which, we can apply Lusztig's second reduction. Now the result follows from the graded Hecke algebra case in previous paragraph.
\end{proof}


For a given segment $\Delta=[a,b]$ for $b-a \in \mathbb Z_{\geq 0}$, the Steinberg module $\mathrm{St}_{\mathcal H}(\Delta)$ of $\mathcal H_{b-a+1}$ is the $1$-dimensional module $\mathbb Cv$ determined by: for all $k$,
\[  T_k.v=-v,\quad \theta_k.v=q^{a+k-1}v . \]

For the St-derivatives, for an $\mathcal H_n$-module $\sigma$ and a given segment $\Delta$, one defines $D_{\Delta}(\sigma)$ to be either zero or the unique $\mathcal H_{n-i}$-module $\tau$ such that 
\[ \tau \boxtimes \mathrm{St}_{\mathcal H}(\Delta) \hookrightarrow \sigma|_{\mathcal H_{n-i}\otimes \mathcal H_i} .\]
Now one can define analogously the terminology of highest derivative segments, derivative resultant segments to formulate and prove the corresponding statements. 

\subsection{Left BZ functor}

Define $\zeta=\zeta_n$ on $\mathcal H_n$ determined by:
\[   \zeta(\theta_k)=\theta_{n-k+1}^{-1}, \quad \zeta(T_k)=T_{n-k}
\]
for any $k$. Note that $\zeta$ will send the relation (4) for the affine Hecke algebra to 
\[   T_{n-k}\theta_{n-k+1}^{-1}-\theta_{n-k}^{-1}T_{n-k}=(q-1)\theta_{n-k+1}^{-1} ,
\]
which is equivalent to $\theta_{n-k}T_{n-k}-T_{n-k}\theta_{n-k+1}=(q-1)\theta_{n-k}$. With the braid relation (2) for $\mathcal H_n$, one deduces the relation (4) for $\mathcal H_n$. It is straightforward to check that $\zeta$ preserves other relations of $\mathcal H_n$. Hence, $\zeta$ defines an automorphism on $\mathcal H_n$.

The left Bernstein-Zelevinsky functor ${}_i\mathbf{BZ}$ in the spirit of \cite{CS21, Ch21} is defined as:
\[  {}_i\mathbf{BZ}(\sigma)= \zeta_{n-i}(\mathbf{BZ}_i(\zeta_n(\sigma))).
\]

For any $\mathcal H_n$-module $\sigma$ and $s \in \mathbb C$, we can define $\chi^s \otimes \sigma$ as 
\[   T_k\cdot_{\chi^s\otimes \sigma}v=T_k \cdot_{\sigma}v , \quad \theta_k\cdot_{\chi^s\otimes\sigma}v=q^{s}\theta_k\cdot_{\sigma}v .
\]
We define the shifted Bernstein-Zelevinsky functors as:
\[  \mathbf{BZ}_{[i]}(\sigma) =\chi^{-1/2} \otimes\mathbf{BZ}_{i}(\sigma) ,
\]
\[  {}_{[i]}\mathbf{BZ}(\sigma) =\chi^{1/2}\otimes {}_i\mathbf{BZ}(\sigma).
\]

As shown in \cite[Theorem 6.2]{Ch21} and the argument in the proof of Theorem \ref{thm transfer multiplicity free}, we have the following asymmetry property:
\begin{theorem}
Let $\sigma$ be an irreducible $\mathcal H_n$-module. Suppose $i$ is not the level of $\sigma$ i.e. not the largest integer such that $\mathbf{BZ}_i(\sigma)\neq 0$. If $\mathbf{BZ}_{[i]}(\sigma)\neq 0$ and ${}_{[i]}\mathbf{BZ}(\sigma)\neq 0$, then any simple quotient (resp. submodule) of $\mathbf{BZ}_{[i]}(\sigma)$ is not isomorphic to that of ${}_{[i]}\mathbf{BZ}(\sigma)$.
\end{theorem}


\begin{thebibliography}{AGRS}
\bibitem[AL22]{AL22} A. Aizenbud and E. Lapid, A binary operation on irreducible components of Lusztig's nilpotent varieties I: definition and properties, to appear in Pure and Applied Mathematics Quarterly.
\bibitem[AGRS10]{AGRS10} A. Aizenbud, D. Gourevitch, S. Rallis and G. Schiffmann, Multiplicity one theorems, Ann. of Math. (2) {\bf 172} (2010), no. 2, 1407-1434. 
\bibitem[Ar96]{Ar96} S. Ariki, On the decomposition numbers of the Hecke algebra of G(M, 1, n), J. Math. Kyoto Univ. 36 (1996), 789-808.
\bibitem[AM20]{AM20} H. Atobe, A. M\'inguez, The explicit Zelevinsky-Aubert duality, Compositio. Math. 159(2) (2023), 380-418. doi:10.1112/S0010437X22007904
\bibitem[BM94]{BM94}  D. Barbasch, A. Moy, Whittaker models with an Iwahori-fixed vector. Contemp. Math. 177, 101-105 (1994)
\bibitem[BC08]{BC08}  D. Barbasch, D. Ciubotaru, (2009). Whittaker unitary dual of affine graded Hecke algebras of type E. Compositio Mathematica, 145(6), 1563-1616. doi:10.1112/S0010437X09004230
\bibitem[Be92]{Be92} J. Bernstein. Represenations of p-adic groups. Harvard University, 1992. Lectures by Joseph Bernstein. Written by Karl E. Rumelhart.
\bibitem[BZ76]{BZ76} I. N. Bernstein and A. V. Zelevinsky, Representations of the group $\mathrm{GL}(n,F)$, where $F$ is a non-archimedean local field, Russian Math. Surveys 31:3 (1976), 1-68. 
\bibitem[BZ77]{BZ77} I. N. Bernstein and A. V. Zelevinsky, {\it Induced representations of reductive p-adic groups}, I, Ann. Sci. Ecole Norm. Sup. {\bf 10} (1977), 441-472.
\bibitem[Bo76]{Bo76} Borel, A. Admissible representations of a semi-simple group over a local field with vectors fixed under an iwahori subgroup. Invent Math 35, 233-259 (1976). https://doi.org/10.1007/BF01390139
\bibitem[BBBF18]{BBBF18} Brubaker, B., Buciumas, V., Bump, D. et al. Hecke modules from metaplectic ice. Sel. Math. New Ser. 24, 2523-2570 (2018). https://doi.org/10.1007/s00029-017-0372-0
\bibitem[Br98]{Br98} Brundan, J. (1998). Modular branching rules and the Mullineux map for Hecke algebras of type A. Proceedings of the London Mathematical Society, 77(3), 551-581. doi:10.1112/S0024611598000562
\bibitem[BK93]{BK93} C.J. Bushnell and P.C. Kutzko, The admissible dual of GL(N) via compact open subgroups, Annals Math. Studies, Princeton Univ. Press (1993)
\bibitem[Ca95]{Ca95} W. Casselman, Introduction to the theory of admissible representations of $p$-adic reductive groups, 1995
\bibitem[CaSh98]{CaSh98} W. Casselman and F. Shahidi, On irreducibility of standard modules for generic representations, Ann. Sci. \'Ecole Norm. Sup. (4) 31 (1998), no. 4, 561-589 (English, with English and French summaries). DOI https://doi.org/10.1016/S0012-9593(98)80107-9 
\bibitem[Ch21]{Ch21} K.Y. Chan, Homological branching law for $(\mathrm{GL}_{n+1}(F),\mathrm{GL}_n(F))$: projectivity and indecomposability, Invent. math. (2021). https://doi.org/10.1007/s00222-021-01033-5
\bibitem[Ch22]{Ch20} K.Y. Chan, Restriction for general linear groups: The local non-tempered Gan-Gross-Prasad conjecture (non-Archimedean case), Crelles Journal, vol. 2022, no. 783, 2022, pp. 49-94. https://doi.org/10.1515/crelle-2021-0066
\bibitem[Ch23]{Ch21+} K.Y. Chan, Ext-multiplicity theorem for standard representations of $(\mathrm{GL}_{n+1}, \mathrm{GL}_n)$,  Math. Z. 303, 45 (2023). https://doi.org/10.1007/s00209-022-03198-y
\bibitem[Ch24]{Ch22+} K.Y. Chan, On the product functor on inner forms of general linear group over a non-Archimedean local field, Transformation Groups (2024). https://doi.org/10.1007/s00031-024-09861-4
\bibitem[Ch22+b]{Ch22+b} K.Y. Chan, Quotient branching law for $p$-adic $(\mathrm{GL}_{n+1}, \mathrm{GL}_n)$ I: generalized Gan-Gross-Prasad relevant pairs, arXiv:2212.05919 (2022)
\bibitem[Ch22+c]{Ch22+c} K.Y. Chan, Duality for generalized Gan-Gross-Prasad relevant pairs for $p$-adic GL, arXiv:2210.17249 (2022)
\bibitem[Ch22+d]{Ch22+d} K.Y. Chan, Construction of Bernstein-Zelevinsky derivatives and highest derivative multisegments II: minimal sequences, preprint.
\bibitem[Ch22+e]{Ch22+e} K.Y. Chan, Construction of Bernstein-Zelevinsky derivatives and highest derivative multisegments III: properties of minimal sequences, preprint.
\bibitem[CS18]{CS18} Kei Yuen Chan and Gordan Savin, Iwahori component of the Gelfand-Graev representation, Math. Z. 288 (2018), no. 1-2, 125-133. MR 3774407, DOI 10.1007/s00209-017-1882-3
\bibitem[CS19]{CS19} K.Y. Chan and G. Savin,  Bernstein-Zelevinsky derivatives: a Hecke algebra approach, International Mathematics Research Notices, Volume 2019, Issue 3, February 2019, Pages 731-760, https://doi.org/10.1093/imrn/rnx138
\bibitem[CS21]{CS21} K.Y. Chan and G. Savin, A vanishing Ext-branching theorem for $(\mathrm{GL}_{n+1}(F), \mathrm{GL}_n(F))$, Duke Math Journal, 2021, 170 (10), 2237-2261. https://doi.org/10.1215/00127094-2021-0028
\bibitem[CPS17]{CPS17} J.W. Cogdell, I.I. Piatetski-Shapiro (2017) Derivatives and L-Functions for GLn. In: Cogdell J., Kim JL., Zhu CB. (eds) Representation Theory, Number Theory, and Invariant Theory. Progress in Mathematics, vol 323. Birkh\"auser, Cham. https://doi.org/10.1007/978-3-319-59728-7-5
\bibitem[De16]{De16} T. Deng, Parabolic induction and geometry of orbital varieties for GL(n), arXiv:1603:06387, PhD thesis, Universit\'e Sorbinne Paris Cit\'e (2016)
\bibitem[GGP20]{GGP20} W.T. Gan, B.H. Gross, and D. Prasad, Branching laws for classical groups: the non-tempered case, Compositio Mathematica, 156(11) (2020), 2298-2367. doi:10.1112/S0010437X20007496
\bibitem[GLS11]{GLS11} Christof Gei\ss, Bernard Leclerc, and Jan Schr\"oer, Kac-Moody groups and cluster algebras, Adv. Math. 228 (2011), no. 1, 329-433.
\bibitem[GV01]{GV01} I. Grojnowski, M. Vazirani, Strong multiplicity one theorems for affine Hecke algebras of type A. Transformation Groups 6, 143-155 (2001). https://doi.org/10.1007/BF01597133
 \bibitem[Gu22]{Gu22} M. Gurevich, On restriction of unitarizable representations of general linear groups and the non-generic local Gan-Gross-Prasad conjecture. J. Eur. Math. Soc. 24 (2022), no. 1, pp. 265-302, DOI 10.4171/JEMS/1093
\bibitem[Gu21]{Gu21} M. Gurevich, Graded Specht modules as Bernstein-Zelevinsky derivatives of the RSK model, to appear in IMRN, arXiv:2110.11381
\bibitem[He11]{He11} V. Heiermann, Op\'erateurs d'entrelacement et alg\`ebres de Hecke avec param\`etres d'un groupe r\'eductif p-adique: le cas des groupes classiques. Sel. Math. New Ser. 17, 713-756 (2011). https://doi.org/10.1007/s00029-011-0056-0
\bibitem[Ja07]{Ja07} C. Jantzen, Jacquet modules of $p$-adic general linear groups, Represent. Theory {\bf 11} (2007), 45-83 https://doi.org/10/1090/S1088-4165-0700316-0
\bibitem[JS83]{JS83} H. Jacquet and J. Shalika, The Whittaker models of induced representations, Pacific Journal of Mathematics 109 (1983), 1, 107-120.
\bibitem[JK22]{JK22} Y. Jo, M. Krishnamurthy, The Langlands-Shahidi method for pairs via types and cover, to appear in Representation Theory
\bibitem[KKKO18]{KKKO18}  S.-J. Kang, M. Kashiwara, M. Kim and S.-J. Oh, J. Amer. Math. Soc. 31 (2018), 349-426, doi.org/10.1090/jams/895
\bibitem[EK08]{EK08}  N. Enomoto, M. Kashiwara, Symmetric Crystals and LLT-Ariki type conjectures for the affine Hecke algebras of type B. Combinatorial representation theory and related topics RIMS K\^oky\^uroku Besstsu, B8, Res. Inst. Math. Sci. (RIMS), Kyoto, pp. 1-20 (2008) 
\bibitem[Kl95]{Kl95} A.S. Kleshchev, Branching rules for modular representations of symmetric groups. II. J. reine angew. Math. 459 (1995), 163 212. 
\bibitem[Kl10]{Kl10}  Alexander Kleshchev, Representation Theory of symmetric groups and related Hecke algebras, Bull. Amer. Math. Soc. 47 (2010), 419-481,https://doi.org/10.1090/S0273-0979-09-01277-4 
\bibitem[LM14]{LM14} E. Lapid and A. M\'inguez. On a determinantal formula of Tadi\'c. Amer. J. Math. 136 (2014): 111-142.
\bibitem[LM16]{LM16} E. Lapid, A. M\'inguez, On parabolic induction on inner forms of the general linear group over a non-Archimedean local field,  Sel. Math. New Ser. (2016) 22, 2347-2400. 
 \bibitem[LM18]{LM18} E. Lapid, A. M\'inguez, Geometric conditions for $\square$-irreducibility of certain representations of the general linear group over a non-archimedean local field, Advances in Mathematics, Volume 339, 2018, Pages 113-190, https://doi.org/10.1016/j.aim.2018.09.027.
\bibitem[LM22]{LM22} E. Lapid, M\'inguez, A binary operation on irreducible components of Lusztig's nilpotent varieties II: applications and conjectures for representations of GLn over a non-archimedean local field, to appear in Pure and Applied Mathematics Quarterly.
\bibitem[LLT96]{LLT96} A. Lascoux, B. Leclerc, and J.-Y. Thibon, Hecke algebras at roots of unity and crystal bases of quantum affine algebras, Comm. Math. Phys. 181 (1996), 205-263.
\bibitem[Le03]{Le03}  B. Leclerc, Imaginary vectors in the dual canonical basis of Uq(n), Transform. Groups 8 (2003), no. 1, 95-104.
\bibitem[Lu89]{Lu89} G. Lusztig, Affine Hecke algebras and their graded versions. J. Amer. Math. Soc. 2 (1989): 599-635.
\bibitem[Lu91]{Lu91} G. Lusztig, Quivers, perverse sheaves, and quantized enveloping algebras J. Amer. Math. Soc. 4 (1991), 365-421 
\bibitem[Mie08]{Mie08} V. Miemietz, On Representations of Affine Hecke Algebras of Type B . Algebr Represent Theory 11, 369 (2008), doi.org/10.1007/s10468-008-9086-5
\bibitem[MW86]{MW86} C. M\oe glin, J.-L. Waldspurger, Sur l'involution de Zelevinski. J. Reine Angew. Math. 372, 136-177 (1986), doi.org/10.1515/crll.1986.372.136
\bibitem[MW12]{MW12} C. M{\oe}glin and J.-L. Waldspurger, La conjecture locale de Gross-Prasad pour les groupes speciaux orthogonaux: le cas general, Sur les conjectures de Gross et Prasad. II, Astersque. No. {\bf 247} (2012), 167-216.
\bibitem[Mi08]{Mi08} A. M\'inguez, Correspondance de Howe explicite : paires duales de type II, Annales scientifiques de l'\'Ecole Normale Sup\'erieure, S\'erie 4, Tome 41 (2008) no. 5, pp. 717-741, doi.org/10.24033/asens.2080
\bibitem[Mi09]{Mi09} A. M\'inguez, Surl'irr\'educitibilit\'e d'une induite parabolique, 2009, no. 629, Crelle's Journal, doi.org/10.1515/CRELLE.2009.028
\bibitem[Ma13]{Ma13} Nadir Matringe, Essential Whittaker functions for GL(n), Documenta Math. 18 (2013) 1191-1214
\bibitem[Of18]{Of18} O. Offen, On symplectic periods and restriction to SL(2n), Math. Z. 294, 1521-1552 (2020), doi.org/10.1007/s00209-019-02390-x
\bibitem[OS10]{OS10} E. Opdam. M. Solleveld. Discrete series characters for affine Hecke algebras and their formal degrees. Acta Math. 205 (1) 105 - 187, 2010, doi.org/10.1007/s11511-010-0052-9 
\bibitem[Pr18]{Pr18} D. Prasad, An Ext-analogue of branching laws, ICM proceedings 2018.
\bibitem[Ra07]{Ra07} Raghuram, A. (2007). A K\"unneth Theorem for p-Adic Groups. Canadian Mathematical Bulletin, 50(3), 440-446, doi:10.4153/CMB-2007-043-5
\bibitem[Re02]{Re02} M. Reeder, Isogenies of Hecke algebras and a Langlands correspondence for ramified principal series representations,  Represent. Theory 6 (2002), 101-126, DOI: https://doi.org/10.1090/S1088-4165-02-00167-X 
\bibitem[Ro85]{Ro85} J.D. Rogawski, On modules over the Hecke algebra of a p-adic group. Invent Math 79, 443-465 (1985). https://doi.org/10.1007/BF01388516
\bibitem[SVV11]{SVV11} P. Shan, M. Varagnolo, E. Vasserot, Canonical bases and affine Hecke algebras of type D, Advances in Mathematics, Volume 227, Issue 1, 2011, Pages 267-291, doi.org/10.1016/j.aim.2011.01.024.
\bibitem[SV17]{SV17} V. S\'echerre, C. G. Venketasubramanian. Modular representations of GL(n) distinguished by GL(n-1) over a p-adic field. IMRN, 2017 (14), pp.4435-4492. 10.1093/imrn/rnw150.
\bibitem[SS12]{SS12} V. S\'echerre, S. Stevens, Smooth Representations of GLm(D) VI: Semisimple Types, International Mathematics Research Notices, Volume 2012, Issue 13, 2012, Pages 2994-3039, https://doi.org/10.1093/imrn/rnr122
\bibitem[Ta87]{Ta87} M. Tadi\'c: Unitary representations of GL(n), derivatives in the non-Archimedean case, V. Mathematikertreffen Zagreb-Graz (Mariatrost/Graz, 1986), Ber. Math.-Statist. Sekt. Forschungsgesellsch. Joanneum, vol. 274, Forschungszentrum Graz, Graz, pp. Ber. No.
281, 19 (1987)
\bibitem[Ta86]{Ta86} M. Tadi\'c, On the classification of irreducible unitary representations of GL(n) and the conjectures of Bernstein and Zelevinsky, Ann. Sci. \'Ecole Norm. Sup., {\bf 19} (1986), 335-382.
\bibitem[VV11]{VV11} M. Varagnolo, E. Vasserot, Canonical bases and affine Hecke algebras of type B. Invent. math. 185, 593-693 (2011). https://doi.org/10.1007/s00222-011-0314-y
\bibitem[Ve13]{Ve13} Venketasubramanian, C.G. On representations of GL(n) distinguished by GL(n-1) over a p-adic field. Isr. J. Math. 194, 1-44 (2013). doi.org/10.1007/s11856-012-0152-7
\bibitem[Wa86]{Wa86} J.-L. Waldspurger, Alg\`ebres de Hecke et induites de repr\'esentations cuspidales, pour GL (N). Journal f\"ur die reine und angewandte Mathematik 370 (1986): 127-191. 
\bibitem[Xu17]{Xu17} B. Xu, On M\oe glin's parametrization of Arthur packets for p-adic quasisplit Sp(N) and SO(N). Canad. J. Math. 69 (2017), no. 4, 890-960.
\bibitem[Ze80]{Ze80} A. Zelevinsky, {\it Induced representations of reductive p-adic groups II}, Ann. Sci. Ecole Norm. Sup. {\bf 13} (1980), 154-210.
\end{thebibliography}
\end{document}